\documentclass[reqno]{amsart}

\usepackage[all]{xy}
\usepackage{graphicx}

\usepackage{amssymb}
\usepackage{amsmath}
\usepackage{mathrsfs}
\usepackage{epsfig}
\usepackage{amscd}
\usepackage{graphicx, color}
\usepackage{comment}
\usepackage{tikz-cd,graphics}
\usepackage{tikz}
\usetikzlibrary{matrix,arrows,decorations.pathmorphing}
	\usepackage{amsmath, amsthm, amssymb, latexsym, mathtools, mathrsfs}
	\usepackage{fancyhdr}
    \usepackage{bbm}

\theoremstyle{plain} 
\newtheorem{thm}{Theorem}[section]
\newtheorem{cor}[thm]{Corollary}
\newtheorem{lem}[thm]{Lemma}
\newtheorem{prop}[thm]{Proposition}

\theoremstyle{definition}
\newtheorem{defn}[thm]{Definition}
\newtheorem{rem}[thm]{Remark}

\newtheorem{cordefn}[thm]{Corollary-Definition}
\newtheorem{exm}[thm]{Example}

\newtheorem{defn-lem}[thm]{Definition-Lemma}
\newtheorem{lem-defn}[thm]{Lemma-Definition}

\newtheorem{notation}[thm]{\rm\bfseries{Notation}}

\newtheorem*{thm*}{Theorem}

\def\E{\ifmmode{\mathbb E}\else{$\mathbb E$}\fi} 
\def\N{\ifmmode{\mathbb N}\else{$\mathbb N$}\fi} 
\def\R{\ifmmode{\mathbb R}\else{$\mathbb R$}\fi} 
\def\Q{\ifmmode{\mathbb Q}\else{$\mathbb Q$}\fi} 
\def\C{\ifmmode{\mathbb C}\else{$\mathbb C$}\fi} 
\def\H{\ifmmode{\mathbb H}\else{$\mathbb H$}\fi} 
\def\Z{\ifmmode{\mathbb Z}\else{$\mathbb Z$}\fi} 
\def\P{\ifmmode{\mathbb P}\else{$\mathbb P$}\fi} 
\def\T{\ifmmode{\mathbb T}\else{$\mathbb T$}\fi} 
\def\SS{\ifmmode{\mathbb S}\else{$\mathbb S$}\fi} 
\def\DD{\ifmmode{\mathbb D}\else{$\mathbb D$}\fi} 

\renewcommand{\a}{\alpha}
\renewcommand{\b}{\beta}
\renewcommand{\d}{\delta}
\newcommand{\D}{\Delta}

\newcommand{\G}{\Gamma}

\newcommand{\s}{\sigma}

\renewcommand{\t}{\tau}

\newcommand{\del}{\partial}
\newcommand{\Cont}{{\operatorname{Cont}}}
\newcommand{\Hom}{{\operatorname{Hom}}}

\newcommand{\ben}{\begin{enumerate}}
\newcommand{\een}{\end{enumerate}}
\newcommand{\be}{\begin{equation}}
\newcommand{\ee}{\end{equation}}
\newcommand{\bea}{\begin{eqnarray}}
\newcommand{\eea}{\end{eqnarray}}
\newcommand{\beastar}{\begin{eqnarray*}}
\newcommand{\eeastar}{\end{eqnarray*}}
\newcommand{\bc}{\begin{center}}
\newcommand{\ec}{\end{center}}

\numberwithin{equation}{section}

\def\R{{\mathbb R}}

\def\E{{\mathbb E}}
\def\Z{{\mathbb Z}}
\def\C{{\mathbb C}}
\def\R{{\mathbb R}}
\def\P{{\mathbb P}}

\def\N{{\mathbb N}}

\def\11{{\mathbb I}}

\def\delbar{{\overline \partial}}

\def\starboxtimes{\stackrel{\star}{\boxtimes}}
\def\starboxplus{\stackrel{\star}{\boxplus}}
\def\Vect{\text{\rm Vect}}
\def\starotimes{\stackrel{\star}{\otimes}}

\def\C{\mathbb{C}}
\def\Z{\mathbb{Z}}

\def\T{\mathbb{T}}

\def\D{\mathbb{D}}
\def\Q{\mathbb{Q}}
\def\X{\mathbb{X}}

\def\R{{\mathbb R}}

\def\E{{\mathbb E}}
\def\Z{{\mathbb Z}}
\def\C{{\mathbb C}}
\def\R{{\mathbb R}}

\def\N{{\mathbb N}}

\def\delbar{{\overline \partial}}

\def\a{\alpha}
\def\b{\beta}

\def\d{\delta}  \def\D{\Delta}

  \def\G{\Gamma}

\def\s{\sigma}  
  
\def\t{\tau}

\def\CA{{\mathcal A}}

\def\CC{{\mathcal C}}
\def\CD{{\mathcal D}}
\def\CE{{\mathcal E}}
\def\CF{{\mathcal F}}

\def\CI{{\mathcal I}}
\def\CJ{{\mathcal J}}

\def\CL{{\mathcal L}}
\def\CM{{\mathcal M}}

\def\CS{{\mathcal S}}

\def\CU{{\mathcal U}}
\def\CV{{\mathcal V}}

\def\darr#1{\raise1.5ex\hbox{$\leftrightarrow$}
\mkern-16.5mu #1}

\def\roughly#1{\raise.3ex\hbox{$#1$\kern-.75em
\lower1ex\hbox{$\sim$}}}

\def\opname#1{\mathop{\kern0pt{\rm #1}}\nolimits}

\def\dim{\opname{dim}}

\def\rank{\opname{rank}}

\def\supp{\operatorname{supp}}

\def\span{\operatorname{span}}

\def\Cont{\operatorname{Cont}}

\def\Index{\operatorname{Index}}

\def\Image{\operatorname{Image}}

\def\Cont{\operatorname{Cont}}

\def\Symp{\operatorname{Symp}}

\def\Index{\operatorname{Index}}

\def\id{\text{\rm id}}
\def\dim{\text{\rm dim}}
\def\codim{\text{\rm codim}}
\def\Exp{\operatorname{Exp}}

\def\Riem{\mathfrak{Riem}}

\DeclareFontFamily{U}{MnSymbolC}{}
\DeclareSymbolFont{MnSyC}{U}{MnSymbolC}{m}{n}
\DeclareFontShape{U}{MnSymbolC}{m}{n}{
    <-6>  MnSymbolC5
   <6-7>  MnSymbolC6
   <7-8>  MnSymbolC7
   <8-9>  MnSymbolC8
   <9-10> MnSymbolC9
  <10-12> MnSymbolC10
  <12->   MnSymbolC12}{}
\DeclareMathSymbol{\intprod}{\mathbin}{MnSyC}{'270}

\begin{document}

\title[Calculus of Legendrian correspondence]
{Non-unital monoidal category of contact manifolds and Legendrian correspondence}

\author{Yong-Geun Oh, Junhyuk Park}
\address{Center for Geometry and Physics, Institute for Basic Science (IBS),
77 Cheongam-ro, Nam-gu, Pohang-si, Gyeongsangbuk-do, Korea 790-784
\& POSTECH, Gyeongsangbuk-do, Korea}
\email{yongoh1@postech.ac.kr}
\address{POSTECH, \&
Center for Geometry and Physics, Institute for Basic Science (IBS),
77 Cheongam-ro, Nam-gu, Pohang-si, Gyeongsangbuk-do, Korea 790-784}
\email{jhp3749@postech.ac.kr}
\thanks{This work is supported by the IBS project \# IBS-R003-D1}

\date{today}

\begin{abstract} There are two purposes of the present paper which are interrelated.
The first goal is to construct the structure of
a  \emph{non-unital monoidal category} $\mathfrak{Cont}$ of contact manifolds, 
not necessarily coorientable, by developing
the contact topology \emph{without contact forms}. The non-unital monoidal product 
is the functorial contact product $\star$, called \emph{star product}, introduced in \cite{oh:shelukhin-conjecture}. 
We prove that the product $\star$ is associative and there exist a collection $\alpha = \{\alpha_{X,Y,Z}\}$
of the \emph{associator} isomorphisms
$\alpha_{X,Y,Z}: X \star (Y\star Z) \cong (X \star Y) \star Z$ for $X, \, Y, \,  Z \in \mathfrak{Cont}$, that
satisfy the \emph{pentagon axiom}, i.e., that the triples $(\mathfrak{Cont}, \star, \alpha)$ form a \emph{nonunital monoidal category}.

The second goal is to develop the calculus of Legendrian correspondences, which are by
definition \emph{embedded} Legendrian submanifolds of the contact product $Q \star Q'$.
Legendrian correspondences will play the role of
1-morphisms in the $2$-categorical structure to be equipped with $\mathfrak{Cont}$ whose two morphisms
are contact instanton cohomologies $HI(R_{ab},R'_{ab})$ associated to a pair of Legendrian 
correspondences $R_{ab}, \, R'_{ab} \in \mathfrak{Leg}(Q_a,Q_b)$. With this future application 
in mind, we define the composition of Legendrian correspondences and prove that the composition of
a generic pair is again embedded and hence canonically becomes a Legendrian correspondence.
\end{abstract}



\maketitle

\tableofcontents


\section{Introduction}

The geometric operation of the \emph{Legendrianization} of contactomorphism $\psi$
was first utilized by Lychagin  \cite{lychagin},  to parameterize the group of contactomorphisms and 
to equip it with suitable $C^r$ topology (See \cite{tsuboi:simplicity} also.)
It is also utilized by Bhupal \cite{bhupal}, Sandon \cite{sandon:translated}
and the first-named author \cite{oh:shelukhin-conjecture}
in their study of contact topology and
dynamics of contactomorphisms via the investigation of \emph{translated points}.

One purpose of the present paper is to amplify this construction to the more
general context of \emph{Legendrian correspondences} and to develop the \emph{calculus of 
Legendrian submanifolds} as the contact counterpart of the 
calculus of Lagrangian submanifolds in symplectic geometry 
\cite{hormander:FIO}, \cite{alan:cbms}: Recall that a direct product of symplectic manifolds 
is a symplectic manifold,  and a Lagrangian submanifold of this product is called a 
\emph{Lagrangian correspondence}.
They play the role of morphisms and so are the building blocks of the \emph{Symplectic category}. 
(See \cite{alan:category}, \cite{wehrheim-woodward}.)

We start with recalling that the canonical symplectization of a contact manifold $(Q,\xi)$
is defined by
$$
SQ = \{\beta \in T^*Q \setminus \{0\} \mid \beta|_\xi = 0 \}
$$
equipped with its symplectic form
$$
\iota^*\omega_0 = - d\iota^*\theta
$$
where $\theta$ is the Liouville one-form on $T^*Q$ and $\iota: SQ \to T^*Q$ is the canonical inclusion map.

\subsection{Contact product}

The first step towards our goal is to provide a functorial definition of the
\emph{product operation} on two contact manifolds which has already been introduced in 
\cite[Appendix]{oh:shelukhin-conjecture}. Explanation of this construction is now in order
starting with the following well-known fact.

\begin{prop} Let $(M, \lambda)$ be any complete Liouville (symplectic) manifold.
Then the set $\del_\infty M$ consisting of (one-sided) Liouville rays carries a
canonical contact structure. When the Liouville flow admits a cross-section,
the restriction of the Liouville one-form $\theta$ to the cross section defines a contact 
one-form whose associated contact structure
is canonically contactomorphic to $\del_\infty M$.
\end{prop}
We denote by $X$ the Liouville vector field and by $\ell_x$ the Liouville ray
issued at a point $x \in M$.

Let $(Q_1, \xi_1)$ and $(Q_2,\xi_2)$ be two contact manifolds, \emph{not necessarily coorientable}.
We consider the product symplectic manifold
$$
(SQ_1 \times SQ_2, \pi_1^*\omega_0^1 + \pi_2^*\omega_0^2), 
$$
where $\omega_0^i := \iota_i^*(-d\theta_i)$ for the Liouville one-form $\theta_i$ of $T^*Q_i$,
$i = 1,\, 2$. Then we have
$$
\pi_1^*\omega_0^1 + \pi_2^*\omega_0^2 = d(\iota_1^*\pi_1\theta_1) + d(\iota_2^*\pi_2\theta_2)
= d\left((\pi_1 \circ \iota_1)^*\theta_1 + (\pi_2 \circ \iota_2)^*\theta_2\right).
$$
With slight abuse of notation, we also denote by $\pi_i: SQ_i \to Q_i$ 
the composition $\pi_i = \pi_i \circ \iota_i$ from now on, whenever there is no danger of confusion.

Then the direct sum of the Liouville vector fields of $SQ_i$ for $i = 1, \, 2$ is a 
Liouville vector field of the product $SQ_1 \times SQ_2$. Using this, a version of contact product
was given in \cite{oh:shelukhin-conjecture}. Here we slightly modify its definition so that the resulting 
product carries more functorial properties.

\begin{defn}[Contact product; compare with \cite{oh:shelukhin-conjecture}] 
For given contact manifolds $(Q_i,\xi_i)$, $i = 1, \, 2$,
we define their \emph{contact product}, denoted by $Q_1 \star Q_2$, to be the ideal boundary
$$
Q_1 \star Q_2 = SQ_1 \times SQ_2/\sim_+
$$
where $\sim_+$ is the equivalence relation of the $\R_+$ action induced by
the \emph{Reeb flow} over $t \in [0,\infty)$, and we equip
it with the canonical contact structure defined by
$$
\xi_{Q_1 \star Q_2} = [\ker (-\pi_1^*\theta_1 - \pi_2^*\theta_2)] \subset T(Q_1 \star Q_2) 
= T(SQ_1 \times SQ_2)/\sim_{\R_+}.
$$
\end{defn}

This product has the following functorial properties.
\begin{prop}[Contact projection $\pi_i^\star$] \label{prop:pi-intro} There exist canonical  projections 
$p_i: Q_0 \star Q_1 \to Q_i$ for $i=0, \, 1$
such that the following diagram commutes:
$$
\xymatrix{ SQ_0 \times SQ_1  \ar[dr]_{\widetilde \pi_i} \ar[rr]^{\pi} && Q_0 \star Q_1 \ar[dl]^{p_i}\\
& Q_i &
}
$$
i.e., that satisfy $p_i \circ \pi = \widetilde \pi_i$. We call $p_i$ the $i$-th \emph{contact projection}
or \emph{star projection} for each $i =0, \, 1$, denote it by $\pi_i^\star$.
\end{prop}
We refer to Section \ref{sec:product} for more detailed explanation.

\begin{prop} The $\pi_1^\star \times \pi_2^\star: Q_1\star Q_2 \to Q_1 \times Q_2$ defines 
a locally trivial fibration that is a $\R$ principal bundle admitting a section. In particular,
$Q_1 \star Q_2$ is coorientable, whether or not $Q_1$ or $Q_2$ is coorientable.
\end{prop}

The non-unital monoidal product 
is the functorial contact product $\star$ introduced in \cite{oh:shelukhin-conjecture}. 
  (See \cite{rivano}, \cite{maclane}, \cite{kerodon} for example, for the precise definition
of (nonunital) monoidal category in general.)

\begin{thm}\label{thm:small-category}
\begin{enumerate}
\item 
For each triple $X =(Q_a,\xi_a), \, Y= (Q_b, \xi_b)$ and $Z = (Q_c,\xi_c)$, there exists 
a natural isomorphism 
$$
\alpha_{X,Y,Z}: X \star (Y \star Z) \cong (X \star Y) \star Z.
$$
\item The triple $(\mathfrak{Cont}, \star, \alpha)$ defines a nonunital monoidal category.
\end{enumerate}
\end{thm}
We prove that the product $\star$ is associative (up to contactomorphisms) and there exists collection $\alpha = \{\alpha_{X,Y,Z}\}$
of isomorphisms
$\alpha_{X,Y,Z}: X \star (Y\star Z) \cong (X \star Y) \star Z$ for $X, \, Y, \,  Z \in \mathfrak{Cont}$,
and that the triple $(\mathfrak{Cont}, \star, \alpha)$ satisfies the axiom of \emph{nonunital monoidal category}
i.e.,  establish the \emph{pentagon identity}.

\subsection{Ideal boundary of the product of Liouville manifolds}

We recall that the set of manifolds with corners forms a monoid under the product.
We now show that this monoidal structure is compatible with the contact product
under the operation of taking the ideal boundary of Liouville manifolds or more generally
ideal boundary of Liouville sectors.

When $M$ is a complete manifold of finite type,
each end is cylindrical, i.e., we may assume
that each connected component of the end is diffeomorphic to
$Q \times [0, \infty)$. In particular, when $(M,\lambda)$  is a Liouville manifold,
the ideal boundary $\del_\infty M = Q$ which also carries a canonical
contact structure.

Now let $(M_1,\lambda_1)$ and $(M_2, \lambda_2)$ be two Liouville manifolds.
The following theorem displays some canonicality of the canonical contact product
\begin{thm}\label{thm:Liouville-contact} Let $(M_i,\lambda_i)$ be two Liouville manifolds. Then there
is a canonical `contact isomorphism'
$$
\del_\infty M_1 \star \del_\infty M_2 \cong \del_\infty(M_1 \times M_2).
$$
\end{thm}
There is one subtlety that enters in the statement of this theorem: There is
some nontrivial issue concerning the choice of smooth structure on the target space
$ \del_\infty(M_1 \times M_2)$. As the definition of the ideal boundary goes, it 
carries a natural topology for which the above correspondence defines a
homeomorphism, but does not automatically define a smooth diffeomorphism. 
One has to specify the smooth structure $\del_\infty(M_1 \times M_2)$
with respect to which Theorem \ref{thm:Liouville-contact} holds at the level of
smooth diffeomorphisms.

We postpone this discussion of  this smoothness issue until a sequel \cite{oh-park:ideal-bdy} 
as well as  the precise construction of the smooth structure 
with respect to which the above homeomorphism becomes a contact diffeomorphism.

\subsection{Contact immersions} 

We now introduce the notion of \emph{contact immersions} which contains morphisms
of the non-unital monoidal category
$$ 
 \mathfrak{Cont}
$$
of contact manifolds $(Q,\xi)$. For this purpose, we regard a contact manifold, not equipped with
a contact form, as  a Jacobi manifold following Loose \cite{loose}, i.e., we have a natural
short exact sequence
\be\label{eq:natural-short}
0 \to \xi \to TM \to TM/\xi \to 0.
\ee
We denote $I^\xi: = TM/\xi$. (See \cite[Section 5]{LOTV} also.)

\begin{defn}[Contact immersions]\label{defn:contact-immersion-intro} 
Let $(Q,\xi)$ be a contact manifold. A smooth immersion $\varphi: N \to Q$ is called a
\emph{contact immersion} if there exists a line bundle $I \to N$ and an $I$-valued one-form $\alpha$ 
such that the following hold:
\begin{enumerate}
\item $\rank \ker \alpha_x = \dim N -1$, and 
\item we have a morphism of short exact sequences:
$$
\xymatrix{0 \ar[r] & \ker \alpha \ar[r]\ar[d]^{d\varphi} & TN \ar[r]^\alpha \ar[d]^{d\varphi}  & I \ar[r]\ar[d]^{\varphi_*} & 0 \\
0 \ar[r] & \varphi^*\xi \ar[r] & \varphi^*TQ \ar[r]^\theta & \varphi^*(I^\xi) \ar[r] & 0
}
$$
on $N$, where $I^\xi$ is the canonical line bundle $TQ/\xi \to Q$.
\item the induced map $(\varphi)_*: I \to \varphi^*(I^\xi)$ is a bundle isomorphism.
\item $\alpha$ is nondegenerate on $\ker \alpha$.
\end{enumerate}
\end{defn}

\begin{rem} The above definition is a variation and specialization of the definition of
Jacobi morphisms when one regard a contact manifold as a Jacobi manifold.
(See \cite[Section 2.3 \& Section 5]{LOTV} and Section 2 of the present paper for 
relevant discussion.)
\end{rem}

Then we prove the following universality of contact immersions.
\begin{prop}[Proposition \ref{prop:universality-product}] \label{prop:universality-product-intro}
Let $(Q_1,\xi_1)$ and $(Q_2,\xi_2)$ be two contact manifolds, and $Q_1 \star Q_2$ be
their contact product. Let $(Q,\xi)$ be any contact manifold admitting 
contact immersions $f_1: Q \to Q_1$ and $f_2: Q \to Q_2$. Then there exists a 
contact immersion $f: Y \to Q_1 \star Q_2$ that fills an
universality diagram \emph{under the contact immersions}.
We denote the resulting contact immersion by $f =: f_1 \star f_2$.
\end{prop}
 
\subsection{Legendrian correspondence}

We now explain the morphism space on the
aforementioned monoidal category $\mathfrak{Cont}$ constructed in Theorem \ref{thm:small-category}
by introducing the calculus of \emph{Legendrian correspondences},
which is the contact counterpart of  the calculus of Lagrangian correspondence in
Weinstein's category in symplectic geometry.

Recall that Lagrangian correspondences are defined to be Lagrangian submanifolds
in the product 
$$
(X,-\omega) \times (X', \omega') := (X \times X', -\omega\oplus \omega')
$$
which play the role of 1-morphisms in the symplectic category $\mathfrak{Symp}$
introduced by Weinstein \cite{alan:category,alan:cbms}. More recently, they play a crucial role 
in the construction of 2-functor from $\mathfrak{symp}$ to the 2-category of 
filtered $A_\infty$ categories constructed by Wehrheim-Woodward \cite{wehrheim-woodward}
and by Fukaya \cite{fukaya:immersed}.  For this purpose, it is important to include the 
Lagrangian immersions to make the construction flexible enough to prove the axioms of 
a 2-functor \cite{fukaya:immersed}. This is because in general the composition of 
two Lagrangian correspondences is \emph{not} embedded even when the given correspondences are
\emph{embedded} from the nature of Lagrangian submanifolds and the definition of
composition. Because the appearance of self-intersections is a stable phenomenon for Lagrangian submanifolds
in symplectic topology, there is no simple way of avoiding the self-intersections.  
This is one of the sources that makes
the relevant Lagrangian Floer theory technically quite hard and the 
construction of the aforementioned monoidal 2-functor technically highly non--trivial as demonstrated
in \cite{fukaya:immersed}.

On the other hand, appearance of self-intersections of Legendrian submanifolds is \emph{unstable}
and can be generically avoided.  This makes the first-named author's theory of bordered
contact instantons appearing in the construction of the Fukaya-type Legendrian category
can be applied for the construction of a similar 2-category $\mathfrak{Cont}$ and 
the relevant 2-functor technically simpler, to the level that there is essentially no non-trivial 
analysis, such as the kind of technicalities \cite{fooo:book1,fooo:book2} 
arising in the analysis of pseudoholomorphic curves in symplectic geometry. 
On the other hand, it involves a new type of technicality arising from
noncompactness of contact product \cite{oh:entanglement1,oh:shelukhin-conjecture}. 
 (See \cite{oh:LCI-category} for its first step.) 

There is one crucial difference between the monoidal products in
$\mathfrak{Symp}$ and $\mathfrak{Cont}$: In the former, the canonical Cartesian product 
$X \times X'$ of the obvious direct sum $-\omega \oplus \omega'$ is used, but
in the latter, the contact product $(X, \xi)\star (X',\xi')$ is used where the definition of
the product $\star$ already involves contact structures $\xi, \, \xi'$ and
some nontrivial constructions of the resulting contact structure, which we denote by
$\xi \star \xi'$. (See \cite[Appendix]{oh:shelukhin-conjecture} and Section \ref{sec:product}
of the present paper.)

 \begin{defn}[Legendrian correspondence]
By a Legendrian correspondence of $(Q,\xi)$ and $(Q',\xi')$, we mean an \emph{embedded} Legendrian 
submanifold $R$  of the contact manifold $(Q \star Q', \xi\star \xi')$.
\end{defn}
We then introduce the notions of \emph{composable} pairs and \emph{strongly composable} pairs.
The first notion is roughly that the two Legendrian correspondences satisfy some natural 
counterpart of the transversality hypothesis entering in the definition of composition of
Lagrangian correspondences \cite{hormander:FIO}, \cite{alan:category}. Then the second notion
is a strengthening of the composability so that the following holds.

\begin{thm}[Theorem \ref{thm:strongly-composable}] For any composable pair $(R_{ab}, R_{bc})$, there
exists a $C^\infty$ small Legendrian perturbation  $(R'_{ab}, R_{bc}')$ of  $(R_{ab}, R_{bc})$ which
becomes strongly composable. 
\end{thm}
We refer readers to Part 2 of the present paper for the more precise description of Legendrian
correspondence in which the contact product $\overline Q \star  Q'$ is also used when
$Q$ is coorientable and is equipped with a coorientation.

It would be instructive if the readers compare the map $\Upsilon_{(\varphi,\psi)}$
in \eqref{eq:Upsilon-psi} with the nonlinear Cauchy-Riemann operator 
$$
\delbar_J: u \mapsto \delbar_Ju
$$
prominently studied in the moduli study of the Gromov-Witten theory. (See Section \ref{sec:nonliftable} for the definitions of 
undefined terms here.)
The following summarizes the comparison between the two:
\begin{center} \Large
\begin{tabular}{|c||c|} 
\hline
\text{Gromov-Witten theory}  & \text{Moduli theory of $\Upsilon$}   \\
\hline
$u: \Sigma \to M$ &   $(x_{12,23},y_{12,23})$ \\ \hline 
$\CJ(M)$ & $\mathfrak P$\\ \hline
$\delbar_J$ & $\Upsilon_{(\varphi,\psi)}$ \\ \hline 
$\delbar_J u$ & $\Upsilon_{(\varphi,\psi)}(x_{12,23},y_{12,23})$ \\ \hline
$\delbar$ & $\Upsilon$ \\ \hline
$\CM(M,J)$ & $({\Pi_{\bar13}^\star})^{-1}\left(\text{\rm Self}(\varphi(R_{12}) \circ \psi(R_{23}))\right)$ \\ \hline
$\CM^{\text{\rm para}}(M)$ & $\mathfrak{Int}^{\mathfrak P}(R_{12},R_{13})$  \\ \hline
\end{tabular}
\end{center}

\medskip

Here we compare the definitions of the parameterized moduli spaces
$$
 \CM^{\text{\rm para}}(M)= \bigcup_{J \in \CJ(M)} \{J\} \times \CM(M,J)
$$
and 
$$
\mathfrak{Int}^{\mathfrak P}(R_{12},R_{13}) = \bigcup_{(\varphi,\psi) \in \mathfrak P} \{(\varphi,\psi)\} \times 
\left\{({\Pi_{\bar13}^\star})^{-1}\left(\text{\rm Self}(\varphi(R_{12}) \circ \psi(R_{23}))\right)\right\}.
$$
However unlike the case of Gromov-Witten theory where
normally no restriction on the parameter $J$ is imposed,
it turns out that we will need to restrict ourselves to some smaller open subset of  the product
$\Cont(Q_{\bar 12,\bar 23})$ given in \eqref{eq:Cont1223}.

Now we put the whole of our study into define the categorical framework by
constructing the monoidal category $\mathfrak{Cont}.$ 
The above generic composability ensures the following is well-defined as a category.
\begin{cordefn}[Contact category $\mathfrak{Cont}$]\label{defn:contactcat}
	$\mathfrak{Cont}$ is the category whose objects are contact manifolds, and morphism sets consist of Legendrian-isotopic classes of Legendrian correspondences. 
\end{cordefn}

The following functorial property under the Legendrian correspondence morphisms is 
proved in Section \ref{sec:bifunctor}.

\begin{prop}[Proposition \ref{prop:bifunctor}]\label{prop:bifunctor-intro}
Consider the category $\mathfrak{Cont}$ with Legendrian correspondence as a morphism.
Then   $\bar\star:(Q_1,Q_2)\mapsto Q_1\star Q_2$ is a bifunctor from $\mathfrak{Cont}\times 
    \mathfrak{Cont}$ to $\mathfrak{Cont}$.
\end{prop}

\begin{rem} We
    can only \emph{generically compose} Legendrian compositions. In a more categorical language, 
    one may expect that there is a `model categorical structure' on $\mathfrak{Cont}$ 
    compatible with the $A_\infty$ structure in some suitable sense so that its homotopy category is 
    an actual category, i.e., the weak equivalence of the model structure is given by the contact perturbation
    making Legendrians composable. To achieve this, we need to carefully study the parameter spaces of 
    Theorem \ref{thm:psi-perturbation} and establish more structural properties about it.
    We will come back to this study of model structure elsewhere.
\end{rem}

\bigskip

\noindent{\bf Conventions and Notations:}

\medskip

\begin{itemize}
\item {(Contact Hamiltonian)} We define the contact Hamiltonian of a contact vector field $X$ to be
$- \lambda(X) =: H$. 
\item
For each given time-dependent function $H = H(t,x)$, we denote by $X_H=X_{(H;\lambda)}$ the associated contact Hamiltonian 
vector field whose associated Hamiltonian $- \lambda(X_t)$ is given by $H = H(t,x)$, and its flow by 
$\psi_H^t = \psi_{(H;\lambda)}^t$.
\item {(Reeb vector field)} We denote by $R_\lambda$ the Reeb vector field associated to $\lambda$
and its flow by $\phi_{R_\lambda}^t$. It is the same as $X_H$ with $H \equiv -1$.
\item $g_{(\psi;\lambda)}$: the conformal exponent defined by 
$\psi^*\lambda = e^{g_{(\psi;\lambda)}} \lambda$
for the contact pair $(\lambda,\psi)$.
\item {(Splitting)} We write $\alpha = V_\alpha^\pi \intprod d\lambda + h_\alpha \, \lambda$ with $h_\alpha: = \lambda(R_\lambda)$
in terms of the splitting $T^*M = (\R \{R_\lambda\})^\perp \oplus \xi^\perp$ dual to the splitting $TM = \xi \oplus \R \{R_\lambda\}$.
\item $\vec E$: the Euler vector field $\vec E = p \frac{\del}{\del p}$ on the cotangent bundle $T^*Q$.
\item {(Variation of contact Hamiltonian vector fields)} We denote 
\beastar
Y_\alpha &: = & \delta_\lambda(R_\lambda)(\alpha) = \frac{d}{ds}\Big|_{s=0} R_{\lambda_s}, \\
Z_\alpha
& : = &\delta_\lambda\left(X_{(H;\lambda)}\right)(\alpha) = \frac{d}{ds}\Big|_{s=0} X_{(H;\lambda_s)}
\eeastar
for the germ of curves $\{\lambda_s\}_{-\epsilon < s < \epsilon}$
satisfying $ \lambda_0 = \lambda, \, \alpha = \dot \lambda_s|_{s = 0}$.
\item $SQ = \{ \alpha \in T^*Q \setminus \{0\} \mid \alpha(\xi) = 0\}$ for contact manifold $(Q,\xi)$.
\item $\pi$ denotes any canonical projection such as $\pi: SQ \to SQ/\sim$, or $\pi:SQ \to Q$ and others.
\item $[f]$: This is the descendent map of a map $f$ after quotienting the \emph{domain} by
an equivalence relation, when it exists.
\item $\pi^\star_i: Q_0 \star Q_1 \star \cdots Q_k \to Q_i$ the contact projection (or the star projection).
\item $\P_+(SQ_0 \times SQ_1) = SQ_0 \times SQ_1/\sim_+$ where $\sim_+$ is the equivalence relation
induced by the $\R_+$ flows $\left\{\phi_{\vec E}^{\log t}\right\}_{t > 0}$, which coincides  just 
with the diagonal multiplication action by $\R_+$.
\end{itemize}

\part{Non-unital monoidal category of contact manifolds}

\section{Contact manifold as a Jacobi manifold}\label{sec:contact}

Since we consider contact manifolds which are not necessarily coorientable, we regard
a contact manifold as  a distinguished class of (transitive) abstract Jacobi manifolds
following the description of \cite{loose}, \cite{LOTV}.

We start with recalling the popular definition of contact manifold.
\begin{defn}
\label{mydef:contact}
A \emph{contact structure} on a smooth manifold $Q$ is a maximally non-integrable hyperplane distribution $\xi$ on $Q$.
A \emph{contact manifold} is a smooth manifold $Q$ equipped with a contact structure $\xi$.
\end{defn}
Denote by $I^\xi: = TQ/\xi$ the quotient line bundle. By construction this has structure group $\Z_2$ 
and so a flat line bundle, i.e., defines a local system.
\begin{defn}[Coorientation line bundle]
We call the line bundle $I^\xi \to Q$ the \emph{coorientation local system} of
the contact structure $\xi$. We say $(Q,\xi)$ is coorientable if 
$I^\xi$ is trivial, and non-coorientable otherwise.
\end{defn}
When $\xi$ is coorientable, nondegeneracy of $\xi$ can be (globally) expressed as the nondegeneracy of 
$d\alpha$ for any contact form $\alpha$, i.e., one-form $\ker \alpha = \xi$. However when $\xi$ is
not coorientable, such a description can be given only locally. Furthermore 
the contact form has been used to study contact (Hamiltonian) dynamics, which forces one to
restrict to the coorientable case.

We start with a general discussion on the hyperplane bundle $H \subset TQ$ on a general manifold $Q$.
\begin{defn}[Structure form] Let $H \subset TQ$ be a distribution of codimension 1 on $Q$.
Denote $I^H : = TQ/H $ and consider the canonical projection 
$$
\theta_H: TQ \to I.
$$
 We call $\theta$ the \emph{structure form of $H$} regarding it as an $I$-valued differential one-form. 
The \emph{curvature form} of $H \subset TQ$ is the vector bundle morphism 
\be\label{eq:omega}
\omega_H:\Lambda^2(H)  \to I^H
\ee
defined by
\be\label{eq:curvature-form}
\omega_H(X,Y)=\theta_H([X,Y]), \, X,Y\in \Gamma(H)
\ee
which is well-defined. 
\end{defn}
\begin{rem}
Note that the definition of curvature form works verbatim for distribution of arbitrary codimension
 (See also \cite[Section 4]{oh-park:coisotropic} for a detailed exposition on the curvature form).
\end{rem}

Let $I \to Q$ be a line bundle and let
\be\label{eq:contact-sequence}
0 \longrightarrow H \longrightarrow TQ \stackrel{\alpha}{\longrightarrow} I \longrightarrow 0
\ee
be a short exact  sequence of vector bundles on $Q$. The $\alpha$ is an $I$-valued one-form
on $Q$ which induces an isomorphism
$$
\widetilde \alpha: TQ/H = I^H  \to I.
$$

\subsection{Definition of contact manifold}

The main goal of this subsection is to describe the analogue of the `contact form'
for the non-coorientable case  following  that of \cite{loose}, \cite{LOTV}.
For the purpose of extending contact dynamics to the non-coorientable
case, it turns out to be useful to regard a contact manifold as a special case of
\emph{Jacobi manifolds} as studied in \cite{LOTV}.

We start with borrowing the definition of \emph{contact vector space} $(V,H,\omega)$ from
\cite{loose}:
\begin{defn}[Contact vector space] A \emph{contact vector space} is a triple $(V,H,\omega)$ 
with a vector space $V$ with a hyperplane $H \subset V$ and a non-degenerate 
skew-symmetric bilinear form $\omega: H \times H \to V/H$.
\end{defn}

A model space of contact vector space is $V = \R^{2n+1}$ equipped with
$$
H= \{(q,p,z) \mid z = 0\}, \quad  V/H \cong \R_z
$$
and the bilinear form $\omega: H \times H \to \R_z$ is given by
$$
\omega = \sum_{i=1}^n dq_i \wedge dp_i.
$$
Now we consider the globalization of the above to the case of vector bundles $V \to Q$.

\begin{defn}[Contact vector bundle] A \emph{contact vector bundle} is a triple 
$(V,H,\omega)$ of a vector bundle $V \to Q$, a hyperplane subbundle $H \subset V$
and a skew-symmetric bilinear form 
$$
\omega: H \times H \to V/H
$$
such that $(V_x,H_x,\omega_x)$ is a contact vector space. We call the pair $(H,\omega)$ 
a \emph{contact structure} of the vector bundle $V$.
\end{defn}
In the above point of view of Jacobi manifolds, a contact structure on $Q$ can be given as follows.
(See \cite[p.134-135]{loose}.)

\begin{defn}\label{defn:alpha-Omega} Let $H \subset TQ$ be a hyperplane distribution and 
$\omega: H \times H \to I^H = TQ/H$
be its curvature form.  We say that the pair $(\alpha,\Omega)$
$$
\alpha: TQ \to I, \quad \Omega: H \times H \to I
$$
\emph{represents a contact structure $(H, \omega)$ of $Q$} if $\ker \alpha = H$
and $\omega_H = \widetilde \alpha^{-1} \circ \Omega$. In this case, we denote $H= \xi$.
\end{defn}

With this preparation, we are now ready to introduce the \emph{generalized contact form}
by specializing the above discussion to the contact distribution $H = \xi$.
\begin{defn}\label{defn:generalized-contact}
Let $(Q,\xi)$ be a contact manifold. We call a section
$\alpha \in \Gamma(T^*Q\otimes I^\xi)$ a \emph{generalized} contact form of $\xi$ if $\ker\a=\xi$ and
$$
d\alpha: \xi \times \xi \to I^\xi
$$
is nondegenerate. We denote by $\mathfrak{C}(Q;I^\xi)$ the set of 
generalized contact forms.
\end{defn}

\begin{exm}[One-jet bundle of line bundle]\label{exm:1jet-bundle}
Let $I \to Q$ be a line bundle. Represent the germ $\nu$ of jets at $x$ 
by a local section $\lambda$ around $x \in Q$.
We denote by $J^1_xQ$ the set of the associated equivalence classes.
It follows that this definition does not depend on the choice of local section $\lambda$ but
depends only on its germ $\nu$ at $x$, and so is well-defined.
There is a canonical contact structure on $J^1 I$, 
sometimes called the \emph{Cartan distribution} and defined as follows. Let $\pi : J^1 I \to Q$,
 and $\mathrm{pr} : J^1 I \to I$ be canonical projections. Consider the pull-back line bundle 
 $\pi^\ast I \to J^1 I$. There is a canonical $\pi^\ast I$-valued one-form $\theta$ on $J^1 I$ 
which is the bundle map $\theta_I: T(J^1I) \to I$ given by
\[
\theta_I (\zeta_\nu ) := (d \mathrm{pr} - d\nu \circ d\pi) (\zeta_\nu), \quad \zeta_\nu \in T_\nu(J^1 I).
\]

The Cartan distribution is then defined as the kernel $\xi_I: = \ker \theta_I$. In particular, $\theta$ induces 
the canonical isomorphism
\be\label{eq:tildetheta}
\widetilde \theta_I : T(J^1 I)/\xi_I \to \pi^\ast I.
\ee
\end{exm}
We refer readers to Section \ref{subsec:1jet-bundle} for further expounding of $J^1Q$.

\begin{rem}
Let $(Q,\xi)$ be a contact manifold.
There exists a natural one-to-one correspondence between
\begin{enumerate}
\item local trivializations (or nowhere zero local sections) of the line bundle $I^\xi\to Q$ and
\item local contact forms of $(Q,\xi)$, i.e.~$1$-forms $\alpha\in\Omega^1(U)$, with $U$ open in $Q$, such that $\xi|_U=\ker\alpha$.
\end{enumerate}
\end{rem}

\subsection{Contactomorphisms and contact vector fields}

For a given contact structure $\xi$ on $Q$ and its curvature form $\omega$, we  
consider the vector bundle morphism 
$$
\omega^\flat : \xi \to \xi^\ast \otimes \CL, \quad X\mapsto\omega^\flat(X):=\omega(X,-). 
$$
Let $(Q,\xi)$ and $(Q^\prime, \xi^\prime)$ be contact manifolds. A contactomorphism 
from $(Q,\xi)$ to $(Q',\xi')$ is a diffeomorphism $\psi:Q\to Q'$ such that
\begin{equation*}
d\psi(\xi)=\xi'.
\end{equation*}

An \emph{infinitesimal contactomorphism} (or \emph{contact vector field}) of a contact manifold 
$(Q,\xi)$ is a vector field $X\in\mathfrak X(Q)$ whose flow consists of contactomorphisms. 
Equivalently, $X \in \mathfrak X(Q)$ is a contact vector field if
$
[X, \Gamma (\xi)] \subset \Gamma (\xi)
$.
Contact vector fields of $(Q,\xi)$ form a Lie subalgebra of $\mathfrak X(Q)$ which will be denoted by 
$\mathfrak{cont}(Q,\xi)$.

\begin{prop}[{\cite[Proposition 2.3]{oh-wang:CR-map2}},{ \cite[Proposition 5.4]{LOTV}}]
\label{prop:contact_splitting} Let $(Q,\xi)$ be a contact manifold. Then $\Gamma(\xi)$ is a direct summand of
$\mathfrak X(Q)$, i.e., there is a natural direct sum decomposition of $\mathbb R$-vector spaces
$$
\mathfrak X(Q)= \mathfrak{cont}(Q,\xi) \oplus\Gamma(\xi).
$$
\end{prop}

Now we give the main dynamical object for a general contact manifold, \emph{not necessarily
coorientable}, as follows.

\begin{lem-defn}[Contact vector fields]\label{lem-defn:contact}
Let $(Q,\xi)$ be a contact manifold, and let $\theta_\xi \in \Gamma(T^*Q \otimes I^\xi)$ 
be the structure form of $\xi \to Q$.

For any given section $\beta \in \Gamma (I^\xi)$, there is a unique vector field satisfying 
\be\label{eq:defining-Rlambda}
\begin{cases} \theta_\xi (X_\beta) = \beta\\
X_\beta \intprod d\theta_\xi = 0
\end{cases}
\ee
We denote by $X_\beta$  the  vector field and call the \emph{contact vector field associated to
$\beta$}.  
\end{lem-defn}
\begin{proof} 
Let $x_0 \in Q$ and $X$ be a local section of $TQ$ 
defined on a neighborhood $U$ of $x$ that satisfies $\beta(x) = [X(x)]$ for all $x \in U$.
Here we denote by $[X(x)] \in I^\xi = TQ/\xi$ the equivalence class represented by $X(x)$.

Suppose two such sections $X, \, X'$ satisfy \eqref{eq:defining-Rlambda}.
Then we have
$$
[X(x)-X'(x)] = 0
$$
for all $x \in V$ for some neighborhood $V \subset U$ of $x$.
This implies $X(x)- X'(x) \in \xi_x$ for all $x \in V$ and so defines a local
section of $\xi$ on $V$. Moreover, we also have
$$
(X - X')(x) \intprod d\theta = 0
$$
by the defining properties.

The form $d\theta_\xi$ is nondegenerate
on the symplectic vector space $\xi$  (since its curvature form $\omega$ is nondegenerate), we have
$$
X(x)-X'(x)= 0 \quad \text{on}\,\xi.
$$
Combining this with $[X - X'](x) = 0$, we have shown
$X(x) = X'(x)$ for all $x \in V$. Therefore by the sheaf property of $\mathfrak X(Q)$, 
we can glue the local sections to produce a smooth vector field satisfying \eqref{eq:defining-Rlambda}.
This finishes the proof.
\end{proof}

\begin{rem}
We would like to highlight the fact that unlike the coorientable case, we do not have a
distinguished contact vector field such as the \emph{Reeb vector field} in general. It is just consumed
into the set of contact vector fields.
\end{rem}

The following notions of \emph{contact immersions} and of \emph{Legendrian immersions} will be important for
further discussions.

\begin{defn-lem}[Contact immersions]\label{defn:contact-immersion} 
Let $(Q,\xi)$ be a contact manifold. A smooth immersion $\varphi: N \to Q$ is called a
\emph{contact immersion} if there exists a line bundle $I \to N$ and a $I$-valued one-form $\alpha$ 
such that the following hold:
\begin{enumerate}
\item $\rank \ker \alpha_x = \dim N -1$, and 
\item we have a morphism of short exact sequences:
$$
\xymatrix{0 \ar[r] & \ker \alpha \ar[r]\ar[d]^{d\varphi} & TN \ar[r]^\alpha \ar[d]^{d\varphi}  & I \ar[r]\ar[d]^{\varphi_*} & 0 \\
0 \ar[r] & \varphi^*\xi \ar[r] & \varphi^*TQ \ar[r]^\theta & \varphi^*(I^\xi) \ar[r] & 0
}
$$
on $N$.
\item the induced map $(\varphi)_*: I \to \varphi^*(I^\xi)$ is a bundle isomorphism.
\item $d\alpha$ is nondegenerate on $\ker \alpha$.
\end{enumerate}
We call the image of $\varphi$ a \emph{contact submanifold} of $(Q,\xi)$.
\end{defn-lem}

\begin{defn}[Legendrian immersions]\label{defn:legendrian-immersion}
Let $R$ be a smooth manifold.
A smooth immersion $\varphi: R \to Q$ is called \emph{isotropic} if $TR \subset \varphi^*\xi$,
and \emph{Legendrian} if it is maximally isotropic, i.e., if it is of dimension $\dim R = \frac12 (\dim Q -1)$.
\end{defn}

We recall that appearance of self-intersections of the image of Legendrian map 
is not a stable phenomenon but that of codimension 1 unlike that of
Lagrangian immersion. It makes a generic Legendrian map
automatically embedded.

\section{Local parametrization of the space of Legendrian submanifolds} \label{sec:loose}

In this section, we recall a version of Darboux-Weinstein tubular neighborhood theorem of Loose
\cite{loose} for a Legendrian submanifold of any contact manifold $(Q,\xi)$, not necessarily coorientable.
Using this, we will parameterize nearby Legendrian submanifolds of a given one $R \subset Q$
as the space of sections of the one-jet bundle  $J^1I^\xi$ of the coorientation bundle $I^\xi$.

\subsection{Properties of the one-jet bundle $J^1I$}
\label{subsec:1jet-bundle}

For readers' convenience,
we first recall the definition of one-jet bundle, denoted by $J^1I$ of general line bundle $I \to N$ and its
canonical contact structure closely following the exposition by Loose in \cite{loose},
since we will need the detailed description of the contact structure on $J^1I$
 of any line bundle $I \to N$ mentioned in Example \ref{exm:1jet-bundle}.

Let $\pi: I \to N$ be a real line bundle and denote by $\Gamma(I)$ the space of global sections
$s: N \to I$ of $\pi$. We denote by $\CE_x$ the germs of smooth functions of $N$ at $x$, and define
$\CI_x$ to be the germs of global sections of $I$. The latter is an $\CE_x$-module.

\begin{defn}[One-jet bundle $J^1I$] Define ${\bf a}_x \subset \CE_x$ to be the maximal 
ideal of vanishing germs,
$$
{\bf a}_x = \{ f_x \in \CE_x \mid f(x) = 0\}
$$
and ${\bf b}_x \subset \CI_x$ the submodule of vanishing germs of sections of $I$,
$$
{\bf b}_x = \{ s_x \in \CI_x \mid s(x) = 0\}.
$$
The \emph{one-jet bundle} $J^1I \to N$ is the vector bundle over $N$,
$$
J^1I = \bigcup_{x \in N} \{x\} \times J^1I_x, \quad J^1I_x: = \CI_x/{\bf a}_x{\bf b}_x.
$$
We call the fiber $J^1I_x$ the space of \emph{1-jets of germs of sections at $x$.}
\end{defn}

Let $s \in \Gamma(I)$ be a section. We denote by $s_x \in \CI_x$ the germ of $s$ at $x$.

\begin{defn}[One-jet map]
For each given section $s \in \Gamma(I)$, its \emph{one-jet map} denoted by $j^1s$ is a 
natural section of $J^1I \to N$ defined by
\be\label{eq:1jet-germ}
(j^1s)(x): = s_x \mod {\bf a}_x {\bf b}_x
\ee
\end{defn}

We have a natural bundle map $\alpha: J^1I \to I$ characterized by 
\be\label{eq:defn-alpha}
\alpha_x(j^1s) = s(x)
\ee
for any section $s \in \Gamma(I)$.

\begin{lem} The bundle map $\alpha: J^1I \to I$ is well-defined by 
the defining equation \eqref{eq:defn-alpha}
\end{lem}
\begin{proof} Let $s, \, s'$ be two sections with
$$
(j^1s)(x) = (j^1 s')(x).
$$
We should prove $s(x) - s'(x)= 0$. By definition, we have
$$
s(x) -s'(x) = 0 \mod {\bf a}_x{\bf b}_x,
$$
i.e., we have $(s-s')(x) = t(x)$ for an element $t \in {\bf a}{\bf b}$.
Therefore we obtain
$$
s(x)-s'(x) = 0.
$$
This proves $\alpha$ is well-defined. The linearity of $\alpha$ follows since the map $s \mapsto j^1s$ is 
a fiberwise linear map.
\end{proof}

Next we represent element of $\alpha_x \in J^1I_x$ more explicitly.
Let $v \in T_x N$ and let $c:(-\epsilon,\epsilon) \to N$ be a germ of paths at $x$
representing $v$. Then we have
$$
ds_x(v) = \frac{d}{dt}\Big|_{t=0} s \circ c \in T_{s(x)}I.
$$
When $s(x) = o_x\in o_I$, we have the natural vertical projection 
$$
\nabla_v s: = \Pi^v \circ (ds_x(v)) \in I_x
$$
with respect to the splitting
$$
T_{o_x}I = T_x N \oplus I_x.
$$ 
(Another way of defining this is by choosing a connection $\nabla$ of the vector bundle $I \to N$ and
consider $\nabla_v s$ whose value does not depend on the choice of connection if $s(x) = 0$.)
We then have 
$$
\nabla s|_x \in \Hom(T_x N,I_x) = T_x^*N \otimes I_x
$$
\begin{prop} Let $\nabla$ be an affine connection of $I \to N$ and $\a$ be the bundle map as \eqref{eq:defn-alpha}. 
We consider the assignment $\d_x: \ker \alpha_x \to T_x^*N \otimes I_x$ defined as follows:
For given $\eta_x \in\ker\a_x\subset J^1I_x$, we represent it as $\eta_x = j^1s|_x$ by a germ of sections $s$
satisfying $s(x) = 0$. We put
$$
\d_x(\eta_x): = \nabla s|_x.
$$
Then this map $\d$ is a well-defined bundle map $\ker\a \to T^*N\otimes I$
which is an isomorphism. In particular, we have the following short exact sequence
\be\label{eq:short-exact}
0 \longrightarrow T^*N \otimes I \longrightarrow  J^1 I \stackrel{\alpha}{ \longrightarrow} I \longrightarrow  0.
\ee
\end{prop}
\begin{proof} We first prove its well-definedness, i.e., if 
$j^1 s|_x = j^1s'|_x$, then $\nabla_v s|_x = \nabla_v s'|_x$ for all $V \in T_x N$.

By definition, the standing hypothesis implies
\be\label{eq:s-s'}
s_x - s'_x \in{\bf a}_x{\bf b}_x
\ee
Then we compute
$$
\nabla_v s - \nabla_v s' = \nabla_v (s-s') = \Pi^v(d(s-s')(v)).
$$
We have
$$
d(s-s')(v) = \frac{d}{dt}\Big|_{t=0} (s-s')(c(t)) = 0
$$
since $s_x - s'_x \in{\bf a}_x{\bf b}_x$ and hence $(s-s')\circ c(t) = o(t)$. This proves that
the map is well-defined.
Once the map is proven to be well-defined, proving its injectivity and surjectivity is
straightforward and so left to the readers.
\end{proof}

\begin{rem} The above short exact sequence is established in the middle of 
the discussion \cite[p. 134]{loose} without detailed proof. The above
proposition provides the proof of this exactness somewhat differently from
that of \cite{loose}.
\end{rem}

\subsection{Canonical contact structure of $J^1I$}

In this subsection, we recall the definition of canonical contact structure of $J^1I$ from \cite[Section 2]{loose}.

Put $C = J^1I$ with $\pi: C \to I$ the natural projection,
 and consider its zero section embedding $\iota: N \hookrightarrow J^1I = C$. 
We have natural splitting $\iota^*TC = TN \oplus J^1I$ given by
$$
\iota^*TC|_x = T_xN \oplus J^1I_x \quad \forall x \in N.
$$
Therefore, by adding $\id: TN \to TN$ to the first two terms of \eqref{eq:short-exact},
 we  have another short exact sequence 
\be\label{eq:short-exact2}
0  \longrightarrow (T^*N \otimes I) \oplus TN \longrightarrow  \iota^*TC 
\stackrel{\widetilde \alpha}
 \longrightarrow I \longrightarrow 0
\ee
where $\widetilde \alpha:  \iota^*TC \to I$ is the map defined by 
$\widetilde \alpha_x(v_x,\eta_x) :  = \alpha(\eta_x)$ at $x \in N$.
 
We will lift this short exact sequence of vector bundles \emph{on $N$ to one on $C$} 
as in \eqref{eq:contact-sequence} 
$$
0 \longrightarrow H_C \to TC  \stackrel{\alpha_C}{\longrightarrow}  I_C \longrightarrow 0
$$
where $\alpha_C: TC \to I_C$ is an epimorphism. Note $\pi^*(\iota^*TC) = TC$ 
under the identification $N \cong \iota(N)$. 

We just put
$$
H_C = \pi^* \left((T^*N \otimes I) \oplus TN\right), \quad I_C: = \pi^*I, \quad \alpha_C = \pi^* \widetilde \alpha
$$
and define the bilinear form $\Omega_C$ to be
$$
\Omega_C: = -d\alpha_C: H_C \times H_C \to I_C.
$$

Note that the fiber map $\alpha_C|_\zeta: T_\zeta C \to (I_C)_\zeta$ at $\zeta \in C$ is nothing but 
$$
\alpha_C|_\zeta = \widetilde \alpha_{\pi(\zeta)}: \iota^*T_{\pi(\zeta)}C \to I_{\pi(\zeta)}
$$
with $\iota^*T_{\pi(\zeta)}C = T_{\pi(\zeta)}N \oplus I_{\pi(\zeta)}$, and so
\beastar
\ker \a_C|_\zeta & = & \ker\widetilde{ \a} _{\pi(\zeta)} 
=  (T_{\pi(\zeta)}^*N \otimes I_{\pi(\zeta)}) \oplus T_{\pi(\zeta)}N \\
& = & \pi^*\left((T^*N \otimes I) \oplus TN \right)|_\zeta = H_C|_\zeta
\eeastar
at all $\zeta \in C$ which proves $\ker \a_C = H_C$. Furthermore $\alpha_C$ also naturally descends to
a bundle isomorphism
$$
[\alpha_C]: TC/\xi_0 \to I_C
$$
by definition if we put $\xi_0 = \ker \a_C$. 
\begin{lem} The bilinear form $d\alpha_C: H_C \times H_C \to I_C$ is nondegenerate.
\end{lem}
\begin{proof} We recall that at $\zeta \in C$,
$$
H_C|_\zeta = (T_{\pi(\zeta)}^*N \otimes I_{\pi(\zeta)}) \oplus T_{\pi(\zeta)}N
$$
and that we evaluate $d\alpha_C|_\zeta = d(\pi^*\alpha)|_\zeta = d\alpha|_{\pi(\zeta)}$ against the pair of vectors
$$
\left((\beta_{\pi(\zeta)}\otimes e_{\pi(\zeta)}, v_{\pi(\zeta)}), (\beta'_{\pi(\zeta)}\otimes e_{\pi(\zeta)}, v'_{\pi(\zeta)})\right)
$$
for a (local) basis element $e_{\pi(\zeta)} \in I_{\pi(\zeta)}$. Then we obtain
$$
d\alpha|_{\pi(\zeta)}\left((\beta_{\pi(\zeta)}\otimes e_{\pi(\zeta)}, v_{\pi(\zeta)}), (\beta'_{\pi(\zeta)}\otimes e_{\pi(\zeta)}, v'_{\pi(\zeta)})\right)
= \left(\beta(v') - \beta'(v)\right)|_{\pi(\zeta)} \otimes e_{\pi(\zeta)}
$$
which clearly verifies nondegeneracy.
\end{proof}

\begin{defn}[Canonical contact structure of $J^1I$]  The canonical contact structure $\xi_0 \subset T(J^1I)$
is given by the hyperplane distribution 
\be\label{eq:xi0}
\xi_0 = \ker \alpha_C \subset TC,
\ee
such that the following holds:
\begin{enumerate}
\item The short exact sequence $0 \to \xi_0 \to TC \to TC/\xi_0 \to 0$ and \eqref{eq:short-exact2} 
induce  the commutative diagram
\be\label{eq:C-xi0}
\xymatrix{
0 \ar[r] & \xi_0 \ar[r]\ar[d]^{=} & TC  \ar[r] \ar[d]^{=} & TC/\xi_0 \ar[r] \ar[d]^{[\alpha_C]}& 0 \\
0 \ar[r] & \xi_0 \ar[r] & TC \ar[r]_{\alpha_C} &  I_C \ar[r] & 0.
}
\ee
\item The two form $\Omega_C: = -d\alpha_C$ is nondegenerate. We will simply call
$(\alpha_C,\xi_0)$ the \emph{canonical contact structure of the one-jet bundle $J^1I$}.
\end{enumerate}
\end{defn}

\subsection{Loose's tubular neighborhood theorem}

Let $(Q,\xi)$ be a contact manifold and $R \subset (Q,\xi)$ a Legendrian submanifold thereof.
We alert readers that \emph{we do not assume $\xi$ is coorientable and no contact form
is involved in the statement of Loose's tubular neighborhood theorem other than contact structure $\xi$.}

We consider the coorientation line bundle $I =I^\xi$ and its restriction to $R$
$$
I_R: = I^\xi|_R.
$$
Then we have the short exact sequence
$$
0 \longrightarrow \xi|_R \longrightarrow TQ \to I_R \longrightarrow \to 0
$$
together with the canonical quotient map $\alpha:  TQ|_R \to I_R$. We have the commutative diagram
$$
\xymatrix{
0 \ar[r] & \xi|_R \ar[r]\ar[d]^{\cong} & TQ|_R  \ar[r] \ar[d]^{\cong} & I_R \ar[r] \ar[d]^{[\alpha_C]}& 0 \\
0 \ar[r] & \xi_0|_R  \ar[r] & TC|_R \ar[r]_{\alpha_C} &  I_C|_R \ar[r] & 0
}
$$
with $C = J^1R$, where the isomorphism $\cong$ in the middle column is the canonical isomorphism
induced by the Five Lemma. We denote its inverse by 
\be\label{eq:ExpR}
\Exp|_R: TC|_R \to TQ|_R
\ee
which is a bundle isomorphism.

The following Legendrian tubular neighborhood theorem will be an important tool for the 
generic embeddedness of Legendrian compositions. It is a slightly refined version of 
\cite[Corollary in p.131]{loose}.

\begin{prop}\label{lem:loose-Darboux}
Let $R \subset (Q,\xi)$ be a Legendrian submanifold, and let $\Exp_R:= \Exp|_R$ be
the above bundle isomorphism.

Then there are neighborhoods $U_R \subset Q$ of $R$, $V_R \subset J^1I_R$ of the 
zero section $o_{J^1R}$, and a contactomorphism 
$$
\Psi_R: (V_R, \xi_0) \to (U_R, \xi)
$$
satisfying $\Psi_R|_R = \id_R$, $d\Psi|_R(\xi_0) \subset \xi$ and
$$
d\Psi|_R = \Exp_R
$$
for the derivative map
$$
d\Psi|_R: I^\xi|_R \to T(J^1I_R)/\xi_0
$$
\end{prop}

An immediate corollary is the following local parametrization of Legendrian submanifolds
nearby a given $R$.

\begin{cor} For any given Legendrian submanifold $R \subset (Q,\xi)$, there is a $C^\infty$ neighborhood 
$\CU_R \subset \mathfrak{Leg}(Q,\xi)$ of $R$ and a $C^\infty$-small neighborhood
$\CV_R \subset \Gamma(J^1R)$ of $0 \in \Gamma(J^1R)$ such that
the map $\Psi_R$ induces one-to-one correspondence
$$
\Psi_R: \CV_R  \to \CU_R 
$$
which is a diffeomorphism in the Frechet manifold sense.
\end{cor}

\section{Canonical symplectization of non-coorientable contact manifolds}

In this section, we discuss the canonical symplectization of a contact manifold $(Q,\xi)$,
not necessarily coorientable. For this purpose, we 
recall a \emph{coordinate-free} definition of symplectization \cite{arnold:book}.

\begin{defn}[Asymptotic rays] Let $Z$ be a vector field on a noncompact manifold $M$.
We say $Z$ is positively (resp. negatively) complete if all initial value problems of the ODE $\dot x = Z(x)$ can
be solved for all time in positive  (resp. negative direction). We call a solution, denote by $\ell_x$,
of $\dot \ell(t) = Z(\ell(t)), \, \ell(0) = x$ on $[0,\infty)$
a \emph{positive ray}, and one satisfying on $(-\infty,0]$ a \emph{negative ray}.
We identify two rays $\ell_x$ and 
$\ell_y$ of the  flow of $Z$ if there exists some $\tau_0> 0$ such that 
$$
\ell_y(\tau+\tau_0) = \ell_x(\tau)
$$
for all $\tau \geq 0$. We call an equivalence class thereof a \emph{positive asymptotic ray}.
We define a \emph{negative asymptotic ray} similarly.
\end{defn}
We will specialize this definition to the case of Liouville vector field $Z$ associated to
(noncompact) exact symplectic manifolds.

\subsection{Definition of ideal boundaries}

We start with some discussion on the complete Liouville manifold.
We first recall that for any manifold $(Q,\alpha)$ its symplectization is given by
$$
(Q \times \R, d\lambda), \quad \lambda = e^s \pi^*\alpha
$$
where $\pi: Q \times \R \to Q$ and $s: Q \times \R \to \R$ are
the canonical projections.

\begin{defn}
Consider an exact symplectic manifold $(M,\omega)$ with $\omega =d \lambda$. 
A \emph{Liouville vector field}
$Z$ associated to $\lambda$ is the unique vector field satisfying 
$$
Z \intprod  \omega = \lambda.
$$
We say $(M, d\lambda)$ is \emph{positively
Liouville-complete} (resp. \emph{negatively Liouville complete}) if there is
a contact manifold $(Q_+,\alpha_+)$ (resp. $(Q_-,\alpha_-)$) and a proper Liouville embedding
$$
\iota_+: Q_+  \times [0,\infty)  \to M \quad (\text{\rm resp.} \iota_-: Q_- \times (-\infty,0] \to M)
$$
where the domains equipped with the symplectic forms $d(e^s \pi^*\alpha)$ and $M$ with $d\lambda$.
We say $(M, d\lambda)$ is a complete Liouville manifold 
 if $Z$ is both positively and negatively complete and there exists
a compact subset $K \subset M$ such that
$$
M \setminus K = \iota_+(Q_+  \times [0,\infty)) \sqcup \iota_+(Q_+  \times (-\infty,0])
$$
where one of the factors is allowed to be empty.
\end{defn}

\begin{defn}[Ideal boundary]
Suppose that Liouville manifold $(M, d\lambda)$ is complete.
The \emph{positive ideal boundary}  (resp. \emph{negative ideal boundary}) $\del_\infty^+ M$
(resp. $\del_\infty^-M$) is defined to be the set of equivalence classes of
the positive (resp. negative) asymptotic rays of Liouville vector field $Z$.
Then the \emph{ideal completion} is the coproduct
$$
\overline M = M \coprod \del_\infty^+ M \coprod\del_\infty^- M
$$
equipped with the obvious topology.
\end{defn}
We mention that $\overline M$ in our definition carries a natural topology with respect to which 
$\overline M$ becomes a
compact Hausdorff topological space.  When $\del_\infty^-M$ (resp. $\del_\infty^+M$) is empty, we call
$\overline M$ a \emph{positive symplectic filling} (resp. \emph{negative symplectic filling}).

\begin{rem} We would like to emphasize that a priori the ideal boundary
defined as above does not carry a smooth structure unless  the Liouville manifold is
specified by the \emph{cylindrical structure at infinity} invariant under with respect to a
the  Liouville flow.
For example, the product $(M_1 \times M_2,\lambda_1 \times \lambda_2)$ is a
Liouville manifold but its ideal completion does not carry a canonical smooth structure 
on its ideal boundary. The natural smooth structure provides only a stratified one such that
$$
\del_\infty (M_1 \times M_2 ) = (\del_\infty M_1 \times M_2) \coprod (M_1 \times \del_\infty M_2)
\coprod (\del_\infty M_1 \times \del_\infty M_2)
$$
as a stratified manifold. (See \cite{oh:gradient-sectorial} and \cite{gps2} for some relevant discussion
on the Liouville structure on the product of two Liouville manifolds.)
We refer readers to \cite{oh-park:ideal-bdy} for further expounding of this smoothness issue.
\end{rem}

We now specialize the above definition to the case of symplectization. (See \cite{arnold:book}.)

\subsection{Geometry of intrinsic symplectization}

Recall the standard definition of canonical symplectization.
\begin{defn}[Canonical symplectization] We call 
$$
SQ =\{\alpha\in T^*Q\setminus o_{T^*Q}, |\, \alpha|_\xi=0\},
$$
equipped with the exact symplectic form induced by
the standard symplectic form $\omega_0 =-d\theta$ on $T^*Q$ to $SQ$.
\end{defn}

If we regard $\lambda = -\theta$ as the Liouville one-form on $T^*Q$, then the associated
Liouville vector field is given by the Euler vector field
$$
\vec E = \sum_i p_i \frac{\del}{\del p_i}.
$$
In this regard, we also write the Liouville vector field
$$
Z = X_{-\theta} : = \vec E
$$
which satisfies $Z \intprod \omega_0 = -\theta$.

Since the Liouville flow, which is given by the formula
$$
\psi^t_{SQ}(\beta) = \phi_{\vec E}^t(\beta) = (q,e^t\, p),
$$
it preserves $\theta$, and $\theta$ canonically 
endows the set of Liouville trajectories of $SQ$ with a 
contact structure  on the quotient $SQ/\sim$ which is induced by $\ker\theta$. (See \cite{arnold:book}.)
\begin{lem}\label{lem:homo-lifting} Any contactomorphism $\psi: Q \to Q$ is lifted to an $\R_+$-equivariant, i.e., $\R_+$-homogeneous symplectomorphism $\widetilde \psi: SQ \to SQ$ given by
\be\label{eq:tildepsi}
\quad \widetilde \psi: = (d\psi^{-1})^*|_{SQ}.
\ee
Conversely, any $\R_+$-homogeneous symplectomorphism of $SQ$ descends to a contactomorphism on $Q$. 
\end{lem}

\begin{lem} $SQ$ is connected if and only if $(Q,\xi)$ is not coorientable.
Furthermore if $(Q,\xi)$ is coorientable, then we have the decomposition 
$$
SQ = SQ_+ \sqcup SQ_-
$$
given by
$$
SQ_\pm = \{\alpha \in SQ \mid \alpha(\pm R_\lambda) > 0\}.
$$
\end{lem}
\begin{proof} 
Since the Liouville flow, , denoted by $\psi^t_{SQ}$, satisfies $(\psi^t_{SQ})^*(-\theta) = -\theta$,
the set of Liouville trajectories naturally  carries a contact structure which is contactomorphic to 
\be\label{eq:tildeQ}
(SQ/\R_\pm, [\ker \theta]).
\ee
Each of them is diffeomorphic to the canonical  (co)orientation double cover of $Q$ equipped with 
$\widetilde \xi$ the contact structure lifted from $\xi$ of $Q$. The first statement is a well-known
fact on the orientation double cover.
\end{proof}
 
Now we make a couple of immediate observations on the coorientable case. 

\begin{rem} First we note that 
a choice of nowhere vanishing section $Q \to SQ$ is equivalent to a choice of contact form $\lambda$
on $(Q,\xi)$. In the coorientable case, we have
 $SQ_+ \cap SQ_- = \emptyset$. Moreover each of $S_\pm Q$ is a
$\R_+$ principal bundle and so connected, if $Q$ is connected. 
Furthermore the two connected components of $SQ$ are
diffeomorphic to the $Q \times\mathbb{R}_+$.
\end{rem}

This remark being made,  we will focus on the case of \emph{non-coorientable} contact manifolds.
Suppose that $\xi$ is not coorientable and consider its coorientation local system $I^\xi = TQ/\xi$. Then 
the bundle $I^\xi$ is a nontrivial $\R$-bundle so that its frame bundle $\pi: \text{\rm Fr}(I^\xi) \to Q$
is a nontrivial $S^0 \cong \Z_2$ principal bundle, which is a non-trivial double
covering of $Q$. Therefore the total space $\text{\rm Fr}(I^\xi)$ is connected.
The double cover $\widetilde Q \to Q$ is naturally equipped with the pull-back contact structure 
$\widetilde \xi = \pi^*\xi$, which is automatically coorientable, and hence carries global contact forms.

For the later purpose, we would like to explain a natural procedure of converting a
\emph{generalized contact form} on $Q$ to a \emph{contact form} on $\widetilde Q$
when a Riemannian metric $g$ is equipped with the base $Q$.

Denote by $\Riem(Q)$ the space of Riemannian metrics $Q$, and let 
$\Gamma(I^\xi)$ be the set of smooth sections of the coorientation line bundle  $I^\xi \to Q$. 

\begin{prop}\label{prop:conversion} Let $(Q,\xi)$ be any contact manifold.
We fix  a Riemannian metric $g$ of $Q$ which naturally pulls back to $\widetilde Q$,
again denoted by $g$. Then there exists a unique contact form $\lambda^g$  on $\widetilde Q$ 
that is characterized by the following defining properties:
\begin{enumerate}
\item $|\lambda^g_{\widetilde q}|= 1$ for all $\widetilde q \in \widetilde Q$.
\item The value $\lambda^g_{\widetilde q} \in T^*\widetilde Q$ descends to $T^*_{\pi(\widetilde q)}Q$
as the unique element in $S_{\pi(\widetilde q)} Q \cap U^*Q$  representing $\widetilde q \in \widetilde Q$
as the linear functional $T_{\widetilde q}Q \to \R$, where $U^*Q$ is the unit conormal bundle of $g$.
\end{enumerate}
We call the contact form $\lambda^g$  a \emph{metric contact form} of $(\widetilde Q,\widetilde \xi)$.
\end{prop}
\begin{proof} We will explicitly construct a contact form, i.e.,
a nowhere vanishing section
$$
\lambda: \widetilde Q \to T^*\widetilde Q.
$$
Using the metric $g$ on $Q$, we denote by  
$$
\alpha^g(\widetilde q)  \in S_{\pi(\widetilde q)} Q
$$
the unique element in $S_q Q$ that represents $\widetilde q$ and satisfies
\be\label{eq:|alpha|=1}
|\alpha^g(\widetilde q)|_g = 1, \quad \alpha^g(\widetilde q)|_{\xi_q} = 0.
\ee
If we identify $\widetilde Q$ with $SQ \cap U^*Q$ 
as a subset of the unit conormal bundle, denoted by $U^*Q$, of $g$, then we have 
\be\label{eq:alphag=tildeq}
\alpha^g(\widetilde q) = \widetilde q
\ee
by definition.

Therefore $\alpha^g$ induces a map 
\be\label{eq:alphag-map}
\alpha^g: \widetilde Q \to U^*Q \cap SQ.
\ee
Finally  at each $\widetilde q$, we define $\lambda^g_{\widetilde q} \in T^*_{\widetilde q}Q$ by
$$
\lambda^g_{\widetilde q} (v)  : = \alpha^g(\widetilde q)(d\pi(v))
$$
for all  $v \in T_{\widetilde q} \widetilde Q$. Clearly the element $\lambda^g_{\widetilde q}$
is well-defined and linear on $T_{\widetilde q}Q$.
Obviously the value vanishes if $v \in \widetilde \xi_{\widetilde q}$
and does not vanish if $d\pi(v) \neq 0$. Since $(Q,\xi)$ is contact
and $d\pi: T\widetilde Q \to TQ$ is a surjective map, we conclude that  $\lambda^g_{\widetilde q} \neq 0$
for all $\widetilde q \in \widetilde Q$. Furthermore the assignment $\widetilde q \mapsto \lambda^g_{\widetilde q}$
is smooth which can be locally checked. Combining the above discussion, we have proved that $\lambda^g$ is a contact form
for $(\widetilde Q,\widetilde \xi)$.  The defining properties (1), (2) of $\lambda^g$ also
follow from \eqref{eq:|alpha|=1} and \eqref{eq:alphag=tildeq}.
This finishes the proof.
\end{proof}

The above proof in particular shows a useful fact that a choice of Riemannian metric on $Q$ naturally 
gives rise to a contact form $\lambda^g$ that has the above defining property.

 Recall $SQ$ is the $\R^*$ principal bundle consisting of (local) contact forms which
 carries the aforementioned free involution, i.e., non-trivial $\Z_2$-action on $\R^*$.
 The following is an easy lemma which is useful for our further discussion on $SQ$.

\begin{lem} 
We can express $SQ$ as the fiber product
$$
SQ = \text{\rm Fr}(I^\xi) \times_{\Z_2} \R^*
$$
so that 
$$
SQ/\R_+  \cong \text{\rm Fr}(I^\xi) \times_{\Z_2} \R^*/\R_+ \cong \text{\rm Fr}(I^\xi).
$$
\end{lem}

\section{Definitions of contact product and contact projection}
\label{sec:product}

Let $(Q_0,\xi_0)$ and $(Q_2,\xi_1)$ be contact manifolds. A
natural product operation of two contact manifolds $Q_0$ and $Q_2$, which is called as `contact product' in 
\cite{oh:shelukhin-conjecture} was introduced for any pair of contact manifolds \emph{not necessarily
coorientable}. In this section, we start to further develop this operation with more categorical
perspective. Again we will focus on the non-coorientable case.

In fact, we can identify $(\widetilde Q, \widetilde \xi)$ with the following
much more visually suggestive functorial notation of the previously
described double covering $\widetilde Q$ with involution $\tau \neq \id$.

\begin{notation}[Directed projectivization]\label{nota:P+} We will also write
\be\label{eq:P+}
\P_+(SQ) = SQ/\R_+ \subset \P_+(T^*Q \setminus \{0\})
\ee
and call the pair
$$
(\P_+(SQ), [\ker i^*\theta]).
$$
the \emph{directed projectivization} of $SQ$. We denote by $\tau: \P_+(SQ) \to \P_+(SQ)$
the canonical contact involution.
\end{notation}

\subsection{Intrinsic canonical symplectization}

We start with some preparatory discussion of the canonical symplectization given by the model
$$
(\widetilde Q, \widetilde \xi) = \left(\P_+(SQ), \ker \alpha\right), \quad \alpha = [i^*\theta]
$$
which we defined in Proposition \ref{prop:conversion} of the previous section, 
where $\theta$ is the Liouville one-form on $T^*Q$.

Recall that $SQ \to Q$ is the principal $\R^*$ bundle and so carries a natural
$\Z_2$-action induced by the action of $\{1, -1\}$ thereon. We denote by $\sigma$
the non-trivial involution induced by the action of $-1$.

With a Riemannian metric equipped with $Q$, we introduce the contact form $\lambda^g$ induced by
$g$ in Proposition \ref{prop:conversion} and the map
$\pi_{SQ}^g: \widetilde Q \times \R^* \to SQ$ defined by 
\be\label{eq:piSQ}
\pi^g_{SQ}(\widetilde q,r) = r \alpha^g(\widetilde q)
\ee
for the element introduced $\alpha^g(\widetilde q)$ as an element in $S_{\pi(\widetilde q)}Q$ in the proof of Proposition \ref{prop:conversion}.
Then the aforementioned $\Z_2$ action on $SQ$ is given by 
\be\label{eq:Z2-actiononSQ}
\sigma(r \alpha^g(\widetilde q))= - r \alpha^g(\tau(\widetilde q)).
\ee
where we use the vector space structure of $T_q^*Q$ for $-1$ appearing in 
the outside of the right hand side.
\begin{lem}\label{lem:piSQ}
The map $\pi_{SQ}^g$ is a $\Z_2$-equivariant 
diffeomorphism under the $\Z_2$-actions equipped with its domain and codomain.
\end{lem}
\begin{proof} 
Let $\tau: SQ \to SQ$ be the involution induced by the anti-symplectic involution 
$\tau: T^*Q \to T^*Q$, $(q,p) \mapsto (q,-p)$. (This canonically induces the
aforementioned involution $\tau: \widetilde Q \to \widetilde Q$. Because of this, we
use the same letter $\tau$ for the two involutions.)
Therefore we have
$\tau(\alpha^g(\widetilde q)) = - \alpha^g(\widetilde q)$.
This implies
\be\label{eq:equivariance}
\pi^g_{SQ}(\tau(\widetilde q),-r) = (-r) \alpha^g(\tau(\widetilde q)) =  -r \alpha^g(\tau(\widetilde q))=
\sigma(\alpha^g(\widetilde q)).
\ee
This proves the map is $\Z_2$-equivariant. 

Since smoothness of the map $\pi^g_{SQ}$ is obvious and the equation \eqref{eq:equivariance} also
implies that $\pi_{SQ}^g$ is surjective,  it remains to prove
its injectivity.

Suppose
$$
 e^{\eta} \alpha^g(\widetilde q) =  e^{\eta'} \alpha^g(\widetilde q')
$$
for two pairs $(\widetilde q, \eta)$ and $(\widetilde q',\eta')$.
Since we assume $|\alpha^g(x)| = 1$ for all $x \in \widetilde Q$, the last equation first implies $\eta =\eta'$ and then
$$
\alpha^g(\widetilde q) =  \alpha^g(\widetilde q')
$$
in $S_qQ$ with $q = \pi(\widetilde q) = \pi(\widetilde q')$.
On the other hand, the assignment  $\widetilde Q \to SQ \cap U^*Q$ given by $\widetilde q \mapsto \alpha^g(\widetilde q)$
is a bijective map by  \eqref{eq:alphag=tildeq}. Obviously \eqref{eq:alphag=tildeq}
 also proves 
$\widetilde q = \widetilde q'$ and so $\pi_{SQ}^g$ is injective.
\end{proof}

We also define the exponential map $\widetilde Q \times \R \to \widetilde Q \times_{\Z_2} \R^*$ to be
$$
(\widetilde q, \eta) \mapsto  [(\widetilde q, e^{\eta})].
$$
The following proposition shows that two exact symplectic
manifolds $(\widetilde Q\times\R,-d(e^s \lambda^g))$ and  $(SQ,- d i^*\theta)$ 
are isomorphic as a Liouville manifold. Since the coorientable case is easier, we will
focus on the non-coorientable case.

\begin{prop}\label{prop:psi} Assume $(Q,\xi)$ is non-coorientable and fix a Riemannian metric $g$ on $Q$.
Let $\lambda^g$ be the metric contact form on $\P_+(SQ)$ induced by $g$ 
introduced in Proposition \ref{prop:conversion}. 
Consider the map
$\psi: = \pi_{SQ}^g :\P_+(SQ) \times \mathbb R \to SQ$.
Then $\psi$ is a Liouville isomorphism such that 
$$
\psi^*\theta = e^\eta \lambda^g, \quad  \psi_*\left(\frac{\del}{\del \eta}\right) = \vec E \, (=: X_{-\theta})
$$
\end{prop}
\begin{proof}  For each $\eta \in \R$,
We define $\psi^\eta: \P_+(SQ)  \to SQ$ by
$$
\psi^\eta(\widetilde q): = \psi(\widetilde q,\eta) = e^\eta \pi_*\lambda^g|_{\widetilde q} = e^\eta\alpha^g(\widetilde q)
$$
We compute
$$
(\tau \circ \psi^\eta)(\widetilde q) = \psi^\eta(-\widetilde q) = e^\eta \alpha^g(-\widetilde q)
$$
where the second equality follows since
$\alpha(-\widetilde q) \in S_qQ$ is the unique element
representing $-\widetilde q$ with unit norm,  by definition. (See Proposition \ref{prop:conversion}.)

We compute
\bea\label{eq:vanishing-E}
&{}&
((\tau \circ \psi^\eta)_*\lambda^g)(\vec E) |_{e^\eta \alpha^g(-\widetilde q)} \nonumber\\
 & =  & 
(\tau_*(\psi^\eta)_*\lambda^g)(\vec E) |_{e^\eta \alpha^g(-\widetilde q)} 
= \lambda^g \left(d(\psi^\eta)^{-1}d\tau^{-1}\left(\vec E |_{e^\eta\alpha^g(- \widetilde q)}\right)\right) 
\nonumber \\
& = & \lambda^g \left(d(e^\eta\pi_*\lambda^g)^{-1}\left(\vec E |_{e^\eta\alpha^g(\widetilde q)}\right)\right)=\frac{1}{e^\eta} 
\theta|_{e^\eta\alpha^g(\widetilde q)} \left(\vec E|_{e^\eta\alpha^g(\widetilde q)}\right) =
0
\eea
where the penultimate equality follows from the defining property of the Liouville form $\theta$
and the last vanishing from $\theta(\vec E) = 0$.

Recall the definition of the one-form $\lambda^g$ on $\widetilde Q$ given by
\be\label{eq:lambda-alphag}
\lambda^g|_{\widetilde q}(v) = \alpha^g(\widetilde q)(d\pi(v))
\ee
from Proposition \ref{prop:conversion}.
We regard $\psi^0=\alpha^g: \widetilde Q \to SQ$ as a smooth map \eqref{eq:alphag-map}. Near $\widetilde q\in\widetilde Q$, $\alpha^g$ is a section.
Then by \eqref{eq:lambda-alphag} and the property of the tautological one-form $\theta$, we derive
\bea\label{eq:theta-2}
\psi^*\theta & = & (e^\eta\psi^0)^*\theta =\theta\left(\frac{\partial\psi}{\partial\eta}\right)d\eta+e^\eta(\psi^0)^*\theta=\theta\left(\vec E|_{e^\eta\alpha^g}\right)+e^\eta\lambda^g
\eea
near $\widetilde q\in\widetilde Q$.
Combining \eqref{eq:vanishing-E} and \eqref{eq:theta-2}, we have proved
$$
\psi^*\theta = e^\eta \lambda^g
$$
This finishes the proof.
 \end{proof}
\subsection{Definition of monoidal contact product}
We are now ready to give the definition of the contact product which we will show
defines the categorical monoidal structure on the set of contact manifolds.

With slight abuse of notation, we will write
$$
\pi_i^*i^* \theta =: \pi_i^*\theta
$$
here and henceforth by omitting the canonical inclusion map $i: SQ \to T^*Q$ form the left hand side.

\begin{defn}[Contact Product]\label{defn:contact-product}
We define the contact product of two contact manifolds $(Q_1,\xi_1)$, $(Q_2,\xi_2)$ by
$$
\left(\P_+(SQ_1 \times SQ_2),[\ker(\pi_1^*\theta_1 + \pi_2^*\theta_2)]\right).
$$
We denote the contact product by 
$$
(Q_1\star Q_2, \xi_1 \star \xi_2) 
$$
where the hyperplane distribution field $\xi_1 \star \xi_2$ is the projection of the 
$\R_+$-invariant distribution $\ker(\pi_1^*\theta_1 + \pi_2^*\theta)$ to the quotient space
$Q_1\star Q_2 = \P_+(SQ_1 \times SQ_2)$.
\end{defn}

Here recall that the equivalence class is given by the flow of the Liouville vector field 
associated to the one-form $\pi_1^*\theta_1+\pi_2^*\theta_2 = \theta_1 \oplus \theta_2$.  
The relevant Liouville manifold  is
$$
(SQ_1 \times SQ_2, \omega_{0,1} \oplus \omega_{0,2}), \quad \omega_{0,i} = - d\theta_i, \, i = 1,\, 2.
$$
The associated Liouville vector field $Z$, defined by 
$$
Z \intprod (\omega_{0,1} \oplus \omega_{0,2}) = - \pi_1^*\theta_1- \pi_2^*\theta_2
$$ 
is nothing but the Euler vector field
\be\label{eq:Z=Euler}
Z = E_{(-\theta_1 \oplus -\theta_2)} = E_{(- \theta_1)} \oplus E_{(- \theta_2)}.
\ee
Recalling that the flow of Euler vector field $E = p \frac{\del}{\del p}$
is given by the map $(t,q,p) \mapsto (q,e^t p)$ which is
 nothing but the diagonal  multiplication by positive real numbers $r = e^t > 0$.

Utilizing Proposition \ref{prop:psi}, we prove the following generalization of a result from
\cite{oh:shelukhin-conjecture}. 

\begin{thm} Let $(Q_1,\xi_1)$ and $(Q_2,\xi_2)$ be any pair of contact manifolds. 
The map $\id \times \psi: SQ_1 \times (\P_+(SQ_2) \times \R) \to SQ_1 \times SQ_2$
descends to a contactomorphism
$$
\CF: SQ_1\times \P_+(SQ_2)\to Q_1\star Q_2
$$
defined by
$$
\CF((\alpha_{q_1},(\widetilde q_2,s)) = [(\alpha_{q_1}, \psi(\widetilde q_2,s))].
$$
\end{thm}
\begin{proof} We know that the map
$$
\id \times \psi:SQ_1\times (\P_+(SQ_2)\times\R) \to SQ_1\times SQ_2
$$
is a symplectic diffeomorphism induced by $\psi$ defined in Proposition \ref{prop:psi}. 
The Liouville vector field of the product 
$$
SQ_1\times \P_+(SQ_2)\times\R
$$
is given by $X: = (\vec E_1,0, \frac{\del}{\del s})$ and that of $SQ_1 \times SQ_2$ is
$\vec E_1 \oplus \vec E_2$. We have already shown
$$
(\id \times \psi)_*X = \vec E_1 \oplus \vec E_2
$$
and hence $(\id \times \psi)(\{s = 0\})$ is a hypersurface of 
$SQ_1 \times SQ_2$ transverse to $(\id \times \psi)_*X$.
Therefore by quotienting out the domain and the codomain, we have
proved the theorem.
\end{proof}

\begin{cor}
The contact product $(Q_1 \star Q_2,\xi_1 \star \xi_2)$ is coorientable for every pair of contact manifolds $(Q_1,\xi_1)$ and
$(Q_2,\xi_2)$.
\end{cor}
\begin{proof} We have only to note that the natural contact structure $\Xi$ has its fiber $\Xi_{(\alpha_1,\widetilde q_2)}$ at 
$(\alpha_1,\widetilde q_2) \in SQ_1 \times \P_+(SQ_2)$ 
 given by the direct sum
$$
\Xi_{(\alpha_1,\widetilde q_2)} = [d(\id \times \psi)^{-1}(\ker (\theta_1 \oplus \theta_2))]= \ker \, \CF^*([\theta_1 \oplus \theta_2])
$$
where $[\theta_1\oplus\theta_2] =i^*(\theta_1\oplus\theta_2)$ for the inclusion $i:(\id\times \psi)(\{s=0\})\to SQ_1\times SQ_2$.
 In other words, we have shown that
the one-form 
$$
 \CF^*([\theta_1 \oplus \theta_2])
$$
is a globally defined contact form of $\Xi$ on $SQ_1 \times \P_+(SQ_2)$. This finishes the proof.
\end{proof}

The following proposition will give rise to the construction of canonical product operation
which equips the set of contact manifolds with the structure of \emph{monoid} in the precise 
categorical sense to be explained in the next few sections.

\begin{prop} \label{prop:pi} There exist canonical  projections 
$p_i: Q_0 \star Q_1 \to Q_i$ for $i=0, \, 1$
such that the following diagram commutes:
$$
\xymatrix{ SQ_0 \times SQ_1  \ar[dr]_{\widetilde \pi_i} \ar[rr]^{\pi} && Q_0 \star Q_1 \ar[dl]^{p_i}\\
& Q_i &
}
$$
i.e., that satisfy $p_i \circ \pi = \widetilde \pi$. We call $p_i$ the $i$-th \emph{contact projection}
for each $i =0, \, 1$, denote it by 
$$
\pi_i^\star: = p_i.
$$
\end{prop}
\begin{proof} 
Let $p_i^{SQ}:SQ_1\times SQ_2\to SQ_i$ be the natural $i$-th projections for $i=1, \, 2$.
Recall the expression
$$
Q_0 \star Q_1 = \P_+(SQ_0 \times SQ_1)
$$
of the contact product for which we have natural projection $\P_+(SQ_0 \times SQ_1)
 \to Q_i$ for $i = 0, \, 1$.
Since the $(i+1\mod 2)$-component  of the Liouville vector field under $p_i^{SQ}$ is zero, 
the map descends to the projection $p_i:Q_1 \star Q_2\to Q_i$
that induces the following commutative diagram
\be\label{eq:star-diagram}
\xymatrix{
   SQ_1\times SQ_2 \ar[d]^{p_i^{SQ}} \ar[r]^{\pi_{12}} & Q_1 \star Q_2 \ar[d]^{p_i}\\
        SQ_i \ar[r]^{\pi_i}                                                 & Q_i
}
\ee 
Obviously this diagram induces the proposition which finishes the proof.
\end{proof}

\begin{rem}
One can show that for coorientable contact manifolds $(Q_i,\alpha)$ 
the above functorial definition of contact structure $\xi_1 \star \xi_2$ coincides with the definition 
\bea\label{lem:product-coorientable}
(Q_1 \times Q_2 \times \R, \Xi), \quad \Xi = \ker \mathscr{A}
\eea
for the contact form $\mathscr{A} = e^{\eta} \pi_1^*\alpha_1 + \pi_2^*\alpha_2$. 
In the non-coorientable case,  one may not expect such a coordinate-wise expression. However,
we have shown in the previous section that any non-coorientable manifold has a canonical double cover 
equipped with the contact structure induced by the original contact structure, which is coorientable.
Exploiting this double cover, we will slightly modify the definition of 
the contact product $Q_1 \star Q_2$ so that any contact product
becomes coorientable.
\end{rem}

\subsection{Definition of canonical $\R$-coordinates $\eta$}

We also have a canonical coordinate map $\eta: Q_1\star Q_2 \to \R$ defined as follows.
Consider the product map
$$
\psi_1 \times \psi_2: (\P_+(SQ_1) \times \R) \times (\P_+(SQ_2) \times \R) \to SQ_1 \times SQ_2  
$$
which is a $\Z_2$-equivariant Liouville isomorphism
under the natural diagonal actions on the domain and the codomain by Proposition \ref{prop:psi}.
It induces a map
$$
(\P_+(SQ_1) \times \R) \times (\P_+(SQ_2) \times \R)/\sim \, \to \P_+(SQ_1 \times SQ_2) = Q_1 \star Q_2.
$$
Combining this with Proposition \ref{prop:psi},  we also have canonical contactomorphism
\be\label{eq:psi1psi2-1}
[\psi_1 \times \psi_2]^{-1}:Q_1 \star Q_2 = \P_+(SQ_1 \times SQ_2) 
\to \P_+(SQ_1)\times \P_+(SQ_2)\times \R. 
\ee
By writing the standard coordinates $\eta_i$ with $i=1, \, 2$ of the factors $\R$ appearing above
and taking the explicit identification
$$
(\P_+(SQ_1) \times \R) \times (\P_+(SQ_2) \times \R) \cong (\P_+(SQ_1) \times Q_2) \times \R^2
$$
given by 
$$
((\widetilde q_1, \eta_1), (\widetilde q_2, \eta_2)) \mapsto (\widetilde q_1, \widetilde q_2),(\eta_1, \eta_2)),
$$
we obtain the projection
$$
\pi_\R:Q_1 \star Q_2 = \P_+(SQ_1 \times SQ_2) \to SQ_1 \times \P_+(SQ_2) \to \R^2/\sim \, \cong \R.
$$
\begin{defn}[Coordinate function $\eta$]\label{defn:eta}
We define the real-valued function $\eta:=\pi_\R$ to be the canonical $\R$-coordinate projection
of $\P_+(SQ_1 \times SQ_2) \cong \P_+(SQ_1)\times \P_+(SQ_2)\times \R$.
\end{defn}
The coordinate map $\eta = \eta_{12}$ is explicitly given by 
\be\label{eq:etaR}
\eta_{12}([(\widetilde q_1, \eta_1), (\widetilde q_2, \eta_2)]) := \eta_2 - \eta_1
\ee
under the above sequence of identifications.
Clearly the function $\eta_{12}$ is well-defined on the quotient space of the diagonal action 
$$
(s,(\eta_1,\eta_2))\mapsto (\eta_1 + s,\eta_2 + s)
$$
which is induced by the flow $\psi_{\vec E_1 \oplus \vec E^2}^s$ 
of the Liouville vector fields $\vec E_1 \oplus \vec E_2$ on $SQ_1 \times SQ_2$.

This completes our discussion of the contact product of contact manifolds 
\emph{without assuming coorientability}.

\section{The K-group functor $K:\mathfrak{Cont} \rightarrow \mathfrak{Group}$}
\label{sec:ktheory}

We now consider the $K$-group $K^0(Q)$. Let $(\Vect(Q), \oplus)$ be the abelian monoid
of vector bundles on $Q$. 
Then $K^0(Q)$ is defined by the 
associated abelian group. As usual, we denote by $[E]$ the element represented by
the vector bundle $E \to Q$. (See \cite{atiyah:ktheory} for the standard definitions of the $K$-group and of its 
(internal) ring structure.)

\subsection{External tensor product and direct sum}

We would like to construct the operations of
\emph{external tensor product} (resp. \emph{external direct sum})
denoted by $\starboxtimes$ (resp. denoted by $\starboxplus$)
$$
\starboxtimes, \, \starboxplus : K^0(Q) \times K^0(Q') \longrightarrow K^0(Q \star Q').
$$
Let $E \to Q$ and $E' \to Q'$ be given. We take the pull-back bundles 
$\pi^*E$ and $(\pi')^*E'$ on $SQ \times SQ'$ under the projections
$\pi: SQ \to Q$ and $\pi':SQ' \to Q'$. 

Consider the vector bundles
$$
\pi^*E \oplus (\pi')^*E', \quad  \pi^*E \otimes (\pi')^*E'
$$
on $SQ \times SQ'$.
\begin{defn-lem} Let $Q\star Q'$ be the contact product of contact manifolds $(Q,\xi)$ and $(Q', \xi')$,
and consider the projection 
$\pi_{QQ'}: SQ \times SQ' \to Q\star Q'$. 
Both vector bundles $\pi^*E \oplus (\pi')^*E'$, $ \pi^*E \otimes (\pi')^*E'$ are projectible to
$Q \star Q'$. We define the \emph{external direct sum}  and 
the \emph{external tensor product}
$$
E \starboxplus E', \quad E \starboxtimes E'
$$
by the associated descendants on $Q\star Q'$, respectively.
\end{defn-lem}
\begin{proof} Obviously both vector bundles are invariant
under the diagonal Liouville action of $\R_+$ on $SQ \times SQ'$. It also preserves
$$
\ker (\pi^*\theta_1 + \pi^*\theta_2) \subset T(SQ \times SQ').
$$
Let $x \in Q\star Q'$ and $(\beta,  \beta' ) \in  SQ \times SQ'$ be its representative.
Put
$$
\pi(\beta) = q, \quad \pi(\beta') = q'.
$$

We start with $\starboxplus$.
Consider the set 
\be\label{eq:Fx}
F_x: = E_{\pi(\beta)} \oplus E_{\pi'(\beta')}.
\ee
We claim that the right hand side does not depend on the choice $(\beta,\beta')$
representing $x$. For let $(\alpha,\alpha')$ be another representative.
Then there is $t \in \R$ such that
$$
(\alpha,\alpha') = (e^t\beta, e^t \beta').
$$
Obviously, we have
$$
\pi(\alpha) = \pi(\beta), \quad \pi'(\alpha') = \pi'(\beta').
$$
This proves the definition \eqref{eq:Fx} of $F_x$ is well-defined.
We then define
$$
E\starboxplus E': = \bigcup_{x \in Q\star Q'} \{x\} \times F_x.
$$
It is straightforward to check that $E \starboxplus E'$ defines a smooth
vector bundle on $Q \star Q'$ which we leave to the readers.

The same argument applies to $\starboxtimes$ which leads to the definition of
$E\starboxtimes E'$. This finishes the proof.
\end{proof}

We now apply this construction to the contact distributions $\xi$ and 
coorientation local systems $I^\xi$ associated thereto.

\begin{exm}\label{exm:starboxplus-star} Consider two contact manifolds $(Q_1, \xi_1)$ and $(Q_2,\xi_2)$ and their
contact product $Q_1\star Q_2$. Then, by definition of $\xi_1 \star \xi_2$ from Definition \ref{defn:contact-product}, 
we have a short exact sequence
\be\label{eq:star-short-sequence}
0 \longrightarrow (\xi_1 \starboxplus \xi_2) \oplus \R \langle [E_1,-E_2]\rangle \longrightarrow
 \xi_1 \star \xi_2 \stackrel{[\theta_1 - \theta_2]}{\longrightarrow} \R \to 0.
\ee
In particular, $\xi_1 \starboxplus \xi_2$ is canonically a direct summand of $\xi_1 \star \xi_2$.
\end{exm}

The following lemma is an important functorial relationship between the two external operations
$\starboxplus$ and $\starboxtimes$.
\begin{lem}\label{lem:iso-starboxtimes} We have a canonical isomorphism
$$
\frac{I^{\xi_1} \starboxplus I^{\xi_2}}{\ker\theta_{12}} \cong I^{\xi_1} \starboxtimes I^{\xi_2}
$$
induced by a surjective linear map $\theta_{12}:I^{\xi_1} \starboxplus I^{\xi_2}\to\R$. 
\end{lem}
\begin{proof} Recalling the definitions
\beastar
I^{\xi_1} \starboxplus I^{\xi_2} & = & (\pi_1^\star)^*I^{\xi_1} \oplus (\pi_2^\star)^*I^{\xi_2} \\
I^{\xi_1} \starboxtimes I^{\xi_2} & = & (\pi_1^\star)^*I^{\xi_1} \otimes (\pi_2^\star)^*I^{\xi_2},
\eeastar
from \eqref{eq:Fx},
we consider the linear map
$$
\theta_{12}: (\pi_1^\star)^*I^{\xi_1} \oplus (\pi_2^\star)^*I^{\xi_2} \to \R.
$$
defined by 
$$
\theta_{12}(e_1,e_2) = \theta_1(e_1)+ \theta_2(e_2)
$$
for $e_i \in  (\pi_i^\star)^*I^{\xi_1}$ for $i = 1, \, 2$.
Then the linear map $\theta_{12}$ descends to  a nontrivial linear map 
defined on the quotient
$\frac{I^{\xi_1} \starboxplus I^{\xi_2}}{\ker \theta_{12}}$.
But we have the canonical isomorphism
$$
\frac{I^{\xi_1} \starboxplus I^{\xi_2}}{\ker \theta_{12} }
\cong  \frac{TQ_1 \oplus TQ_2}{ \xi_1 \oplus \xi_2 \oplus\ker \theta_{12}}
$$
induced by the defining exact sequences
$$
0 \to \xi_i \to TQ_i \to I^{\xi_i} \to 0, \quad i = 1, \, 2
$$
of $I^{\xi}$. Therefore we have 
$$
\left(\frac{I^{\xi_1} \starboxplus I^{\xi_2}}{\ker \theta_{12}}\right)^*
\cong \left(\xi_1 \oplus \xi_2 \oplus \ker\theta_{12}\right)^\perp \cong \R \langle \theta_1 \oplus\theta_2 \rangle.
$$
By taking its dual, we obtain
$$
\frac{I^{\xi_1} \starboxplus I^{\xi_2}}{\ker \theta_{12}}
 \cong (\R \langle \theta_1 \oplus\theta_2 \rangle)^*.
$$
On the other hand, by definition of the tensor product $(\pi_1^\star)^*I^{\xi_1} \otimes (\pi_2^\star)^*I^{\xi_2}$, 
we have a natural isomorphism
\beastar
(\pi_1^\star)^*I^{\xi_1} \otimes (\pi_2^\star)^*I^{\xi_2} & \cong &
\Hom(((\pi_1^\star)^*I^{\xi_1})^* \otimes ( (\pi_2^\star)^*I^{\xi_2})^*, \R)\\
& = & \Hom((\pi_1^\star)^*(\xi_1 \oplus E_1)^\perp \otimes  (\pi_2^\star)^*(\xi_2 \oplus E_2)^\perp, \R)\\
& \cong & (\R \langle \theta_1 \oplus\theta_2 \rangle)^*.
\eeastar
Combining the last two isomorphisms, we have finished the proof.
\end{proof}

\subsection{The fiber functor $ (\mathfrak{Cont}, \star) \to (\mathfrak{Loc}, \starotimes)$}

Next we recall that the structure group of any real line bundle on a manifold 
can be reduced to $\Z_2$ which implies that any line bundle can be 
equipped with a flat connection. We denote by
$$
\mathfrak{Loc}(Q)
$$
the set of local systems, i.e., of line bundles on $Q$ equipped with 
flat connection $\nabla$. 

\begin{prop}\label{prop:local-systems} The
pair $(\mathfrak{Loc}, \starboxtimes)$ defines a monoid which is a 
submonoid of $(\text{\rm Vec}_1, \starboxtimes)$ consisting of line bundles. Then 
for each given contact manifold $(Q,\xi)$,  the external product $\starboxtimes$ induces
the internal tensor product $\starotimes$ on $\mathfrak{Loc}(Q)$ 
which equips the monoidal structure $(\mathfrak{Loc}(Q),\starotimes)$.
\end{prop}
\begin{proof} Let $(I_1, \nabla_1)$ and $(I_2, \nabla_2)$ be flat line bundles on $Q$ and consider the star tensor product
$I \starotimes I'$.  We consider the pull-back connection
$$
\nabla: = (\pi_1^\star)^*\nabla_1 \otimes (\pi_2^\star)^*\nabla_2
$$
on $Q$. It remains to show that $\nabla$ is also flat. But this immediately follows
from the functoriality of the curvature.  In particular the pull-back of any flat connection is
automatically flat, and that tensor product of flat connections is also flat.
\end{proof}

Recall that by definition (Definition \ref{defn:alpha-Omega}), any contact manifold $(Q,\xi)$ admits
a representative $(H,I,\alpha,\Omega)$ such that
 we have the short exact sequence
\beastar
0 \longrightarrow T^*Q \otimes I \longrightarrow  J^1 I \stackrel{\alpha}{ \longrightarrow} I \longrightarrow  0
\eeastar
and the isomorphism
$$
J^1I \cong (T^*Q\otimes I) \oplus \pi^*I
$$
as discussed in Section \ref{sec:loose}.  We denote by 
$$
\Pi:J^1I\to (\pi^*I) ^\perp = H(J^1Q)  \cong T^*Q \otimes I
$$
the associated  projection. Therefore, for a given $(Q,\xi,I)$ with $I = I^\xi$, 
we can define a linear map 
\bea\label{eq:connection}
\nabla^\xi:\Gamma(I)\to\Gamma(T^*Q\otimes I);\quad \nabla^\xi s=\Pi(j^1s).
\eea
Observe that
\beastar
\nabla^\xi(fs)
& = &
\Pi(j^1(fs))=\Pi(f_xs_x+df\cdot s+fj^1s)\\
& = & \Pi(df\cdot s)+\Pi(fj^1s)\\
& = & df\otimes s+ f\nabla^\xi s
\eeastar
so this indeed defines a connection.
\begin{lem} The connection $\nabla^\xi$ is flat.
\end{lem}
\begin{proof}
Let $U\subset Q$ be a trivializing open subset of $I = I^\xi$ 
so that $I^\xi|_U \cong U \times \R$. We denote by
the unit section $\mathbbm{1} \in \Gamma(U;I)$ associated to the constant function $1 \in \R$.

Then on $U$, any local section $s$ has the form $s(x) = (x, \widetilde s(x)) = \widetilde s(x) \mathbbm{1}|_x$
for $x \in U$, where $\widetilde s: U \to \R$ is a real-valued function. We compute
$$
\nabla^\xi s = \nabla^\xi (\widetilde s \cdot \mathbbm{1}) 
= d \widetilde s(x) + \widetilde s(x)  (\nabla^\xi \mathbbm{1})
$$
so $A=\nabla^\xi\mathbbm{1}=\Pi(j^1\mathbbm{1})$ on $U$ and evaluate the curvature
\beastar
F_{\nabla^\xi}|_x
& = & dA+A\wedge A|_x \\
& = & d(\Pi(j^1\mathbbm{1}))|_x+\Pi(j^1\mathbbm{1})\wedge\Pi(j^1\mathbbm{1})|_x\\
& = & \Pi\left(d(j^1(\mathbbm{1}-\mathbbm{1}_x))\right)\Big|_x\\
& = & 0
\eeastar
for all $x \in U$.
The above discussion shows the triple $(Q,\xi,I)$ carries a canonical choice of flat connection on $I.$
\end{proof}

Finally we have the following monoidal property of the `fiber functor'
$$
\mathfrak{Cont} \to \mathfrak{Loc}; \quad (Q,\xi) \mapsto (I^\xi, \nabla^\xi).
$$
\begin{rem} We refer readers to
\cite[Chapter 3]{deligne-milne} for a similar notion in relation to the study of
Tannakian category \emph{over the rigid tensor category}. We mention that our monoidal category
$(\mathfrak{Cont}, \starboxtimes)$
is \emph{non-unital} and so not a tensor category in the sense of \cite{deligne-milne}.
This is the reason why we put the quotation marks around the term \emph{fiber functor}.
\end{rem}

\begin{prop}\label{cor:starboxtimes-I} 
Let $I^\xi$ and $I^{\xi'}$ be the coorientation local systems of $(Q,\xi)$ and $(Q',\xi')$.
Then the assignment $(Q,\xi) \mapsto I^\xi$ defines a monoidal functor 
$$
\CI: \mathfrak{Cont} \to \mathfrak{Loc}, \quad \CI(Q,\xi): = (I^\xi,  \nabla^\xi).
$$
In particular, there is a canonical isomorphism
$$
I^{\xi\star \xi'} \cong I^\xi \starboxtimes I^{\xi'}
$$
that intertwines the pentagon diagrams of the two monoidal category.
(See the pentagon diagram  in Theorem \ref{thm:pentagon-identity}.)
\end{prop}
\begin{proof}
    On the external tensor product $I^\xi \starboxtimes I^{\xi'}$, we define the (flat) connection $\nabla$
    by
    $$
    \nabla = (\pi_1^\star)^*\nabla_1 \otimes (\pi_2^\star)^*\nabla_2
    $$
    for $I^{\xi_1} \starboxtimes I^{\xi_2} = (\pi_1^\star)^*I_1^{\xi_1} \otimes (\pi_2^\star)^*{\xi_2}$.
    In particular, the connection is flat.
    
It follows from   
\eqref{eq:connection} that we have
    $\nabla= \nabla^{\xi\star\xi'}$ by the defining properties of $\nabla^{\xi\star\xi'}$ from its definition
\eqref{eq:connection} applied to $\xi = \xi_1 \star \xi_2$.    
     This finishes the proof of the first statement.

For the second statement,  we start with the commutative diagram
$$
\xymatrix{
0 \ar[r]  & \xi_1 \starboxplus \xi_2  \ar[d] \ar[r] & 
TQ_1 \starboxplus TQ_2  \ar[d] \ar[r] & I^{\xi_1} \starboxplus I^{\xi_2}
\ar[d] \ar[r] & 0 \\
 0 \ar[r]  & \xi_1 \star \xi_2  \ar[r] & (T(SQ_1) \boxplus T(SQ_2))/\sim \ar[r]
 &I^{\xi_1 \star \xi_2} \ar[r]  & 0 &
}
$$
of vector bundles on $Q_1\star Q_2$: It follows from a diagram chasing in the above diagram
that the last downward arrow map is uniquely defined because 
the left two downward arrow maps are injective maps with 2 and 1 dimensional kernels respectively.
Then the map has its kernel given by $\ker\theta_{12}$
and so descends to an isomorphism 
\be\label{eq:canonial-identification}
\frac{I^{\xi_1} \starboxplus I^{\xi_2}}{\ker\theta_{12}} \cong I^{\xi \star \xi'}.
\ee
Then Lemma \ref{lem:iso-starboxtimes} finishes the proof.
(Compatibility with the pentagon axiom is an immediate consequence of 
the pentagon axiom that will be proved in Theorem \ref{thm:pentagon-identity}.)
\end{proof}

\section{Reeb vector field in the contact product}

Let $\widetilde Q_i$ be coorientation double covers of $(Q_i,\xi_i)$ for $i = 1, \, 2$.
We consider generalized contact forms $\alpha_i \in \Gamma(I^{\xi_i})$ thereof respectively.
The following explicit formula of the Reeb vector field for the contact product
is derived  in \cite{oh:shelukhin-conjecture} for the coorientable case.

In the present section, we generalize these formulae to non-coorientable cases.
We denote by $\mathfrak{C}(Q,I^\xi) := \Gamma(I^{\xi_i})$ the set of generalized contact forms of $(Q,\xi)$.

We now show that each choice of pair of sections $\zeta_i \in \Gamma(Q_i;I^\xi_i)$ 
canonically gives rise to one in $\Gamma(Q_1 \star Q_2; I^{\xi_1 \star \xi_2})$.
Let
$$
\pi_{12}:SQ_1\times SQ_2 \to Q_1\star Q_2
$$ 
be the natural projection and put $\xi_{12}: = \xi_1 \star \xi_2$. We put
$$
\Gamma(\pi_{12}) : = \Gamma(Q_1 \star Q_2; I^{\xi_1 \star \xi_2}).
$$

\begin{prop}\label{prop:product-form}
Let   $(Q_1,\xi_1)$ and $(Q_2,\xi_2)$ be a pair of contact manifolds. There is a natural map 
\beastar
\Gamma(\pi_{12}) \to \mathfrak{C}\left(Q_1\star Q_2; I^{\xi_1 \star \xi_2}\right) &;& 
\sigma \mapsto \CA_\sigma \\
\mathfrak C(Q_1;I^{\xi_1}) \times \mathfrak C(Q_2;I^{\xi_2}) \to \Gamma(\pi_{12})
&;& (\alpha_1, \alpha_2) \mapsto \sigma(\alpha_1,\alpha_2)
\eeastar
such that
\begin{enumerate}
\item  we have $\xi_1 \star \xi_2 = \ker \CA_{\sigma(\alpha_1, \alpha_2)}$,
\item The map $\Gamma(\pi_{12}) \to \mathfrak{C}(Q_1\star Q_2; I^{\xi_1 \star \xi_2})$
is equivariant under the actions on the domain and the codomain 
of $\CA_{\sigma(\cdot, \cdot)}$ by the
Liouville flows generated by  $\vec E: = \vec E_1 \oplus \vec E_2$
on $SQ_1 \times SQ_2$.
\end{enumerate}
\end{prop}
 \begin{proof}  Because $\pi_{12}:SQ_1\times SQ_2\to Q_1\star Q_2 = \P_+(SQ_1 \times SQ_2)$
  is an $\R_+$ bundle,  we consider the assignment  given by
   $$
\CA_\sigma =\sigma^*(\pi^*_1\theta_1+\pi^*_2\theta_2)
   $$
on $Q_1 \star Q_2$.  We now show that $\CA_\sigma$ is contact form with
$$
\ker \CA_\sigma = \xi_1 \star \xi_2.
$$
By definition,  
$$
d\sigma(\ker\CA_\sigma)\subset \ker(\pi^*_1\theta_1+\pi^*_2\theta_2)
$$
and thus
$$
d\pi_{12}[d\s(\ker\CA_{\s})]\subset \xi_1\star\xi_2.
$$
Therefore we have $\ker\CA_\sigma \subset \xi_{12} = \xi_1 \star \xi_2.$ On the other hand, we have $\dim \ker \CA_{\s}= \dim \, Q_1\star Q_2-1$ because the section $\s$ is an embedding 
and $\sigma^*(\pi^*_1\theta_1+\pi^*_2\theta_2)$ is a non-vanishing 1-form. 
This proves that $\CA_\sigma$ is a contact form of $\xi_1\star \xi_2$.

We consider the pull-back
$$
\widetilde{\a}_i=\pi^*\a_i
$$
for a given pair of generalized contact forms $\alpha_i \in \Gamma(T^*Q_i \otimes I^{\xi_i})$. 
Note that $\pi^*\alpha_i \in \Gamma(T^*\widetilde Q_i \otimes \pi_i^*\xi_i)$ is $\widetilde I_i$-valued
one form where we know $\widetilde I_i$ are trivial, and thus can be regarded a genuine real-valued
nowhere-vanishing one-form, and so a genuine contact form. 

Then via the identification \eqref{eq:psi1psi2-1}
$$
Q_1 \star Q_2 \cong \P_+(SQ_1) \times \P_+(SQ_2) \times \R,
$$
the assignment
\bea\label{eq:canonicalsection}
\P_+(SQ_1) \times \P_+(SQ_2) \times \R
& \to & 
SQ_1\times SQ_2\nonumber\\
(\widetilde q_1,\widetilde q_2,\eta)
& \mapsto & \left(e^{\frac{\eta}{2}}\pi_*\widetilde\a_1(\P_+(SQ_1)),e^{-\frac{\eta}{2}}\pi_*\widetilde\a_2(\widetilde q_2)\right)
\eea
defines a section of $\pi_{12}$, and we call it $\s(\a_1,\a_2).$
For Statement (2),  both (left) actions are induced by the flow of $\vec E$
on $SQ_1 \times SQ_2$ by the map $\phi \circ \sigma \to \sigma$ on $\Gamma(\pi_{12})$, and
$\CA \mapsto \phi_*\CA$ on $\mathfrak{C}(Q_1 \star Q_2)$ since $\CA_{\s(\a_1,\a_2)}$ is given by
$$\CA_{\sigma(\alpha_1,\alpha_2)}= e^{\frac{\eta}{2}}\pi^*_1\widetilde{\a}_1+e^{-\frac{\eta}{2}}\pi^*_2\widetilde{\a}_2
$$
and, from its formula, the action is just to multiply a (positive) constant at $\CA$, so the second map of the statement is equivariant.
 \end{proof}
 
\begin{rem}\label{rem:productcontactform} In fact, we can define one-parameter family of
`linear' family of  contact form (or section) $\{\CA_{\s(\a_1,\a_2)}^s\}_{s \in [0,1]}$,  
with respect to which we have $\CA = \CA^{1/2}$ in Proposition \ref{prop:product-form}.c
More explicitly 
we can express the $[0,1]$-family of the contact form $\CA_{\sigma(\alpha_1,\alpha_2)}^s$ on $Q_1 \star  Q_2$ as
$$
\CA_{\sigma(\alpha_1,\alpha_2)}^s =  e^{(1-s)\eta}\pi^*_1\widetilde{\a}_1+e^{-s \eta}\pi^*_2\widetilde{\a}_2, 
\quad s \in [0,1].
$$
\end{rem}

\begin{prop}\label{prop:generalized-Reeb} Let $s \in [0,1]$ be any point. Then Reeb vector field of the contact form 
$\CA_{\sigma(\alpha_1,\alpha_2)}^s$ is given by
$$
\frac12\left(e^{(1-s)\eta}\widetilde R_{\alpha_1}, e^{s \eta}\widetilde R_{\alpha_2}, 0\right)
$$
where $\widetilde R_{\alpha_i}$ is the lifting to $\widetilde Q_i$ of the
canonical contact vector field  associated to $\alpha_i$ in the sense of 
Definition \ref{lem-defn:contact}.
\end{prop}
\begin{proof} This directly follows from the definitions of 
generalized contact form $\alpha_i$ and its associated contact vector field $X_{\alpha_i}$
defined in \eqref{eq:defining-Rlambda}.
\end{proof}

Recall  that $Q_0 \star Q_1$ is a coorientable contact manifold. In particular, $I^{\xi_1 \star \xi_2}$ is 
a trivial line bundle on $Q_0 \star Q_1$
 and so we can safely omit $I^{\xi_1 \star \xi_2}$ from notation by writing simply
$$
\mathfrak{C}(Q_1\star Q_2,\xi_1\star \xi_2) := \mathfrak{C}\left(Q_1\star Q_2;I^{\xi_1\star \xi_2}\right).
$$
When $Q_1 \star Q_2$ is equipped with an orientation, 
we have the natural decomposition
$$
\mathfrak{C}(Q_1\star Q_2,\xi_1\star \xi_2) = \mathfrak{C}_+(Q_1\star Q_2,\xi_1\star \xi_2)
\bigcup \mathfrak{C}_-(Q_1\star Q_2,\xi_1\star \xi_2)
$$
depending on the sign of the Liouville volume form $\lambda \wedge (d\lambda)^n$ of
contact form $\lambda \in \mathfrak{C}(Q_1\star Q_2,\xi_1\star \xi_2)$.

\begin{defn}[$\star$-product of generalized contact forms]\label{defn:staralpha1alpha2}
We have a product map
$$
\mathfrak{C}(Q_1,\xi_1)\times \mathfrak{C}(Q_2,\xi) \to \mathfrak{C}(Q_1\star Q_2,\xi_1\star \xi_2)
$$
defined by $(\alpha_1,\alpha_2) \mapsto \sigma(\alpha_1,\alpha_2)$ as above. We call 
$\sigma(\alpha_1,\alpha_2)$ the $\star$ product of $\alpha_1,\, \alpha_2$ and denote it by
$$
\alpha_1 \star \alpha_2.
$$
\end{defn}

\section{Universal property of contact product}

Recall that in the category theory, the product $X_1 \times X_2$ of two objects 
$X_1, \, X_2$ is defined to be the object that satisfies the following universal property:
\be\label{eq:catergory-universality}
\xymatrix{& Y \ar@{.>}[d]^f \ar[dl]_{f_1} \ar[dr]^{f_2} \\
X_1 & X_1 \times X_2 \ar[l]^{\pi_1} \ar[r]_{\pi_2} & X_2.
}
\ee
In this section, we  establish this
universal property of contact immersions, and prove other basic functorial properties of contact product.

We start with the following generalization of the ampleness property of the set of contact
vector fields for the non-coorientable contact manifolds. This property for the coorientable case
 is well-known to the experts. (See \cite{oh-wang:CR-map2}, for example.)
 
  \begin{lem}\label{lem:amplenss}
  Let $x \in Q$ and  $v \in T_xQ$.
  Then there is a contact vector field $X$ such that $X(x) = v$.
   \end{lem}
        \begin{proof} Recall the short exact sequence
        $$
       0 \to \xi_x \to T_xM \to I^\xi_x \to 0.
       $$
       We take a global splitting of the sequence on a neighborhood $U_x$
       so that $TQ|_{U_x} = \xi|_{U_x} \oplus I^\xi|_{U_x}$ and so that
       we can express 
       $$
       v_x = v_x^\pi + (v_x)^\perp.
       $$
       Take any (local) vector field $X$ such that $X(x) = v_x$, and the 
       (locally defined) one-form $\lambda$ on $U_x$ defined by
 $$
\beta = X \intprod d\lambda
$$
and then multiply a  bump functions $\chi$ such that there exists another neighborhood 
of $x$, $V_x \subset \overline V_x \subset U_x$
            $$
            \supp \chi \subset U_x, \quad \chi \equiv 1\,  \text{\rm on } \, V_x.
            $$
\end{proof}
 
  \begin{cor}\label{cor:2point-ampleness} Let $(Q,\alpha)$ be any contact manifold.
  Let $x, \, y$ be a pair of distinct points of $Q$ and let 
  $v\in T_xQ$ and $w \in T_yQ$.
  Then there is a contact vector field whose values at $x$ and at $y$ are $v$ and $w$
  respectively.
 \end{cor}
        \begin{proof}
            Choose bump functions $\chi_x$ and $\chi_y$ such that 
            $$
            \supp \chi_x = U_x,\, \supp\chi_y= U_y
            $$
            where $U_x$ and $U_y$ are open neighborhoods of $x$ and $y$ respectively
            and $U_x \cap U_y = \emptyset$.Then consider any contact vector fields
            $X$ and $Y$ with $X(x)=v$, $Y(y)=w.$  By $\beta_X \in \Gamma(I^\xi)$ be the unique section
            associated to $X$, and $\beta_Y$ for $Y$. Then we consider the 
            section $\beta^\chi:=\chi_x\beta_X+\chi_y\beta_Y$ and its associated contact vector field $X_{\beta^\chi}$
             has the desired property.
        \end{proof}

For further exposition, the following formal definition is useful.
\begin{defn}[Equipped contact manifolds] We call a contact manifold $(M,\xi)$ \emph{equipped} 
when a contact form $\alpha$ is given, and call the triple $(M, \xi,\alpha)$ an
\emph{equipped contact manifold}. A family $\{\alpha_{\CA(z)} \mid z \in K\}$ encoded by a
piecewise linear continuous map  $\CA: K \to \mathfrak{C}(M,\xi)$ for an 
affine manifold $K$ is called  a \emph{flat family} of contact forms for $(M,\xi)$.
We call the family just a $K$-family when the map $\CA$ is understood.
\end{defn}

\subsection{Universality of the product of contact immersions}

Our definition of contact product  admits a variation of this universality in the following form.

\begin{prop}\label{prop:universality-product}
Let $(Q_1,\xi_1)$ and $(Q_2,\xi_2)$ be two contact manifolds, and $Q_1 \star Q_2$ be
their contact product. Let  $f_1: Y \to Q_1$ and $f_2: Y \to Q_2$ be 
contact immersions. Then there exists a 
contact immersion $f: Y \to Q_1 \star Q_2$ that fills the above
universality diagram \emph{under the contact immersions}.
We denote the resulting contact immersion by $f =: f_1 \star f_2$.
\end{prop}
\begin{proof}  By definition of contact immersion, 
we have a line bundle $I_i \to Y$  that is
 a morphism of short exact sequences   for each $i = 1, \, 2$. :
$$
\xymatrix{0 \ar[r] & \ker \alpha_i \ar[r]\ar[d]^{df_i} & TY\ar[r]^{\alpha_i} \ar[d]^{df_i}  & I_i  \ar[r]\ar[d]^{(f_i)_*} & 0 \\
0 \ar[r] & f_i^*\xi_i \ar[r] & f_i^*TQ_i \ar[r]^{\pi_i} & f_i^*(TQ_i /\xi_i) \ar[r] & 0
}
$$
on $Y$, and the induced map $(f_i)_*: I_i \to f_i^*(I^\xi)$ is a bundle isomorphism.

We consider the map $f_{12}:=f_1 \times f_2: Y \to Q_1 \times Q_2$ and its derivative 
\be\label{eq:df12}
df_{12}: TY\to f_1^*TQ_1\times f_2^*TQ_2.
\ee
We take the external direct sum of the line bundles $I_i$
$$
I_1 \boxplus I_2 = \pi_1^*I_1 \oplus \pi_2^*I_2
$$
on $Y \times Y$. Then we have bundle isomorphism
$$
(f_1)_* \boxplus (f_2)_*: I_1 \boxplus I_2 \to  f_1^*(TQ_1 /\xi_1) \boxplus  f_2^*(TQ_2 /\xi_2)
$$
given by
$$
(f_1)_* \boxplus (f_2)_*(s_1, s_2): = ((f_1)_*(s_1), (f_2)_*(s_2)) \in f_1^*(TQ_i /\xi) \boxplus  f_2^*(TQ_i /\xi).
$$
This  is a bundle isomorphism of plane bundles.

This then gives rise to the following commutative diagram over $Y$ by precomposing 
the diagonal embedding $\Delta_Y: Y \to Y \times Y$
\be\label{eq:CD1}
\xymatrix{0 \ar[r] & \ker \alpha \ar[r]\ar[d]^{df_1 \oplus df_2} & TY\ar[r]^{\alpha}
\ar[d]^{df_1 \oplus df_2}  & I \ar[r]\ar[d]^{(f_1)_*  \oplus (f_2)_*} & 0 \\
0 \ar[r] & f_1^*\xi_1 \oplus f_2^*\xi_2 \ar[r] & f_1^*TQ_1 \oplus f_2^*TQ_2  
\ar[r]^<<<<<{\pi_1\oplus \pi_2} & 
f_1^*I^{\xi_1}  \oplus f_2^*I^{\xi_2} \ar[r] & 0
}
\ee
such that all down arrows are monomorphisms.

We also have the commutative diagram of bundles over $Y$
\be\label{eq:CD2}
\xymatrix{
0 \ar[r] & f_1^*\xi_1 \oplus f_2^*\xi_2\ar[r] & f_1^*TQ_1 \oplus f_2^*TQ_2  
\ar[r]^<<<<<<{\pi_1\oplus \pi_2} & 
f_1^*I^{\xi_1}  \oplus f_2^*I^{\xi_2} \ar[r] & 0\\
0 \ar[r] & \widetilde{(f_1^*\xi_1)} \oplus \widetilde{(f_2^*\xi_2)} \ar[r]\ar[u]_{\widetilde{d\pi_1}\oplus 
\widetilde{d\pi_1}}  &
\widetilde{(f_1^*TQ_1)} \oplus \widetilde{(f_2^*TQ_2) }
\ar[r]^<<<<<{\widetilde{\pi_{12}}}\ar[r]\ar[u]_{\widetilde{d\pi_1}\oplus \widetilde{d\pi_2}}& 
\frac{\widetilde{(f_1^*TQ_1)} \oplus \widetilde{(f_2^*TQ_2)}}{\ker(\widetilde{\theta_1}\oplus \widetilde{\theta_2})}
 \ar[u] & 
}
\ee
where the `tilde' objects are liftings to $TSQ_i$ of those in $TQ_i$. 
Note that the lower sequence is not exact.

It is straightforward to check that  the composition 
$$
F_{12}: = \widetilde{\pi_{12}} \circ(( (f_1)_*) \oplus (f_2)_*): 
I \to \frac{\widetilde{(f_1^*TQ_1)} \oplus \widetilde{(f_2^*TQ_2)}}{\ker(\widetilde{\theta_1}\oplus \widetilde{\theta_2})}
$$
is a bundle isomorphism, and that the diagram
\be\label{eq:CD3}
\xymatrix{TY\ar[r]^{\alpha}   \ar@/_5pc/ @{.>}[dd]_{G_{12}} & I  \ar[d]^{F_{12}}
\\
\widetilde{(f_1^*TQ_1)} \oplus \widetilde{(f_2^*TQ_2)}
\ar[r]^<<<<<{\widetilde{\pi_{12}}} 
\ar[d]^{d\pi_{12}}  & 
\frac{\widetilde{(f_1^*TQ_1)} \oplus \widetilde{(f_2^*TQ_2)}}{\ker(\widetilde{\theta_1}\oplus 
\widetilde{\theta_2})} \ar[d]^{{\pi_{12}}_*} \\
(f_1^*TQ_1) \starboxplus (f_2^*TQ_2)
\ar[r]^<<<<<<<{[\widetilde{(\pi_{12})}]}&
I^{f_1^*\xi_1\star f_2^*\xi_2} 
}
\ee
commutes.
By definition, it follows that the far right downward composition $\pi_{12}{}_*\circ F_{12}$ is an isomorphism, and 
that we have
$$
d\pi_{12}\left((\widetilde{\pi_{12}})^{-1}\left(\ker (\widetilde{\theta_1}\oplus 
\widetilde{\theta_2}\right)\right) = \{0\}
$$
in the middle row of the diagram.

We also recall that we are given  the map 
$$
(\widetilde{df_1},\widetilde{df_2}): TY\to TSQ_1 \times TSQ_2
$$
from the standing hypothesis that $f_1$ and $f_2$ are contact immersions. 
\begin{lem} Consider the descendent map
$$
[(\widetilde{df_1},\widetilde{df_2})]: TY\to  T(Q_1 \star Q_2)
$$
induced by $(\widetilde{df_1},\widetilde{df_2}): TY\to TSQ_1 \times TSQ_2$. Then
the following diagram
$$
\xymatrix{0 \ar[r] & \xi \ar[r] \ar[d] & TY\ar[r]^\alpha \ar[d]^{[(\widetilde{df_1},\widetilde{df_2})]}& I \ar[r] \ar[d]^{\cong} & 0\\
0 \ar[r] & \xi_{Q_1 \star Q_2} \ar[r] & T(Q_1 \star Q_2) \ar[r]   & I^{\xi_{Q_1 \star Q_2}} \ar[r] & 0
}
$$
is the short exact sequences associated to the \emph{contact morphism} $f_1 \star f_2: Y \to Q_1 \star Q_2$ 
 induced by $f_1, \, f_2$.
\end{lem}
\begin{proof} It remains to prove existence of the last downward arrow map
that makes the diagram commute and is an isomorphism.
By considering the lifted bundle
$$
\widetilde{(f_1^*TQ_1)} \oplus \widetilde{(f_2^*TQ_2)}
$$
on $SQ_1 \times SQ_2$ and the
map \eqref{eq:df12} followed by the pushforward thereof along the canonical projection
$$
\pi_{12}: SQ_1 \times SQ_2 \to Q_1 \star Q_2,
$$
we obtain the map 
$$
G_{12}: TY\to (f_1^*TQ_1) \starboxplus (f_2^*TQ_2)
$$
above. This is a canonically induced  monomorphism.  
This  then induces the map 
$$
\widetilde{\pi_{12}} \circ G_{12}: TY\ \to   I^{f_1^*\xi_1}\starboxtimes I^{f_2^*\xi_2}
$$
which is equivalent to the map 
$$
[(\widetilde f_1, \widetilde f_2)_*]: T(Q_1 \star Q_2) \to I^{\xi_1}\starboxtimes I^{\xi_2}.
$$
Finally, we apply the isomorphism $ I^{\xi_1}\starboxtimes I^{\xi_2} \cong I^{\xi_1 \star \xi_2} = I^{\xi_{Q_1 \star Q_2}} $
from Lemma \ref{lem:iso-starboxtimes}, and complete construction of the diagram.
\end{proof}

Finally, we describe the to-be-defined the \emph{map} of contact immersion $f: Y \to Q_1 \star Q_2$.
Since $f_i: Y \to Q_i$ is a contact immersion, $df_i: TY\to TQ_i$
injects $\xi \to \xi_i$ and induces an isomorphism $I^\xi_i$. For each given $x \in Y$
and $[v_i] \in I^\xi$, we consider the pair
$$
(f_i(x), [v_i])
$$
where $(f_i(x),v_i) \in T_{f_i(x)}Q_i$. Then we define
$$
f(x): = [(\widetilde v_1,\widetilde v_2)]
$$
where $\widetilde v_i \in T_{\beta_i}SQ_i$ for a $\beta_i \in SQ_i$ with $\pi_i(\beta_i) = f_i(x)$.
It is easy to check that this is indeed well-defined map. To prove its smoothness, one can 
locally check it in a straightforward way by choosing local contact forms of $Q_i$ and
recalling that locally the contact product is given by $Q_1 \times Q_2 \times \R$ equipped 
with the contact form $\mathscr A$. We left the details of the proof to the readers.
This finishes the proof.
\end{proof}

We now study the natural product map
$$
(\pi_1^\star,\pi_2^\star): Q_1 \star Q_2 \to Q_1 \times Q_2.
$$
Recall $\dim Q_1 \star Q_2 = \dim Q_1 + \dim Q_2 + 1$.

The following submersivity of the map $(\pi_1^\star,\pi_2^\star)$ will be 
important in the generic calculus of Legendrian correspondences in Part 2.
Recall $Q_1 \star Q_2$ is coorientable, whether or not $Q_1$ or $Q_2$ is coorientable.

\begin{prop}\label{prop:bi-ampleness}
\begin{enumerate}
\item 
 The map $(\pi_1^\star,\pi_2^\star): Q_1\star Q_2 \to Q_1 \times Q_2$
is a submersion. In particular, the $(\pi_1^\star, \pi_2^\star): Q_1\star Q_2 \to Q_1 \times Q_2$ defines 
a locally trivial fibration that is an $\R_+$ principal bundle admitting a section. 
\item 
Furthermore for every equipment $\alpha$ of the contact structure $\xi_1 \star \xi_2$, 
the following definition of a section $\sigma = \sigma_\alpha$ is well-defined:
For each $(x_1,x_2) \in Q_1 \times Q_2$, we choose $\beta_i \in S_{x_i}Q_i$ for $i=1, \, 2$ so that
\be\label{eq:sigma-defining}
x_i = [\beta_i],\quad
(\beta_1,\beta_2) \in \Image \alpha
\ee
Then we put $\sigma(x_1,x_2) = [\beta_1,\beta_2] \in Q_1 \star Q_2$. 
\end{enumerate}
\end{prop}
    \begin{proof} Consider the  diagram
    $$
    \xymatrix{SQ_1 \times SQ_2 \ar[rr]^{\pi_{12}}\ar[dr]^{\pi_1\times \pi_2}
   && \P_+(SQ_1 \times SQ_2) \ar[dl]_{(\pi_1^\star, \pi_2^\star)} \\
    & Q_1 \times Q_2 
    }
    $$
    is commutative where all arrows are natural projections. 
 The bottom down left arrow $\pi_1 \times \pi_2$ defines a principal $\R_+ \times \R_+$ bundle while 
 the top arrow $\pi_{12}$ is a principal $\R_+$ bundle. 
     
By definition, 
 if we let $x \in Q_1 \star Q_2 =  \P_+(SQ_1 \times SQ_2)$ and put $x_i: = \pi_i^\star(x)$.
Then $(x_1,x_2) \in Q_1\times Q_2$ and we have 
$$
  \pi_{12}^{-1}(x) \subset (\pi_1 \times \pi_2)^{-1}(x_1,x_2).
$$
This also makes each fiber $(\pi_1^\star,\pi_2^\star)^{-1}(x_1,x_2)$ carries a natural free
$$
 \R_+^2/ \R_+\{(1,1)\} \cong \R_+
$$
action and hence $(\pi_1^\star,\pi_2^\star)$ becomes a principal $\R_+$ bundle.

Moreover, we also have       
\bea\label{eq:inclusion}
 T_{(\beta_1,\beta_2)}\left(\pi_{12}^{-1}(x)\right)& = &\R\{\vec E_{12}(\beta_1,\beta_2)\} 
 \subset \R\{\vec E_1(\beta_1)\}\oplus \R\{\vec E_2 (\beta_2)\} \nonumber\\
 &= & T_{(\beta_1,\beta_2)}\left((\pi_1\times \pi_2)^{-1}(x_1,x_2)\right)
 \eea
at all $(\beta_1,\beta_2) \in  \pi_{12}^{-1}(x)$.        
Since
    $Q_1 \star Q_2 = \P_+(SQ_1 \times SQ_2)$ is coorientable,  the projection
    $\pi_{12}:SQ_1 \times SQ_2 \to  \P_+(SQ_1 \times SQ_2)$ admits a 
    section, which we denote by $\alpha:\P_+(SQ_1 \times SQ_2) \to SQ_1 \times SQ_2$.
Combining the above with the obvious fact that  $\pi_1 \times \pi_2$ is a submersion,  
it follows that the composition
 $$
 (\pi_1^\star,\pi_2^\star) = (\pi_1 \times \pi_2) \circ \alpha
 $$   
is a submersion.

To construct a section of $(\pi_1^\star, \pi_2^\star)$, let $(x_1, x_2) \in Q_1 \times Q_2$.  
Take representatives $\beta_i \in S_{x_i} Q_i$ of $x_i$ respectively so that 
\be\label{eq:sigma-defining1}
x_i = [\beta_i], \quad i =1, \, 2.
\ee
By \eqref{eq:inclusion}, we may suitably the $R_+$ action suitably to $\beta_i$, $i = 1, \, 2$
respectively, we can make the pair $(\beta_1,\beta_2)$ also satisfy
\be\label{eq:sigma-defining2}
(\beta_1,\beta_2) \in \Image \alpha.
\ee
Then we define the map
$$
 \sigma: Q_1 \times Q_2 \to \P_+(SQ_1 \times SQ_2)
$$
by putting
$$
\sigma(x_1,x_2): = [\beta_1,\beta_2].
$$
Once the map $\sigma$ is well-defined, it satisfies $(\pi_1^\star,\pi_2^\star)\circ \sigma = id$ by construction and 
so is a section.
Therefore it remains to prove the map $\sigma$ is well-defined. Suppose $(\beta_1', \beta_2')$ be
another such pair associated to $(x_1,x_2)$ satisfying \eqref{eq:sigma-defining1}, \eqref{eq:sigma-defining2}.
Then \eqref{eq:sigma-defining1} and \eqref{eq:sigma-defining2}  imply
$$
(\beta_1',\beta_2') = (t_1\beta_1,t_2 \beta_2) \in \Image \alpha 
$$
and 
$$
(\beta_1,\beta_2)  \in \Image \alpha.
$$
Put $ (t_1\beta_1,t_2 \beta_2) = \alpha(y')$ and $(\beta_1,\beta_2)= \alpha(y)$.
The we have 
$$
(\pi_1^\star(y),\pi_2^\star(y)) = (\pi_1^\star(y'),\pi_2^\star(y')) = (x_1,x_2).
$$
Since the projection $\pi_1 \times \pi_2$ is invariant under the flow of $\vec E_{12} = \vec E_1 \oplus \vec E_2$,
the flow preserves the fiber thereof at any $(x_1,x_2) \in Q_1 \times Q_2$.
Since the section $\alpha$ is transverse to $\vec E_{12}$, we must have
$$
(t_1 \beta_1,t_2 \beta_2) = (\beta_1,\beta_2)
$$
which implies $t_1=1 = t_2$ recalling $SQ \subset T^*Q \setminus \{0\}$. This proves the map $\sigma$
is well-defined.   
 \end{proof}
 
The following notation will be useful for  the later discussion of Legendrian correspondences in Part 2.
\begin{notation}[$x_1 \star x_2$] Let $(Q_,\xi_1)$ and $(Q_2,\xi_2)$ as above and 
equip $Q_1 \star Q_2$ with a contact form $\alpha$ as above.
Consider the above defined 
distinguished section $\sigma_\alpha: Q_1 \times Q_2 \to Q_1 \star Q_2$. Then we write
\be\label{eq:xstarx'}
x_1 \star_\alpha x_2: = \sigma_\alpha(x_1,x_2)
\ee
for each $(x_1,x_2) \in Q_1 \times Q_2$. We just write $x_1 \star x_2$ omitting $\alpha$
from notation when the choice of $\alpha$ is implicitly given from the given context without mentioning.
\end{notation}
 
 \subsection{Contact fiber product}
 \label{subsec:contactfiberproduct}
 
 We now establish the associativity of the contact product which will be important in our
 study of composition of Legendrian correspondence in the next sections.  
 For this purpose,  we need prepare  the notion of `fiber product' for the \emph{contact product}.
 
 \begin{defn}[Contact diagonal] Let $(Q,\xi)$ be a contact manifold and consider the contact product
 $Q \star Q$. We define
 $$
 \Delta_{Q}^\star: = \P_+(\Delta_{SQ}) = \{ [(\beta,\beta)] \in Q \star Q \mid \beta \in SQ\}.
 $$
\end{defn}
This is naturally a Legendrian submanifold of $Q \star Q$. (With a suitable sign of (generalized) contact forms. See Lemma \ref{lem:graph}.) 
Later, we are going to define 'contact graph' $\G_\psi$, then the contact diagonal is actually the contact graph of the identity map
$$
\Delta^\star_{Q} = \G_\id =\{(\beta, (\id)_*(\beta)) \in SQ \times SQ \, \mid \, \beta \in SQ\}/\sim
$$

More generally, let $f_1: Q_1 \to Q$ and $f_2: Q_2 \to Q$ be two
contact immersions.  We would like to define the \emph{contact fiber product} of 
$f_1$ and $f_2$, which we will denote by
 $$
 Q_1{}_{f_1}\star_{f_2} Q_2
 $$
 whose definition is now in order.
 
 \begin{lem} Let $f: Q \to Q'$ be a contact immersion. Then we have
 the covariant bundle map denoted by
 $$
 Sf: SQ \to SQ'.
 $$
 \end{lem}
 \begin{proof} By the definition of contact immersion (Definition \ref{defn:contact-immersion}), we have
 $\rank \ker \alpha_x = \dim Q -1$ for all $x \in Q$, and have a commutative diagram
\be\label{eq:CD-f}
\xymatrix{0 \ar[r] & \ker \alpha \ar[r]\ar[d]^{df} & TQ \ar[r]^\alpha \ar[d]^{df}  & I \ar[r]\ar[d]^{f_*} & 0 \\
0 \ar[r] & f^*\xi' \ar[r] & f^*TQ' \ar[r]^\theta & f^*(I^\xi) \ar[r] & 0
}
\ee
where $df$ in the first two columns are monomorphisms and $f_*$ is an isomorphism.
By the injectivity of $d_xf: T_xQ \to T_{f(x)}Q'$,  we have the natural projection map 
$$
\pi_f^*: = (df^{-1})^* : T^*Q \to f^*T^*Q',
$$
that is the composition of two maps
$$
T^*Q \stackrel{(df^{-1})^*}{\longrightarrow} f^*TQ' \hookrightarrow T^*Q.
$$
Recalling $SQ \subset T^*Q \setminus \{0\}$, we define
$$
Sf(\beta): = \pi_f^*(\beta), \quad \beta \in SQ.
$$
It remains to show that $Sf(\beta) \in SQ'$. For this we need to verify $Sf(\beta) \neq 0$ and
$Sf(\beta)(\xi') = 0$.

Since $\alpha \neq 0$, there is $v\in T_{\pi(\beta)}Q$ such that $\alpha(v) \neq 0$.
Then we  check the value
$$
Sf(\beta)(d_xf (v)) = ((d_xf)^{-1})^*\beta(d_xf(v)) = \beta((d_x f)^{-1}(d_xf)(v)) = \beta(v) \neq 0.
$$
Therefore $Sf(\beta) \neq 0$.

On the other hand, the commutativity of the first square implies $d_xf(w) \in \xi_{f(x)}'$.
Now let $w' \in \xi'_{f(x)}$ so that $\theta_{f(x)}(w') = 0$.
Then we can similarly evaluate
$$
Sf(\beta)(w') = (((d_xf)^{-1})^*\beta)(w') = \beta((d_x f)^{-1}(w')) = 0:
$$
Here we use the inclusion relation
$$
(d_xf)^{-1}(\xi_{f(x)}') \subset \xi_x.
$$
arising from the above  the commutativity of the diagram \eqref{eq:CD-f} which implies
 the last vanishing irrespective of the representative of the preimage 
 $(d_x f)^{-1}(w') \subset T_xQ$.  This proves that the 
 covariant map $Sf: SQ \to SQ'$ is well-defined.
\end{proof}

This lemma enables us to define the aforementioned contact product as follows.
We start with considering the ordinary Cartesian fiber product 
\bea\label{eq:f1f2Sproduct}
 f_1^*SQ {}_{f_1}\times_{f_2} f_2^*SQ: = 
\{(\beta_1,\beta_2) \in f_1^*SQ \times_Q f_2^*SQ \mid \beta_1 = \beta_2 \}.
\eea 
for $\R_+$-equivariant maps  $F_1:SQ_1\to SQ$ and $F_2:SQ_2\to SQ$.

\begin{defn}[Contact fiber product]\label{defn:starfiberproduct} 
Let  let $f_1: Q_1 \to Q$ and $f_2: Q_2 \to Q$ be two contact immersions.  
Consider Let $F_1:SQ_1\to SQ$ and $F_2:SQ_2\to SQ$ be the induced smooth bundle maps
given by $F_i = Sf_i$. 

We define the \emph{contact fiber product}, denoted by $Q_1{}_{f_1}\star_{f_2} Q_2$, 
to be the subset of $Q_1 \star Q_2$ given by
 \bea
 &{}& Q_1{}_{f_1}\star_{f_2} Q_2 \nonumber \\
 &  : =  & {\P}_+\left( F_1^*SQ {}_{F_1}\times_{F_2} F_2^*SQ\right) \nonumber \\
 & = & 
 \{[\alpha,\beta] \in Q_1 \star Q_2 \mid (\alpha,\beta) \in SQ_1 \times SQ_2,\, 
 [(Sf_1(\alpha_1)), Sf_2(\alpha_2)] \in \Delta_{Q}^\star \}.\nonumber \\
 &{}& 
 \eea
\end{defn}

We will just denote by $Q_1 \star_Q Q_2$ by suppressing the contact immersions $f_i: Q_i \to Q$ 
from the notations, whenever there is no danger of confusion.
 
Recall the contact projection map
$$
\pi^\star_i:Q_1\star Q_2\to Q_i.
$$
\begin{defn}[Contact product of subsets]
Let $S_i \subset Q_i$, $i = 1,\, 2$ be a subset of $Q_i$ respectively. Then we define 
 the contact product (or contact fiber product) thereof to be
 $$
 S_1 \star S_2 : = (\pi_1^\star, \pi_2^\star)^{-1}(S_1, S_2).
 $$
 \end{defn}
 
Applying this definition to Lagrangian correspondences $R_{ij} \subset Q_i\star Q_j$'s, we define 
the contact product (or contact fiber product) of two correspondences which 
again becomes a Legendrian correspondence.
We first recall from Proposition \ref{prop:bi-ampleness} that the map
$$
(\pi^\star_{12},\pi^\star_{34}):(Q_1\star Q_2)\star(Q_3\star Q_4)\to (Q_1\star Q_2)\times (Q_3\star Q_4)
$$
is a well-defined submersion.

\begin{lem-defn}[Contact product of Legendrian correspondences]\label{lem:preimage}
    For Legendrians $R_{12}\subset Q_1\star Q_2$ and $R_{34}\subset Q_3\star Q_4$,
    $$
    R_{12}\star R_{34}=(\pi^\star_{12},\pi^\star_{34})^{-1}(R_{12},R_{34})
    $$
    holds. In particular, it is a Legendrian submanifold of $(Q_1\star Q_2)\star(Q_3\star Q_4).$
\end{lem-defn}
\begin{proof}
    By definition of the contact product of a subsets, (or contact fiber product), we have
    $$
    R_{12}\star R_{34}={R_{12}}_{c_*}\star_{c_*} R_{34}=\{[\a,\b]\in {\pi^\star_{12}}^{-1}(R_{12})\times  {\pi^\star_{34}}^{-1}(R_{34})\}/\sim
    $$
    where $c_*$ is a constant map to a point $*$. Since $(\pi^\star_{12},\pi^\star_{34})$ is a submersion
    with fiber isomorphic to $\R$, $R_{12}\star R_{34}$ is a smooth submanifold of dimension given by
\beastar
    \dim R_{12} +  \dim R_{34} + 1 & = & \frac12(\dim Q_1 + \dim Q_2) + \frac12(\dim Q_3 + \dim Q_4) + 1\\
   & = & \frac12(\dim Q_1 + \dim Q_2 + \dim Q_3 + \dim Q_4) + 1.
\eeastar
     The isotropic
    property immediately follows from
    $$
    TR_{ij}\subset[\ker\theta_{ij}]\implies T({\pi^\star_{ij}}^{-1}(R_{ij}))\subset\ker\theta_{ij}.
    $$
    Noting that
    $$
    \dim (Q_1\star Q_2)\star(Q_3\star Q_4) = \dim Q_1 +\dim Q_2 + \dim Q_3 + \dim Q_4 + 3.
    $$
 Combining the above, we have proved that  $ R_{12}\star R_{34}$ is a Legendrian submanifold
 of of $(Q_1\star Q_2)\star(Q_3\star Q_4)$.
\end{proof}

 \begin{lem}
    Let $f:Q_1\to Q_2$ be a contact immersion and $\a_i$ be a generalized contact form of $Q_i$. Then $f$ induces the bundle isomorphism
    $$
    F:I^{\xi_1}\to f^*I^{\xi_2}
    $$
    over $Q_1$ which satisfies
    $$
    F^*\a_2=\lambda \a_1
    $$
    for some never-vanishing smooth function $\lambda$ on $Q_1.$
\begin{proof}
The first statement is immediate from the definition of the contact immersion and $F=df_*:I^{\xi_1}\to f^*I^{\xi_2}$. 
    so we get a fiberwise linear map $F^*:\mathfrak{C}(Q_2;I^{\xi_2})\to \mathfrak{C}(Q_1;I^{\xi_1})$ given by
    $$
    (F^*\a_2)_x(v)=(df_*)^{-1}(\a_2{}_{f(x)}(df(v)))
    $$
    for $v\in T_xQ_1$. Then for each $ v\in \xi_1|_x, \, x\in Q_1$, we have
        $$
        df_x(v)\in\xi_2|_{f(x)}
        $$
        which is equivalent to
        $$
        df_x(v)\in \ker \a_2|_{f(x)}.
        $$
        The latter is also equivalent to
$$
        v\in\ker (F^*\a_2)_x.
$$
    By chasing the contact product diagram $df(\xi_1^\perp)\notin\xi_2$ thus we have
    $$\rank\a_1=\rank F^*\a_2$$
    and in particular $F^*\a_2$ is another generalized contact form defining the contact structure of $(Q_1,\a_1)$, and
    $F^*\a_2=\lambda\a_1$
    for some function $\lambda:Q_1\to\R_{>0}$(or $\R_{<0}).$
\end{proof}
 \end{lem}
 
\section{Associativity of contact product and its associator} 
\label{sec:associativity}

Denote by
$$
\Delta_{Q_2} \subset Q_2\times Q_2
$$
the diagonal subset of the product $Q_2 \times Q_2$. 
Then we introduce two sets
\bea\label{eq:DeltaQ2S}
\Delta_2^S=\Delta_{Q_2}^S &: = & \Delta_{SQ_2}, \nonumber \\
\Delta_2^\star=\Delta_{Q_2}^\star &: = & {\P}_+(\Delta_{Q_2}^S) \subset Q_2 \star Q_2
\eea

  \begin{thm}\label{thm:contact-equivalence}
  There exists a roof diagram of canonical contact equivalences 
 \be\label{eq:roof}
 \xymatrix{
& (Q_0 \star Q_1) \star_{Q_1} (Q_1 \star Q_2) \ar[dr]^{ p_{0,(1,2)}} \ar[dl]_{p_{(0,1),2}} \\
 (Q_0 \star Q_1) \star Q_2 \ar[rr]^{p_{0,(1,2)}\circ (p_{(0,1),2})^{-1}} && Q_0 \star (Q_1 \star Q_2)
 }
 \ee
 where the top two maps are contact diffeomorphisms.
 In particular, the $\star$ product is associative up to contactomorphisms.
 \end{thm}
 \begin{proof} We first consider the projection map
 $$
\Pi_{(0,1),2}: (SQ_0 \times SQ_1) \times_{Q_1} (SQ_1 \times SQ_2) \to (SQ_0 \times SQ_1) \times SQ_2
$$
given by
$$
\Pi_{(0,1),2}\big((x_0,x_1),(x'_1,x_2)\big) = (x_0,x_1, x_2)
$$
that satisfy $\pi_{Q_1}(x_1) = \pi_{Q_1}(x_2)$, and
$$
\Pi_{0,(1,2)}: (SQ_0 \times SQ_1) \times_{Q_1} (SQ_1 \times SQ_2) \to SQ_0 \times (SQ_1\times SQ_2)
$$
given by
$$
\Pi_{0,(1,2)}\big((x_0,x_1),(x'_1,x_2)\big) = (x_0,(x'_1, x_2))
$$
that satisfy $\pi_{Q_1}(x_1) = \pi_{Q_1}(x_2)$.

 We first show that these maps descend to the map 
 $$
 p_{(0,1),2}:  (Q_0 \star Q_1) \star_{Q_1} (Q_1 \star Q_2) \to (Q_0 \star Q_1) \star Q_2
 $$
and to 
 $$
 p_{0,(1,2)}:  (Q_0 \star Q_1) \star_{Q_1} (Q_1 \star Q_2) \to Q_0 \star (Q_1 \star Q_2)
 $$
 respectively. We rewrite
 \beastar
 (Q_0 \star Q_1) \star_{Q_1} (Q_1 \star Q_2) & = & \P_+(S(Q_0 \star Q_1) \times_{Q_1} S(Q_1 \star Q_2)) \\
 & = & \P_+(S(\P_+(SQ_0 \times SQ_1)) \times_{Q_1} S(\P_+(SQ_1 \times SQ_2))) \\
 & = & \P_+(SQ_0 \times (SQ_1 \times_{Q_1} SQ_1) \times SQ_2)\\
 & = & \P_+(SQ_0 \times SQ_1) \times_{Q_1} (SQ_1 \times SQ_2)).
 \eeastar
 Once this is obtained, it is obvious that we have natural maps both to
 \beastar
 \P_+(\Pi_{(0,1),2}): \P_+(SQ_0 \times (SQ_1 \times_{Q_1} SQ_1) \times SQ_2) 
 & \to &
 \P_+((SQ_0 \times SQ_1) \times SQ_2) \\
 & \to &   (Q_0 \star Q_1) \star Q_2,
 \eeastar
 and similarly 
 \beastar
\P_+(\Pi_{0,(1,2)}):  \P_+(SQ_0 \times SQ_1) \times_{Q_1} ( SQ_1 \times SQ_2) 
& \to &
 \P_+(SQ_0 \times (SQ_1 \times SQ_2))\\
 & \to & Q_0 \star (Q_1 \star Q_2).
 \eeastar
 By chasing the diagrams, it also follows that these maps are contact diffeomorphisms.
\end{proof}

We are now ready to define the canonical isomorphism $\alpha_{X,Y,Z}$ 
leading to the associativity mentioned in the introduction. 

\begin{defn}[Associator $\alpha_{Q_0,Q_1,Q_2}$]
We define the associator $\a_{Q_0,Q_1,Q_2}$ by
$$
\a_{Q_0,Q_1,Q_2}:= p_{0,(1,2)}\circ p_{(0,1),2}^{-1}:(Q_0\star Q_1)\star Q_2\to Q_0\star(Q_1\star Q_2)
$$
\end{defn}
We call  the collection
\be\label{eq:associator}
\alpha = \{\alpha_{X,Y,Z}\}_{X,Y, Z \in \mathfrak{Cont}}
\ee
the \emph{associator} of $\mathfrak{Cont}$. We now prove that each element $\alpha_{X,Y,Z}$
is a contactomorphism. For this purpose, we start with the following lemma.

\begin{lem}
    There is a canonical contact diffeomorphism, the swapping map,
    $$
    \tau:Q_1\star Q_2\to Q_2\star Q_1.
    $$
\end{lem}
\begin{proof}
    We consider the swapping' map
    $$
    T:SQ_1\times SQ_2\to SQ_2\times SQ_1.
    $$
    This is a symplectic map, which also intertwines the Liouville flows. Therefore it naturally descends to a map
    $$
    \t:Q_1\star Q_2\to Q_2\star Q_1.
    $$
    which maps the kernel of Liouville one-form of the domain to that of the target. 
    This shows $\tau$ is a contactomorphism.
\end{proof}

Now we also prove another key isomorphism towards the pentagonal axiom of the
(nonunital) monoidal category.

\begin{prop}\label{prop:sigma}
    There is a collection of  contactomorphisms
    $$
    \sigma = \{\sigma_{X,Y,Z, W}\}
    $$
    such that the map $\sigma_{X,Y,Z, W}$ leads to a contactomorphism 
        $$
\sigma_{(Q_1, Q_2, P_1, P_2)} :(Q_1\star P_1)\star (Q_2\star P_2)\cong (Q_1\star Q_2)\star (P_1\star P_2)
        $$
    for each quadruple $(X,Y,Z,W) = (Q_1, Q_2, P_1, P_2)$.
\end{prop}
\begin{proof}
We will find an explicit expression of $\s = \sigma_{X,Y,Z, W}$.
We lift each contactomorphism $\Psi:Q_1\to Q_2$ 
to the symplectomorphism $S\Psi:SQ_1\to SQ_2$ by restricting
the  symplectic diffeomorphism
$$
(d\Psi^{-1})^*: T^*Q_1 \to T^*Q_2
$$
which intertwines the Euler vector fields $\vec E_i$.
(This is the canonical symplectic diffeomorphism 
between symplectization induced by a contactomorphism. )
We define the map 
$$
\id_{Q_0}\star\Psi: Q_0 \star Q_1 \to Q_0 \star Q_2
$$
by 
$$
\id_{Q_0}\star\Psi=[\id_{SQ_0}\times S\Psi]:Q_0\star Q_1\to Q_0\star Q_2.
$$
Then we have the following sequence of contactomorphisms:
\beastar
(Q_1\star P_1)\star (Q_2\star P_2)& \xrightarrow{\a_{Q_1,P_1,(Q_2\star P_2)}} & Q_1\star \big(P_1\star(Q_2\star P_2)\big)\\
& \xrightarrow{\id_{Q_1}\star\a^{-1}_{P_1,Q_2,P_2}} & Q_1\star \big((P_1\star Q_2)\star P_2\big)\\
& \xrightarrow{\id_{Q_1}\star(\t_{P_1,Q_2}\star\id_{P_2})}& Q_1\star \big((Q_2\star P_1)\star P_2\big)\\
& \xrightarrow{\id_{Q_1}\star\a_{Q_2,P_1,P_2}
} & Q_1\star \big(Q_2\star (P_1\star P_2)\big)\\
& \xrightarrow{\a^{-1}_{Q_1,Q_2,P_1\star P_2}} & (Q_1\star Q_2)\star (P_1\star P_2).
\eeastar
This finishes the proof.
\end{proof}

We are now ready to prove that the collection $\alpha = \{\alpha_{X,Y,Z}\}$
satisfies the following pentagonal axiom, whose proof will be occupied by
the entirety of the next subsection.

\begin{thm}[Pentagon identity]\label{thm:pentagon-identity}
The following diagram is commutative.
    \begin{center}
     \begin{tikzpicture}
\node (P0) at (90:2.8cm) {$((Q_0\star Q_1)\star Q_2)\star Q_3$};
\node (P1) at (90+72:2.5cm) {$(Q_0\star (Q_1\star Q_2))\star Q_3$} ;
\node (P2) at (90+2*72:2.5cm) {$\mathllap{Q_0\star ((Q_1\star} Q_2)\star Q_3)$};
\node (P3) at (90+3*72:2.5cm) {$Q_0\star (Q_1\mathrlap{\star (Q_2\star Q_3))}$};
\node (P4) at (90+4*72:2.5cm) {$(Q_0\star Q_1)\star (Q_2\star Q_3)$};
\draw
(P0) edge[->,>=angle 90] node[left] {$\a_{Q_0,Q_1,Q_2}\star\id_{Q_3}$} (P1)
(P1) edge[->,>=angle 90] node[left] {$\a_{Q_0,Q_1\star Q_2,Q_3}$} (P2)
(P2) edge[->,>=angle 90] node[above] {$\id_{Q_0}\star \a_{Q_1,Q_2,Q_3}$} (P3)
(P4) edge[->,>=angle 90] node[right] {$\a_{Q_0,Q_1,Q_2\star Q_3}$} (P3)
(P0) edge[->,>=angle 90] node[right] {$\a_{Q_0\star Q_1,Q_2,Q_3}$} (P4);
\end{tikzpicture}    
    \end{center} 
\end{thm}

\section{Pentagon axiom of contact products}
\label{sec:proof}

Recalling the definition \eqref{eq:associator} of $\alpha = \{\alpha_{X,Y,Z}\}$,
we consider the triple fiber product
\bea\label{eq:X0123}
\X_{0123} = \P_+(SQ_0 \times SQ_1 \times SQ_2 \times SQ_3).
\eea
We will denote the latter by $Q_1 \star Q_2 \star Q_3 \star Q_4$.

We would like to emphasize that this expression as a smooth manifold only involves 
the standard Cartesian product and
the usual projection. As a smooth manifold, there is a natural diffeomorphism from $\X_{0123}$ to each of the
5 vertices. To describe each of the 5 maps as a specific contactomorphism, we fix (local) contact forms
$\sigma_i$ on $(Q_i,\xi_i)$.

\subsection{Stable rooted trees and Stasheff polytopes}

We associate to a word with suitable parentheses a \emph{stable rooted tree} whose
definition we recall now. We refer readers to many literature on the subject. Our
exposition follows that of \cite{fukaya-oh,fooo:book2}, \cite[Section 1.2]{oh:kias} 
to which we refer readers for detailed explanation relevant to the purpose of the present 
paper.

\begin{defn} \index{rooted tree} \index{rooted tree!stable}
A tree $T$ is called a \textbf{rooted tree} if one vertex of $T$  has been designated
the root,  in which case the edges have a natural orientation, towards the root. A rooted tree $T$ is called \textbf{stable} if $T$ does not contain any vertices of valence 2.
\end{defn}

We start with the definition of ribbon trees.

\begin{defn} A \textbf{ribbon tree} is a pair $(T,i)$ such that
\begin{enumerate}
\item $T$ is a tree.
\item $i: T \longrightarrow D^{2}$ is an embedding such that
$$
 i^{-1}(\del D^{2}) = V_{\text{\rm ext}}(T).
$$
\end{enumerate}
A ribbon tree is called stable if $T$ is stable.

We say two ribbon trees $(T,i)$ and $(T',i')$ are isomorphic if
there is a homeomorphism $\phi: T \to T'$ as a CW map such that the two maps
$i$ and $i' \circ \phi$ are isotopic to each other. We denote by $[T,i]$ the
isomorphism class of a ribbon tree $(T,i)$ and call  it the \emph{combinatorial type}
of a ribbon tree.
\end{defn}

\begin{defn} A \textbf{rooted ribbon tree} is a pair $((T,i), v_{0})$ consisting of
\begin{enumerate}
\item $(T,i)$ is a ribbon tree.
\item $v_{0} \in V_{\text{\rm ext}}(T)$
\end{enumerate}
with the orientation on $T$ is given by the rule that
\begin{enumerate}
\item (Ribbon structure)
the ordering of exterior vertices starting from $v_{0}$ counterclockwise in $D^{2}$,
\item $v_{0}$ is the unique incoming exterior vertex and all others are outgoing
\item there exists a unique outgoing edge at all interior vertices.
\end{enumerate}
We denote by $([T,i],v_0)$ its isomorphism class.
\end{defn}

There exists a unique orientation on a rooted tree satisfying the
rule mentioned above.

\begin{defn}[$G_{n+1}$]\label{defn:Gn+1}
We denote by $G_{n+1}$ the set of isomorphism classes $([T,i], v_{0})$
where $n$ is the number of letters and we call $n+1$ the \textbf{flag number} of $T$.
\end{defn}
Note that $G_{n+1}$ is the set of different combinatorial types of rooted ribbon tree.

\subsection{Stasheff pentagon $K_4$ and contact products}\label{sect:iterated}

For the moment, we
consider the set of coorientable contact manifolds and denote it by
$$
\mathfrak{Cont}^{\text{\rm co}} \subset \mathfrak{Cont}.
$$
\begin{defn}[Equipped contact manifolds] Let $(Q,\xi)$ be a coorientable contact manifold.
We call a pair $(Q,\lambda)$ with $\lambda \in \mathfrak{C}(Q,\xi)$ an \emph{equipped contact manifold
of $\xi$}, and call $\lambda$ an \emph{equipping} of $\xi$.
 In general, a pair $(Q,\lambda)$ with contact form $\lambda$ is an equipped contact manifold.

\end{defn}
For given $\sigma_i \in \mathfrak{C}(Q_i,\xi_i)$ for $i = 0, \, 1$ with $n=2$, we define
the $K_2$-family of contact forms (Compare with Remark \ref{rem:productcontactform}.)
\be\label{eq:CA2}
\CA_2(\{\text{pt}\}; \sigma_0,\sigma_1) =\CA^{\frac{1}{2}}= e^{\frac{1}{2}\eta} \pi_0^*\sigma_0 + e^{-\frac{1}{2}\eta} \pi_1^*\sigma_1
\ee
for $\{\text{pt}\} = K_2$ which gives rise to a map
$$
\CA_2: K_2 \times \mathfrak{C}(Q_0,\xi_0) \times \mathfrak{C}(Q_1,\xi_1) 
\mapsto \mathfrak{C}(Q_0 \star Q_1,\xi_0 \star \xi_1).
$$
(Compare the map given in Definition \ref{defn:staralpha1alpha2}.)

For the study of monoidality of the contact product, we extend this affine family to a
piecewise linear map on the Stasheff polytopes $K_n$ for $n = 2, \, 3, \cdots, 4$.
Recall that $K_2 = \{pt\}$, $K_3 = [0,1]$ and $K_4 = \text{\rm Stasheff pentagon}$ so that the
element of stable rooted trees in $G_5 = G_{4+1}$ are assigned to the vertices of $K_4$
as pictured in Figure 1:

\begin{figure}[htb]\label{fig:K4}
  \centering
 \def\svgwidth{150pt}
  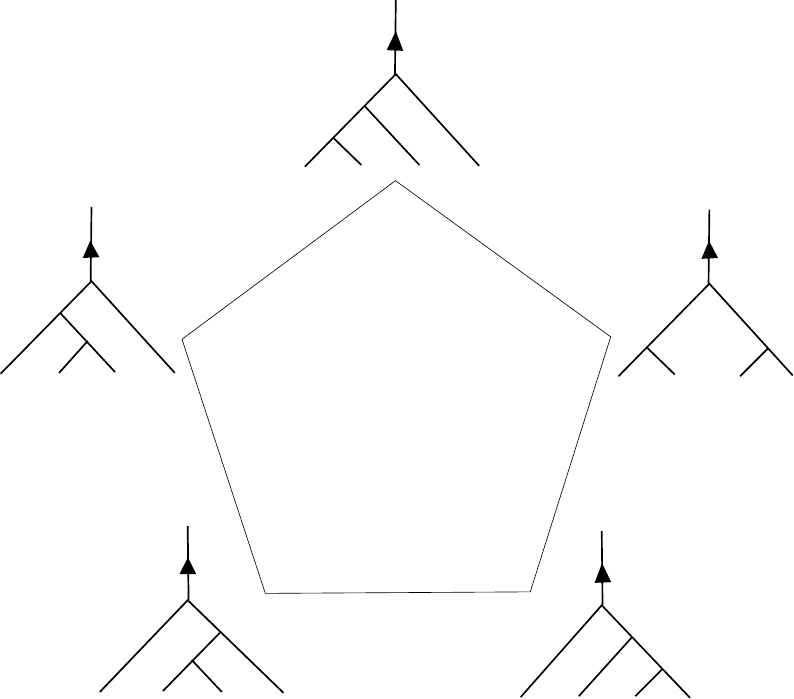
  \caption{Tree-decorated $K_{4}$}
\end{figure}

For general $K_n$, we can define a $K_n$-family 
\be\label{eq:CAn}
\CA_n: K_n \times \Pi_{i=0}^{n-1}\mathfrak{C}(Q_i) 
\longrightarrow \mathfrak{C}(Q_0 \star Q_1 \star Q_2 \star \cdots \star Q_{n-1}).
\ee
of contact forms on $Q_0 \star \cdots \star Q_{n-1}$
that satisfies the $A_n$-type relation.
(We refer readers to \cite{stasheff:polytope1,stasheff:polytope2} for the original definition of the $A_\infty$
relation, and for many others, e.g.,  \cite[Chapter 1]{oh:kias} for
a summary of its construction.)

\begin{defn}[Contact product of contact forms] Consider two contact manifolds
$(Q,\xi)$ and $(Q',\xi')$. Let $\sigma, \, \sigma'$ be contact forms thereof respectively
and denote by
$$
\Phi_{\sigma\sigma'}: (Q\star Q', \xi\star \xi') \to (Q \times Q' \times \R,\Xi_{\CA_2(\frac12;\sigma,\sigma')})
$$
the $\Z_2$-equivariant contactomorphism. We call the contact form $\CA_2(\{\text{pt}\};\sigma,\sigma')$
the product of $\sigma$ and $\sigma'$ and denote it by
$$
\sigma \star \sigma'.
$$
More generally, we define
\be\label{eq:star-s}
\sigma \star_s \sigma': =  e^{(1-s)\eta} \pi_0^*\sigma + e^{-s\eta} \pi_1^*\sigma', \quad s \in [0,1].
\ee
\end{defn}
Note that by definition, we have $\sigma \star \sigma' = \sigma \star_{\frac12} \sigma'$. An implication of 
the $\Z_2$ symmetry enables us to define the contact product iteratively in the following way,
which is obvious by definition.

\begin{lem}  The choice of $\sigma, \, \sigma'$ above gives rise to a canonical
strict contactomorphism
$$
(\CA_2)_*:  (Q \star Q',  \sigma \star \sigma')
\to (Q \times Q' \times \R, \CA_{\sigma\sigma'})
$$
\end{lem}

Now let us consider 3 equipped contact manifolds $(Q_i,\xi_i,\sigma_i)$ with $i = 0, \, 1\, \, 2$.
We will define the iterated contact product in the way similar to the way 
how we define the `iterated concatenation' of based loops starting from the 
definition of the concatenation $\gamma_1 * \gamma_2$,
$$
(\gamma_1* \gamma_2)(t) = \begin{cases} \gamma_1(2t) \quad & t \in [0,\frac12] \\
\gamma_2(2t-1) \quad & t \in [\frac12,1]:
\end{cases}
$$
\emph{It partitions the unit interval $[0,1]$ to two half-intervals, and do this equal partitioning
in each iteration, which makes the size of the $k$th-iteration $\frac{1}{2^{k-1}}$.}

For each given pair $(Q_1,\a_1)$ and $(Q_2,\a_2)$ of equipped contact manifolds, 
we have the coordinate map (see Definition \ref{defn:eta})
$$\eta_{12}:Q_1\star Q_2\to\R.$$
Then the triple $(\alpha_1,\alpha_2,\eta_{12})$ defines a section
$$
\sigma_{12}: Q_1 \star Q_2 \to SQ_1 \times SQ_2 \subset T^*Q_1 \times T^*Q_2 \cong T^*(Q_1 \times Q_2)
$$
by
$$
\sigma_{12}(x): = \left(e^{\frac12\eta_{12}(x)} \alpha_1(\pi_1^\star(x)), e^{-\frac12\eta_{12}(x)} \alpha_2(\pi_2^\star(x))\right).
$$
Then it follows  from the ``defining property of the Liouville one-form" (see \cite[Proposition 4.3.2]{oh:kias},
for example) on general cotangent bundle $T^*M$,
we have
$$
\a_1\star_{\frac{1}{2}}\a_2 =
e^{\frac{1}{2}\eta_{12}}\pi_1^*\a_1+e^{-\frac{1}{2}\eta_{12}}\pi_2^*\a_2 =\s_{12}^*(\theta_1\oplus\theta_2).
$$
We extend the above observation to the iterated contact product as follows.

\begin{cordefn}[Iterated contact product of contact forms] We denote by
$$
\eta_{(01)2}:(Q_0\star Q_1)\star Q_2\to\R
$$
the $\R$ coordinate function of $(Q_0\star Q_1)\star Q_2$,
and a section
\bea\label{eq:contactsection}
& {} &\s_{(01)2}:(Q_0\star Q_1)\star Q_2\to (SQ_0\times SQ_1)\times SQ_2;\nonumber\\
& {} &\s_{(01)2}(Z) =\left\{(e^{\frac{1}{2}\nu}\a_0\star_{\frac{1}{2}}\a_1(\pi_{01}^\star(Z)),e^{-\frac{1}{2}\nu}\a_2(\pi_2^\star(Z))\right\}
\eea
for each $Z \in (Q_0\star Q_1)\star Q_2$.
Now we define the (iterated) contact product of contact form $(\a_0\star\a_1)\star \a_2$ by
\bea\label{eq:iteratedcontactform}
(\a_0\star\a_1)\star\a_2 & = &
e^{\frac{1}{2}\eta_{(01)2}}\pi_{01}^*(\a_0\star_{\frac{1}{2}}\a_1)+e^{-\frac{1}{2}\eta_{(01)2}}\pi_2^*\a_2\nonumber\\
& = & e^{\frac{1}{2}\eta_{(01)2}+\frac{1}{2}\eta_{01}}\pi_0^*\a_0+e^{\frac{1}{2}\eta_{(01)2}
-\frac{1}{2}\eta_{01}}\pi_1^*\a_1+e^{-\frac{1}{2}\eta_{(01)2}}\pi_2^*\a_2.
\eea
We define $\alpha_0 \star (\alpha_1 \star \alpha_2)$ similarly by changing the order of products.
\end{cordefn}
In particular, we have
\be\label{eq:alpha(01)2}
(\a_0\star\a_1)\star\a_2=\s_{(01)2}^*(\theta_0\oplus\theta_1\oplus\theta_2)
\ee
again by the defining property of the Liouville one-form.

\begin{rem}
    The number of $\frac{1}{2}$'s in exponents of \eqref{eq:iteratedcontactform} exhibits `how many times the component is iterated'. The same situation occurs in the case of the iterated concatenations of based loops. The (time) length of loops are given by $\left(\frac{1}{2}\right)^k$, where $k$ represents the number of iteration.
\end{rem}
The formula \eqref{eq:contactsection} shows that the coordinate maps $\nu=\eta_{(01)2}$ and $\eta=\eta_{01}$ appear just as the 
scaling factors in symplectizations. In this regard, the contact structure  
$$
\big((Q_0\star Q_1)\star Q_2,(\xi_0 \star \xi_1)\star \xi_2\big)
$$
carries the canonical equipping \eqref{eq:alpha(01)2}
induced by the section $\sigma_{(01)2}$ as a map into $T^*(Q_0 \times Q_1 \times Q_2)$. Recall that by definition
its image is contained in $SQ_0\times SQ_1\times SQ_2$ given by
$$
(Q_0\star Q_1)\star Q_2=\left\{(e^{\frac{1}{2}\nu+\frac{1}{2}\eta}\a_0,e^{\frac{1}{2}\nu-\frac{1}{2}\eta}\a_1,e^{-\frac{1}{2}\nu}\a_2)\right\}\subset SQ_0\times SQ_1\times SQ_2.
$$
Furthermore, we have the linear homotopy $t \mapsto H_t$ where the map 
\be\label{eq:Ht}
H_t: Q_0 \times Q_1 \times Q_2 \times \R^2 \to SQ_0\times SQ_1\times SQ_2
\ee
 is defined by 
$$
H_t((q_0,q_1,q_2), \nu, \eta) = H(t, (q_0,q_1,q_2),\nu,\eta)
$$
connecting $H_0 = \sigma_{(01)2}$ and $H_1 = \sigma_{0(12)}$,
where we define
\bea\label{eq:homotopy}
H(t,\nu,\eta,(q_0,q_1,q_2)):=\left(e^{\frac{1}{2}\nu(1-t)+\frac{1}{2}\eta}\a_0(q_0),e^{\frac{1}{2}\nu-\frac{1}{2}\eta}\a_1(q_1),e^{-\frac{1}{2}\nu-\frac{1}{2}\eta t}\a_2(q_2)\right).
\eea
This homotopy defines $[0,1]$-family of subsets 
$$
\CS_t : = \Image H_t \subset SQ_0\times SQ_1\times SQ_2
$$
 with
\beastar
\CS_0 &= & \Image H_0 =  \sigma_{(01)2} (Q_0\star Q_1)\star Q_2), \\
\CS_1& = &  \Image H_1 = \sigma_{0(12)}(Q_0\star( Q_1\star Q_2)).
\eeastar
The following lemma is apparent by noting that each $H_t$ is transverse to the
Liouville vector field $E_0 \oplus E_1 \oplus E_2$ of $T^*(Q_0 \times Q_1 \times Q_2) \cong T^*Q_0 \times T^*Q_1 \times T^*Q_2$. 

\begin{lem}\label{lem:fontactomorphic_family} The family 
$\CS_t$ is a smooth family of contact type-hypersurface of 
the Liouville manifold $SQ_0 \times SQ_1 \times SQ_2$ equipped with
the Liouville structure $\theta_0 \oplus \theta_1 \oplus \theta_2$.
In particular they are contactomorphic with one another.
\end{lem}

In Section \ref{sec:associativity}, we define the associator map $\a_{Q_0,Q_1,Q_2}$
$$
\a_{Q_0,Q_1,Q_2}:(Q_0\star Q_1)\star Q_2\to Q_0\star(Q_1\star Q_2)
$$
using Theorem \ref{thm:contact-equivalence}. In fact the projection 
$$
\pi_\Delta: SQ_0 \times SQ_1 \times SQ_2 \to SQ_0 \times SQ_1 \times SQ_2/\sim
\cong Q_0 \times Q_1 \times Q_3 \times (\R_+^3/\sim)
$$
restricts to $\CS_t$ a contactomorphism $\pi_t: \CS_t \to Q_0 \times Q_1 \times Q_3 \times (\R_+^3/\sim)$
with $\pi_t = \pi_\Delta|_{\CS_t}$.
In particular, we have the correspondence diagram
\be\label{eq:pull-push}
\xymatrix{ &SQ_0 \times SQ_1 \times SQ_2 \ar[dr]^{\pi_1} &  \\
 (Q_0\star Q_1) \star Q_2   \ar[ur]^{\sigma_{(01)2}} &{} & Q_0\star (Q_1 \star Q_2) 
 }
\ee
such that we can factorize $\alpha_{Q_0,Q_1,Q_2}$ as the pull-push 
$$
\alpha_{Q_0,Q_1,Q_2} = \pi_1 \circ \sigma_{(01)2}
$$
associated to Diagram \eqref{eq:pull-push}.

Now using the above geometric realization of the contact forms
via the embedding into 
$$
SQ_0 \times SQ_1 \times SQ_2 \cong Q_0 \times Q_1 \times Q_2 \times \R_+^3
\stackrel{\text{\rm Log}}{\cong} Q_0 \times Q_1 \times Q_2 \times \R^3,
$$
we can express $\mathcal{S}_t$ for $t \in [0,1]$  as the image of 
$$
H_t = \id_{Q_0 \times Q_1 \times Q_2} \times \widetilde \sigma_t
$$
with $\widetilde \sigma_t: \R^2 \to \R^3$
respectively, where each $H_t$ is the map given in \eqref{eq:Ht} which induces a bijective correspondence 
with 
$$
Q_0 \times Q_1 \times Q_2 \times \R^3/\sim.
$$

In fact we have the following coincidence result of their images.
\begin{lem}\label{lem:samesubset}
    $\CS_0=\CS_1$ as subsets in $SQ_1\times SQ_2\times SQ_3$.
\end{lem}
\begin{proof} Under the  above description
    $$
    SQ_0\times SQ_1\times SQ_1\cong Q_0\times Q_1\times Q_2\times \R^3,
    $$
both   $\CS_0$ and $\CS_1$  are contact-type hypersurfaces.
By definition, it also follows that the images of $\widetilde \sigma_{(01)2}$ 
and $\widetilde \sigma_{0(12)}$ contained in the hyperplane $\{x+y+z = 0\} \subset \R^3$.    
This finishes the proof.
\end{proof}
This lemma also proves that
the maps $\widetilde \sigma_0$ 
and $\widetilde \sigma_1$ descends to linear isomorphisms 
$$
\R^3/\sim\,\,  \stackrel{\cong}{\longrightarrow} \{x+y+z = 0\}.
$$
The following proposition shows that the linear character of the associator $\alpha_{Q_0,Q_1,Q_2}$
which also shows the similarity of the pentagon axiom of the associator associated to
the concatenation  operation of based loops.

\begin{prop} \label{lem:SSt} Let $\alpha_i$ be a given equipping of the contact manifold $Q_i$ respectively.
Regard the contact manifolds $(Q_0\star Q_1)\star Q_2$ and $Q_0\star( Q_1\star Q_2)$ 
with the induced equippings as 
the respective contact-type hypersurfaces of $SQ_0\times SQ_1\times SQ_2$.  Then the associator
$$
\a_{Q_0,Q_1,Q_2}:(Q_0\star Q_1)\star Q_2\to Q_0\star(Q_1\star Q_2)
$$
can be identified with the map  
$$
\sigma_{0(12)} \circ \sigma_{(01)2}^{-1}: \CS \to \CS
$$
where 
$$
\CS =  (Q_0 \times Q_1 \times Q_2) \times  \{(x,y,z) \in \R^3 \mid x+y + z = 0\} \subset \R^3.
$$
Furthermore the $\R^3$-component of the map coincides with a linear map
$$
A = [\widetilde \sigma_1]\circ  [\widetilde \sigma_0]^{-1}: \R^2 \to \R^2
$$
\end{prop}
\begin{proof} For this purpose, it is enough to express the map $A$ in coordinates of 
$\R^3/\sim$. For this purpose, we utilize the coordinate transformation $\R^3 \to \R^3$
such that $(x,y,z) \mapsto (X,Y,Z)$ given by
$$
Z = \frac{x+y+z}{3},\, X = y-x, \, Y = z-y
$$
and make the identification $(X,Y): \R^3/\sim \to  \R^2$. 

We first express $\CS_0 = \Image \sigma_{(01)2}$ and $\CS_1 = \Image \sigma_{0(12)}$ as 
hypersurfaces of $\R^3$ respectively. By definition, we have
$$
    \left(e^{\frac{1}{2}\nu}\a_0\star_{\frac{1}{2}}\a_1,e^{-\frac{1}{2}\nu}\a_2\right)
    = \left(e^{\frac{1}{2}\nu+\frac{1}{2}\eta}\a_0,e^{\frac{1}{2}\nu-\frac{1}{2}\eta}\a_1,e^{-\frac{1}{2}\nu}\a_2\right)
$$
for $(\nu,\eta) \in \R^2$, and
$$
    \left(e^{\frac{1}{2}\nu'}\a_0,e^{-\frac{1}{2}\nu'} \a_1 \star\a_2\right)
    = \left(e^{\frac{1}{2}\nu'}\a_0,e^{- \frac{1}{2}\nu'+ \frac{1}{2}\eta'}\a_1,e^{- \frac{1}{2}\nu' - \frac12 \eta'}\a_2\right)
$$   
for $(\nu',\eta') \in \R^2$.  Observe that both images are contained in $\{Z = 0\}$, while we have
\beastar
X(\sigma_{(01)2}) = - \eta, \quad, Y(\sigma_{(0(12)}) = -\nu + \frac12\eta \\
X(\sigma_{0(12)}) = -\nu' + \frac12\eta', \quad Y(\sigma_{0(12)}) = -\eta'.
\eeastar
By setting $X(\sigma_{(01)2}) = X(\sigma_{0(12)})$ and $Y(\sigma_{(01)2}) = Y(\sigma_{0(12)})$,
we obtain the relation
$$
- \eta =  -\nu' + \frac12\eta', \quad  -\nu + \frac12\eta =  -\eta'.
$$
This equation can be written as the matrix equation
defined by the invertible linear transformation $A$
\begin{equation*}
    \begin{pmatrix}
    \nu'\\ 
    \eta'
    \end{pmatrix}
    =
    \begin{pmatrix}
       \frac12 & \frac34 \\
        1 & - \frac12
    \end{pmatrix}
    \begin{pmatrix}
        \nu\\
        \eta
    \end{pmatrix}
    =A
    \begin{pmatrix}
        \nu\\
        \eta
    \end{pmatrix}.
\end{equation*}
Thus the time-1 map $H_1$ is induced by the linear transformation 
$A\in \text{SL}^\pm(2,\R)$.  By replacing $\CS_1$ by $\CS_t$, we obtain a 
linear isotropy of linear maps $A_t$, $t \in [0,1]$ between $\id_{\R^2}$ and $A = A_1$.

This finishes the proof.
\end{proof}

The homotopy $H$ gives us a $K_3$-family. In the next subsection, we extend this procedure to $K_4$-family to prove the \emph{pentagon axiom}.  In fact, the existence of such a family is already implicit in the
course of the proof of Lemma \ref{lem:SSt}. We will give a more conceptual proof in the next subsection.

\subsection{Proof of pentagon axiom}\label{section:proofofpentagon}

Similarly to the case of concatenations of based loops, we can extend the 5 quadruple contact products of
$(Q_i,\xi_i)$ for $0 \leq i \leq 3$ encoded by $G_5$ to a flat $K_4$-family of products.

We first recall the \emph{grafting operation} which produces a collection of grafting maps 
$$
o_{i}: G_{k+1} \times G_{l+1} \longrightarrow G_{k+l}
$$
with $ 1 \leq i \leq k$ by iterating tree grafting procedures. (See Figure \ref{fig:grafting}.)

\smallskip
\smallskip
\begin{figure}[htb]
\centering
  \def\svgwidth{200pt}
  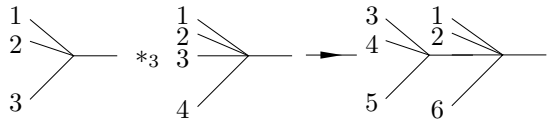
  \caption{Grafting}
  \label{fig:grafting}
  \end{figure}
Inductively, we obtain
\beastar \index{$o_{i}: K_{k} \times K_{l} \longrightarrow K_{k+l-1}$}
 K_{n} &= &\text{ cone over } \del K_{n} \\
 \partial K_n  & = & \bigcup_{i,s} o_{i} (K_{n-s+1} \times K_{s} ).
\eeastar

Then the pentagon axiom is a special case of the following more general $A_m$-relation for $m = 4$.

\begin{prop}[$A_m$ relation] Let $(Q_i,\xi_i)$ for $1 \leq i \leq m$ be contact manifolds equipped with contact
forms $\sigma_i$ respectively. Let $\CA_3$ be the $[0,1]$ flat family given by
$$
\CA_3(s;\sigma_1,\sigma_2,\sigma_3) =\iota_s^*(\theta_1\oplus\theta_2\oplus\theta_3),\quad \iota_s:Q_1\star Q_2\star Q_3\to \CS_s. 
$$
Then there exists a unique flat $K_m$-family  
$$
\CA_m : K_m  \to \mathfrak{C}(Q_1 \star Q_2 \star \cdots \star Q_m, \xi_1 \star \xi_2 \star \cdots \star \xi_m)
$$
of contact forms promoting $\CA_3$ so that
\begin{enumerate}
\item 
The family induces the boundary family  
$$
\del \CA_m^{1\cdots m}(\cdot)  = \CA_m^{1\cdots m}(\cdot)\Big|_{\del K_m}
$$
in the way that each facet thereof induces a flat family.
\item The map satisfies the $A_m$ relation for all $m$ in that
$$
\CA_{m} : K_m \times  \Pi_{i=0}^m \mathfrak{C}(Q_i,\xi_i) \longrightarrow \mathfrak{C}(Q_1 \star \cdots
\star Q_m, \xi_1 \star \xi_2 \star \cdots \star \xi_m)
$$
for $2 \leq k \leq n$ which are compatible in the following sense: they satisfy
$$
 \CA_m(\alpha *_i \beta; x_1,\cdots , x_m)=\CA_{m-s+1}(\alpha;x_1,\cdots,x_{i-1},\CA_s(\beta;x_i,\cdots,x_{i+s-1}),x_{i+s},\cdots,x_m)
$$
where  $\alpha\in G_{m-s+2}$ and  $\beta \in G_{s+1}$ and $\alpha*_i \beta: = o_i(\alpha,\beta)$.
\end{enumerate}
\end{prop}
\begin{proof} Once we identify the above correct statements, the proof of the proposition 
is the same as that of the standard procedure of establishing the $A_m$-relation, e.g.,
as in \cite{stasheff:polytope1,stasheff:polytope2} or in \cite[Section 1.5]{oh:kias} and so omitted.
\end{proof}

The specialization of the proposition assigns the following equipped
contact structures with the common smooth manifold
$$
\P_+(SQ_0 \times SQ_1 \times SQ_2 \times SQ_3)
$$
\beastar
&{}& (((Q_0 \star Q_1) \star Q_2)\star Q_3,  ((\sigma_0 \star \sigma_1) \star \sigma_2)\star \sigma_3), \\
&{}& ((Q_0 \star (Q_1 \star Q_2)) \star Q_3,(\sigma_0 \star (\sigma_1 \star \sigma_2)) \star \sigma_3),\\
&{}& (Q_0 \star ((Q_1\star Q_2)\star Q_3), \sigma_0 \star ((\sigma_1\star \sigma_2)\star \sigma_3)),\\
&{}& ( Q_0 \star (Q_1 \star (Q_2 \star Q_3)), \sigma_0 \star (\sigma_1 \star (\sigma_2 \star \sigma_3))),\\
&{}& ((Q_0 \star Q_1)\star (Q_2 \star Q_3), (\sigma_0 \star \sigma_1)\star (\sigma_2 \star \sigma_3))
\eeastar

\begin{proof}[Proof of the pentagon axiom] Let us first assume that $(Q_i,\xi)$ are all
coorientable, and fix contact forms $\sigma_i$ for each $0 \leq i \leq 3$.

We then consider the flat $K_4$-family of
$$
\sigma = \{\sigma_T \mid T \in K_4\}
$$
the value of which at the vertices of $K_4$ are given as above.  In other words, we are given a
flat family $E \to K_4$ of equipped contact manifolds over $K_4$ whose fiber is given by
$$
(Q_z, \xi_z, \sigma_z), \quad z \in K_4.
$$
We put the contact form
	\bea\label{eq:4contactform}
& {} &Q_0\star Q_1\star Q_2\star Q_3\nonumber\\
& = & \left\{(e^{\frac{1}{\sqrt[3]4}(\nu_1+\nu_2+\nu_3)}\a_0,e^{\frac{1}{\sqrt[3]4}(\nu_1+\nu_2-\nu_3)}\a_1,e^{\frac{1}{\sqrt[3]4}(\nu_1-\nu_2-\nu_3)}\a_2,e^{\frac{1}{\sqrt[3]4}(-\nu_1-\nu_2-\nu_3)}\a_3)\right\}\nonumber\\
& {} &
\eea
and
$$
\alpha_0 \star \alpha_1 \star \alpha_2 \star \alpha_3
: = \sigma_{0123}^* (\theta_0\oplus\theta_1\oplus\theta_2\oplus\theta_3)
$$
where $\sigma_{0123}: Q_0\star Q_1\star Q_2\star Q_3 \to SQ_0 \times SQ_1 \times SQ_2 \times SQ_3$
is the section defined by the formula \eqref{eq:4contactform}.

Enumerate the vertices of the pentagon by $0 \leq i \leq 4$ \emph{counterclockwise} starting
from $i = 0$ at the top vertex thereof.
Then we have a natural map
\beastar
&{}& \Phi_{0}: \P_+(SQ_0 \times SQ_1 \times SQ_2 \times SQ_3) \\
& \longrightarrow &
\P_+(SQ_0 \times SQ_1) \times_{Q_1} \P_+( SQ_1 \times SQ_2) \times_{Q_2} \P_+( SQ_2 \times SQ_3) \\
& \longrightarrow &
((Q_0 \star Q_1) \star Q_2) \star Q_3
\eeastar
from the center of the pentagon to the top vertex,
 because the equipped contact form in the domain is given by
$$
[\theta_0 \oplus \theta_1\oplus \theta_2 \oplus \theta_3]
$$
\emph{with the same conformal factors for all 4 factors}. 
(See \eqref{eq:4contactform}) By the same token, we can define 
the similar canonical map to each vertex. Furthermore it follows (see Figure 3)
 that
each edge map, denoted by $\Psi_{i(i+1)}$ for $ 0 \leq i \leq 4 (\mod 4)$,
is factored into the composition 
$$
\Psi_{i(i+1)} = \Phi_{i+1}\circ \Phi_i^{-1}, \quad 0 \leq i \leq 4.
$$

\begin{figure}[htb]\label{fig:trapizoid-decomposition}
  \centering
 \def\svgwidth{150pt}
  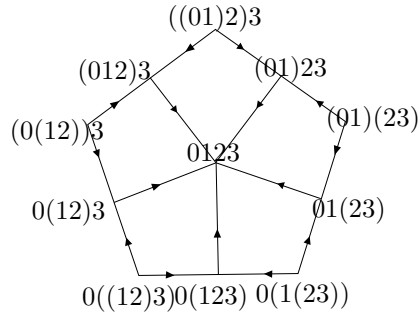
  \caption{Trapezoidal-decomposition}
\end{figure}

Then it is straightforward to check the equality
$$
\Phi_{23}\circ \Phi_{12} \circ \Phi_{01} =  \Phi_{43} \circ \Phi_{04}
$$
which is precisely the pentagon identity.
This finishes the proof.
\end{proof}

\begin{rem}\label{rem:nonorientable-case}
Finally we briefly mention how we can generalize our proof of pentagon axiom to the 
general case including non-coorientable cases.
In the proof of the pentagon axiom so far, we have  restricted ourselves to the coorientable case.
This was because contact forms
enable us to rigidify our choice of the linear pentagonal homotopy of associator isomorphisms.
However the same rigidification can be carried out by utilizing generalized contact forms similarly to
the usage of contact forms, if we use the $\star$ product of generalized contact forms instead as in
Definition \ref{defn:staralpha1alpha2}.
\end{rem}

\section{Monoidal property of $\Cont: \mathfrak{Cont} \to \mathfrak{Group}$}
\label{sec:star-product}

In this section, we examine the relationship between the monoidal structure 
on the category of $\mathfrak{Cont}$ contact manifolds and the category $\mathfrak{Group}$
of groups
$$
\mathfrak{Cont} \to \mathfrak{Group}; \quad (Q,\xi) \mapsto \Cont(Q,\xi).
$$

Let $(Q,\xi)$ and $(Q',\xi')$ be a pair of contact manifolds and consider its $\star$ product 
$Q\star Q' = \P_+(SQ \times SQ')$.

\begin{notation}\label{nota:x=}
Let $x \in Q \star Q'$ and express it as
$$
x = [\alpha,\beta] \in \P_+(SQ \times SQ')
$$
for a representative $(\alpha, \beta) \in SQ \times SQ'$. We denote by $x = x_{ab}$ when
$Q=Q_a$ and $Q'=Q_b$.
\end{notation}

We define the  map $\psi\star \psi': Q \star Q' \to Q \star Q'$ by putting
the value $\psi \star \psi'(x)$ for $x \in Q \star Q'$ to be
\be\label{eq:psistarpsi'-defn}
\psi \star \psi'(x): = [(d\psi^{-1})^*(\alpha), (d\psi'^{-1})^*(\beta)]
\ee
for any lift $(\alpha,\beta) \in SQ \times SQ'$ satisfying $x = [(\alpha,\beta)]$. As before,
we write
$$
S \psi =  (d\psi^{-1})^* : SQ \to SQ
$$
which we recall is $\R_+$-equivariant. 
\begin{prop} \label{prop:star-homomorphism} The map
$$
(\psi,\psi') \mapsto \psi \star \psi'
$$
given in \eqref{eq:psistarpsi'-defn}  is well-defined and
defines a group homomorphism 
$$
\Cont(Q,\xi) \times \Cont(Q', \xi') \to \Cont(Q\star Q').
$$
We denote the image  by
$$
\Cont(Q,\xi) \star \Cont(Q', \xi') \subset\Cont(Q\star Q').
$$
\end{prop}
\begin{proof} In this proof, for the notational simplicity, we  put 
$$
\widetilde \psi := S\psi.
$$
First we prove that this is well-defined: Let $x = [\alpha',\beta']$
for another representative $(\alpha',\beta')$. By definition, we have
$$
(\alpha', \beta') = (\phi_{\vec E_1}(\alpha), \phi_{\vec E_2}^t(\beta)) = (e^t\alpha,e^t\beta)
= e^t(\alpha,\beta) = \phi_{\vec E_1 \oplus \vec E_2}^t(\alpha,\beta).
$$
Therefore we have 
\beastar
(\psi\star\psi')[\alpha',\beta'] &=& [\widetilde{(\psi\star \psi')}(\alpha', \beta')]
= [\widetilde \psi(\alpha'), \widetilde \psi'(\beta')]\\
&=&[e^t \widetilde \psi (\alpha),e^t \widetilde \psi'(\beta)]
= [\widetilde \psi (\alpha),\widetilde \psi'(\beta)]
=(\psi\star\psi') [\alpha,\beta].
\eeastar

Now we consider the pair $(\psi_1, \psi_1'), \, (\psi_2,\psi_2') \in \Cont(Q) \times \Cont(Q')$ and
its product $(\psi_1\psi_2, \psi_1'\psi_2')$. 
We compute
\beastar 
(\psi_1\psi_2) \star (\psi_1'\psi_2')(x) 
& =  &[\widetilde{\psi_1 \psi_2}(\alpha), \widetilde{\psi_1' \psi_2}(\beta)]
= [\widetilde \psi_1 \widetilde \psi_2(\alpha), \widetilde\psi_1'\widetilde  \psi_2'(\beta)]\\
& = & (\psi_1 \star \psi_1') [ \widetilde \psi_2(\alpha), \widetilde \psi_2'(\beta)]\\
& = &  (\psi_1 \star \psi_1') ( \psi_2 \star \psi_2') [(\alpha,\beta)]\\
& = &  (\psi_1 \star \psi_1') ( \psi_2 \star \psi_2') (x).
\eeastar
Since this holds for all $x \in Q_1 \star Q_2$, this verifies 
$$
(\psi_1\psi_2) \star (\psi_1'\psi_2') = 
(\psi_1 \star \psi_1') ( \psi_2 \star \psi_2').
$$
It is straightforward to check that $\id_Q \star \id_Q = \id_{Q\star Q}$.
This finishes the proof.
\end{proof}

Now we describe the associated Lie algebra. Recall the $\R_+$ action on $SQ_1 \times SQ_2$
induced by $\vec E_{12} = E_1 \oplus E_2$.
We denote by 
$$
\Symp^{\R_+}(SQ_1 \times SQ_2)
$$
the set of $\R_+$-equivariant symplectomorphisms of $SQ_1 \times SQ_2$, and 
by $\mathfrak{symp}^{\R_+}(SQ_1 \times SQ_2)$ its Lie algebra. We denote by
$$
\Pi: \Symp^{\R_+}(SQ_1 \times SQ_2) \to \Cont(Q_1 \star Q_2)
$$ 
the natural projection, and define by
$$
\CE_{1 2} \to SQ_1 \times SQ_2
$$
the line bundle whose fiber is given $\R \{\vec E_{12}(\beta_1,\beta_2)\}$
with $\vec E_{12} = E_1 \oplus E_2$.

Then the following 
description of $\mathfrak X(Q_1 \star Q_2)$ is immediate from definition.
 
\begin{lem}\label{lem:frakX-lft} We have the description
\beastar
\mathfrak X (Q_1 \star Q_2)  = 
\{(\widetilde X_1, \widetilde X_2)  \in \mathfrak{symp}^{\R_+}(SQ_1 \times SQ_2) 
\mid X_i \in \mathfrak X(Q_i, \xi_i), \, i= 1, \, 2\} /\sim
\eeastar
and hence an exact sequence
$$
0 \longrightarrow   \Gamma( \CE_{1 2})
\longrightarrow \mathfrak{symp}^{\R_+}(SQ_1 \times SQ_2) \stackrel{d\Pi}{\longrightarrow}
 \mathfrak X(Q_1 \star Q_2)
\longrightarrow 0.
$$
In particular, $\mathfrak X (Q_1 \star Q_2)$
is an infinite dimensional Frechet vector space.
\end{lem}
\begin{proof} Since the Lie algebra $\mathfrak{symp}^{\R_+}(SQ_1 \times SQ_2)$ of $\Symp^{\R_+}(SQ_1 \times SQ_2)$ plays a role of the model space of Frechet manifold $\Symp^{\R_+}(SQ_1 \times SQ_2)$ \emph{independently} on the choice of base point $\phi\in\Symp^{\R_+}(SQ_1 \times SQ_2),$ we have only to prove 
$\ker d\Pi|_{(d\varphi_1^{-1},d\varphi_2^{-1})} =  \Gamma( \CE_{1 2})$.

Suppose $(\zeta_1,\zeta_2)$ is an $\R$-invariant vector field on $SQ_1 \times SQ_2$
satisfying
$$
d\Pi(\zeta_1,\zeta_2) = 0, \quad \zeta_i \in T_{\varphi_i}\Cont(Q_i), \, i=1, \, 2.
$$
Take the germ of
curves $\{\varphi_i^s\}$ representing $\zeta_i$.  Then we have
$$
\frac{d}{ds}\Big|_{s = 0} [(\varphi_1^s, \varphi_2^s)] = \frac{d}{ds}\Big|_{s = 0} (\varphi_1^s \star \varphi_2^s)=d\Pi(\zeta_1,\zeta_2).
$$
This is equivalent to
$$
\frac{d}{ds}\Big|_{s = 0} (\varphi_1^s \star \varphi_2^s)(x) = 0
$$
for all $x \in  Q_1 \star Q_2$.  We put $q_i = \pi_i^\star(x)$, and
represent
$$
(\varphi_1^s \star \varphi_2^s)(x) = [\beta_1^s, \beta_2^s], \quad \beta_i^s \in S_{q_i} Q_i
$$
such that $\beta_i^s$ is a germ of curves in $SQ_i$ satisfying 
$$
\varphi_i^s(\pi_i^\star(x)) = \pi_i(\beta_i^s).
$$
We may choose the pair $(\varphi_1^s,\varphi_2^s)$ which descends 
to a germ of curves $\varphi_1^s \star \varphi_2^s$ by Proposition \ref{prop:star-homomorphism}.
Then, by taking the derivatives in $s$ at $s = 0$, we obtain
$$
X_i((\pi_i^\star(x)) = d\pi_i(\zeta_i(\beta_i))
$$
and $(X_1,X_2)$ descends to $X_1 \star X_2$ that is represented by the germ 
$\{\varphi_1^s \star \varphi_2^s\}$
at $\varphi_1 \star \varphi_2$ by  the way how we chose the germs $\varphi_i^s$.
Therefore we have
\beastar
\left(X_1((\pi_1^\star(x))), X_2((\pi_2^\star(x)))\right) & = & \left(d\pi_1(\zeta_1(\beta_1)), d\pi_2(\zeta_2(\beta_2))\right)\\
& = & d(\pi_1 \times \pi_2)  \left(\zeta_1(\beta_1), \zeta_2(\beta_2)\right)
\eeastar
By taking the quotient, we obtain
$$
(X_1 \star X_2)(x) = d\Pi(\zeta_1,\zeta_2)(x) = 0.
$$
Since this holds for all $x \in Q_1 \star Q_2$, we obtain $X_1 \star X_2 = 0$.
The latter is equivalent to 
$$
(\zeta_1,\zeta_2)(\beta_1,\beta_2) = (\zeta_1(\beta_1), \zeta_2(\beta_2)) 
= c \vec E_{12}(\beta_1,\beta_2)$$
for some $c = c(\beta_1,\beta_2)  \in \R$. Since the left hand side is globally defined on $SQ_1 \times SQ_2$,
so is the function $c$ which is also smooth.
By the $\R$-invariance of the vector fields, $c$ must satisfy
$$
\CL_{\vec E_{1 2}}(c) = 0.
$$
This implies that $(\zeta_1,\zeta_2)$ defines a section of the line bundle $\CE_{12}$,
which proves exactness of the sequence.
\end{proof}

\part{Calculus of Legendrian correspondence}

 In symplectic topology, the \emph{Cartesian product} of symplectic manifolds is again a symplectic manifold, and Lagrangian submanifolds of this product are called the `Lagrangian correspondence'. Lagrangian manifolds are
  the main building block of the `Symplectic Category'. (See \cite{alan:category} and \cite{wehrheim-woodward}.) 

However Cartesian products of contact manifolds are not contact manifolds and one should at least
replace the Cartesian product by the \emph{functorial} contact product we introduce in the present paper 
to think of  the notion of Legendrian correspondence in
the contact category $\mathfrak{Cont}$. As mentioned before,
 the presence of self-intersections of a Legendrian submanifold is not a stable phenomenon unlike that of
Lagrangian submanifolds. We anticipate that this makes the calculus of Legendrian correspondence technically 
easier to construct in contact topology than in symplectic topology,
 when one tries to construct the Fukaya-type 2-category using the theory of quilted version of
 contact instanton homology as the 1-morphisms as in the symplectic category \cite{wehrheim-woodward}.

\section{Contact product and the definition of Legendrian correspondences}

We start with the definition of Legendrian correspondence.

 \begin{defn}[Legendrian correspondence] We call
 an \emph{embedded} Legendrian submanifold $R$ of the contact product $Q_0 \star Q_1$
 a Legendrian correspondence from $Q_0$ to $Q_1$.
 \end{defn}
 
\begin{rem}
Unlike the symplectic case, there is no obvious `identity' element for the $\star$ product.
It is an interesting question to see whether  one should treat the contact monoid 
as `unital or not'  in a weaker sense, such as `up to homotopy.'
\end{rem}

We will show that similarly as for the Lagrangian correspondence, a contactomorphism from $Q_0$ to $Q_1$
can be regarded as a Legendrian correspondence via the process of \emph{Legendrianization}
through the new monoidal structure,  contact product  in the contact category we define. 
In the coorientable cases equipped with contact forms, the correspondence can be easily described
as in \cite{lychagin}, \cite{bhupal}, \cite{oh:shelukhin-conjecture} which we recall.

\begin{defn}[Conformal exponent] Assume $(Q,\xi)$ is coorientable and equip it with
contact form denoted by $\lambda$. The \emph{conformal exponent} of contactomorphism $\psi:Q \to Q$
satisfies $\psi^*\lambda = \pm e^{g_\psi} \lambda$ for some real-valued function $g_\psi$.
\end{defn}
To each positive contactomorphism $\psi:(Q,\xi )\to (Q,\xi)$ can be associated a
Legendrian submanifold,  $\Gamma_\psi\subset Q_0\times Q_1\times \mathbb{R}$, called the
\emph{contact graph}
$$
\Gamma_\psi =\{(x,\psi(x),g(x))\in Q \times Q\times \mathbb{R} \mid x\in Q, \psi^{*} \lambda
=e^{g(x)}\lambda \}
$$
and some direct calculation verifies that $\Gamma_\psi$ is a Legendrian submanifold of 
$$
(Q\times Q\times \mathbb{R}, \mathscr A), \quad \mathscr A = - e^\eta \pi_1^* \lambda + \pi_2^*\lambda.
$$
Furthermore, for two given  contactomorphisms $\psi_1, \, \psi_2 : Q \to Q$
a straightforward calculation provides an explicit expression of 
the graph of composition  $\Gamma_{\psi_1\circ \psi_0}$, which is given by
$$
\Gamma_{\psi_1\circ\psi_0}=\{(x,\psi_1\psi_0(x),g_{\psi_1\psi_0}(x)) \in Q \times Q \times \R\mid x \in Q\}.
$$
We will later describe $\Gamma_{\psi_1\circ\psi_0}$ as the composition of 
two Legendrian correspondences $\Gamma_{\psi_0}$ and $\Gamma_{\psi_1}$
for the non-coorientable case.

Let $(Q_i,\xi_i)$ for $i = 1, \, 2, \, 3$ be contact manifolds. We write the $\star$ projection
$Q_1 \star Q_2 \star Q_3 \to Q_i \star Q_j$ by $\pi_{ij}^\star$ for any pair with $1 \leq i \leq j \leq 3$.

\subsection{Legendrianization of contactomorphisms}

Let $\psi: (Q,\xi) \to (Q', \xi')$ be a contact diffeomorphism. Then the derivative $d\psi: TQ \to TQ'$
canonically induces a morphism of the short exact sequences
$$
\xymatrix{0 \ar[r]  &\xi \ar[r] \ar[d]_{d\psi|_\xi}& TQ \ar[r] \ar[d]_{d\psi}  & I^\xi \ar[r] 
\ar[d]_{\psi_*} & 0\\
0 \ar[r]  &\xi' \ar[r] & TQ' \ar[r]  & I^{\xi' } \ar[r] &  0
}
$$
where the map $\psi_*: I^\xi \to I^{\xi'}$ is one that is canonically induced from the tangent map
$d\psi$. Furthermore, $d\psi$ also preserves the zero section

and hence induces the map $(\psi)_*: SQ \to SQ'$.

\begin{lem}[Contact graph]\label{lem:graph} Denote by $\sim$ the Liouville equivalence relation 
used in the definition of
contact product $(Q_1 \star Q_2, \Xi_{\bar 12})$, which is induced by the 
Liouville one-form $[-\theta_1 \oplus \theta_2]$.  Consider the subset
$$
\Gamma_\psi: = \{(\beta, (\psi)_*(\beta)) \in SQ_1 \times SQ_2 \, \mid \, \beta \in SQ_1\}/\sim
$$
Then $\Gamma_\psi$ is a Legendrian submanifold of $\overline{Q}_1 \star Q_2$.
\end{lem}
\begin{proof} Let $\zeta \in  T\Gamma_\psi$ with $\pi(\zeta) = [(\b,\psi_*(\b))]$. 
Let $\widetilde \zeta = (\widetilde \zeta_1, \widetilde \zeta_2) \in T(SQ_1 \times SQ_2)$ 
be any lift of $\zeta$. Then there is a germ of curve
$(\gamma_1,\gamma_2):(-\epsilon,\epsilon) \to SQ_1 \times SQ_2$ such that
$$
(\dot \gamma_1(0), \dot \gamma_2(0)) = (\dot \gamma_1(0), \psi_*(\dot \gamma_1)(0)) + c (E_1(\beta), E_2(\psi_*(\b)))
$$
for some constant $c$. We evaluate the product Liouville one-form
\beastar
(-\theta_1,\theta_2)((\dot \gamma_1(0), \dot \gamma_2(0)))
& = & (-\theta_1,\theta_2) \left(\dot \gamma_1(0), \psi_*(\dot \gamma_1)(0) + c (E_1(\beta),E_2(\psi_*(\beta)))\right)\\
& = &  -\theta_1((\dot \gamma_1(0)) + \theta_2( \psi_*(\dot \gamma_1)(0))\\
& = & -\beta(d\pi_1 (\dot \gamma_1(0))) + \psi_*(\beta) (d\pi_2 (\psi_*(\dot\gamma_1(0)))) \\
& = & -\beta(d\pi_1 (\dot \gamma_1(0)))+\beta(d\pi_1 (\dot \gamma_1(0))) = 0.
\eeastar
This proves
$$
T \Gamma_\psi \subset \Xi_{\bar 1 2}|_{\Gamma_\psi}
$$
which finishes the proof.
\end{proof}

For two given contactomorphisms 
$$
\psi_{01}:(Q_0,\xi_0)  \to (Q_1,\xi_1) , \quad \psi_{12}:(Q_1,\xi_1) \to (Q_2,\xi_2),
$$
we consider  the associated  Legendrian graphs $\Gamma_{\psi_{01}}$  and
$\Gamma_{\psi_{12}}$ respectively. By a slight abuse of notations, we also denote 
the relevant  Legendrian embedding by the same symbol
$$
\Gamma_{\psi_{01}}: Q_0 \to Q_0 \star Q_1, \quad \Gamma_{\psi_{12}}: Q_1 \to Q_2.
$$
Here we give a coordinate-free definition of contact graphs of contact diffeomorphisms. This will be generalized to \emph{the composition of Legendrian correspondence.} See Definition \ref{defn:composition}.
\begin{defn}
    We define $\Gamma_{\psi_{12}} \circ \Gamma_{\psi_{01}}$ to be the composition of contact graphs of contactomorphisms 
    $\psi_{01}:Q_0\to Q_1$ and $\psi_{12}:Q_1\to Q_2$, by
    $$\Gamma_{\psi_{12}} \circ \Gamma_{\psi_{01}}:=[\pi_{02}\left(\{(\a,(\psi_{01})_*(\a),\b,(\psi_{12})_*(\b))\mid \a\in SQ_1,\b\in SQ_1\}\cap \D_{SQ_1}\right)].$$
\end{defn}

\begin{prop} We have
$$
\Gamma_{\psi_{12}\circ\psi_{01}} = \Gamma_{\psi_{12}} \circ \Gamma_{\psi_{01}}.
$$
\end{prop}
\begin{proof} By construction given in Lemma \ref{lem:graph}, we have
$$
\Gamma_{\psi_{12}\circ\psi_{01}} = \{(\beta, (\psi_{12}\circ\psi_{01})_*(\beta) \in SQ_0 \times SQ_2) \mid 
\beta \in SQ_0\}/\sim_{02}
$$
where $\sim_{02}$ is the Liouville equivalence relation induced by the Euler vector field
$$
\vec E_{ 02}
$$
of $T^*Q_0 \times T^*Q_2$. We can rewrite
\bea\label{eq:psi*-compose}
(\id \times (\psi_{12}\circ\psi_{01}))_*  (\beta, \beta)
& = & (\id \times \psi_{12})_*(\id \times \psi_{01})_*(\beta,\beta) \nonumber\\
& = & (\id \times \psi_{12})_*(\id \times \psi_{01})_*(\beta,\beta).
\eea

\begin{lem} Let $\psi: (Q,\xi) \to (Q',\xi')$ be any diffeomorphism, and let $\vec E, \, \vec E'$
be the Euler vector fields of $T^*Q$ and $T^*Q'$ respectively. Then 
$\psi_*(\vec E) = \vec E'$.
\end{lem}
\begin{proof}  Let $\alpha \in T^*Q$ with $\pi(\alpha) = q \in Q$.
We compute
\beastar
d\psi(\vec E_1(\alpha)) &= & d\psi \left(\frac{d}{dt}\Big|_{t=0} (e^t \alpha)\right) 
=  \left(\frac{d}{dt}\Big|_{t=0} e^t (d\psi( \alpha))\right) \\
& = &  \left(\frac{d}{dt}\Big|_{t=0} d \left(e^t \cdot \psi(\alpha)\right)\right) 
=  \left(\frac{d}{dt}\Big|_{t=0}  \left(e^t \cdot d\psi(\alpha)\right)\right) \\
& =  & \vec E_2(d\psi(\alpha))
\eeastar
where 
the fourth equality follows since the scalar multiplication by $e^t$ is a (fiberwise)
linear map,
and  the third equality arises the two curves $\gamma_1, \, \gamma_2$ defined by
$$
\gamma_1(t) = d\psi(e^t \cdot \alpha), \quad \gamma_2(t) =   e^t \cdot d\psi(\alpha)
$$
satisfies 
$$
\gamma_1(0) = \gamma_2(0) = \d\psi(\alpha), \quad \gamma_1'(0) = \gamma_2'(0) = \vec E_2(d\psi(\alpha)):
$$
Here $e^{(\cdot)}$ is the scalar multiplication at $T_{\pi(\a)}^*Q$ and at $T_{\psi(\pi(\alpha))}^*Q'$
respectively.
Therefore we have proved $\psi_*(\vec E_1) = \vec E_2$.
\end{proof}
By this lemma, we can descend the equality \eqref{eq:psi*-compose} to the quotient space and get
\beastar
[(\id \times (\psi_{12}\circ\psi_{01}))_*  (\beta, \beta)] & = & 
[(\id \times \psi_{12})([(\beta, (\psi_{01})(\beta))] \\
& = & [\id \times \psi_{12})]_*[\id \times \psi_{01}]_*[(\beta,\beta)]
\eeastar
for all $[\beta,\beta] \in \Delta_Q^\star$.
It proves
$$
\Gamma_{\psi_{12}\circ \psi_{01}} = (\id \star \psi_{12})(\Gamma_{\psi_{01}}) = 
\Gamma_{\psi_{12}} \circ \Gamma_{\psi_{01}}.
$$
This finishes the proof.
\end{proof}

\subsection{Coorientation change of contact products}

It is a common practice in symplectic category to denote by $\overline M$ 
$$
(M,-\omega)
$$
the \emph{symplectic dual} of  a given symplectic manifold $(M,\omega)$. 
In the contact category $\mathfrak{Cont}$, it is not obvious whether
there is a corresponding notion for a general non-coorientable contact manifold $(Q,\xi)$
or not. 

At least for the coorientable case equipped with a coorientation, one might try to define the dual
by taking the same contact manifold $(Q,\xi)$ equipped the \emph{opposite coorientation}. 
For example, recalling that our contact product $Q \star Q'$ is coorientable, one might try to
develop a notation to systematically distinguish between $Q \star Q'$ to its dual, especially 
the contact counterpart $\overline Q \star Q'$ of 
$$
\overline M \times M'
$$
in symplectic category. 

In this subsection, we develop the systematic calculus of `taking the bar'. Recall $T^*Q$ 
is a symplectic manifold equipped with the canonical symplectic form
\be\label{eq:omega0}
\omega_0 = -d\theta
\ee
where $\theta$ is the \emph{classical Liouville form} $\theta = p\, dq$, which is
canonically given one-form on the cotangent bundle \emph{irrespective of the choice of 
the sign} in the definition of base symplectic form on the cotangent bundle.
depending only on the smooth structure of $M$ in the following sense:

\begin{prop}[Defining property of Liouville one-form]\label{prop:defining-theta}
The Liouville one-form $\theta$ on a general smooth manifold $N$ is uniquely characterized by the property that 
\be\label{eq:defining-theta}
 \tilde{\beta}^* \theta = \beta
\ee
where $\tilde{\beta}:N \longrightarrow T^* N$ is the section map associated to  
any one-form $\beta$ on $N$.
\end{prop}
Similarly the \emph{Euler vector field} $\vec E$ on $T^*N$ is defined to be the vector field 
generating the conformal isotopy on $T^*N$ via
$$
(r,q,p) \mapsto (q,r p) = \phi_{\vec E}^{\log r}
$$
In fact, we have the explicit formula
$$
\vec E(q,p) = \sum_{i=1}^n p_i \frac{\del}{\del p_i}
$$
in any canonical coordinates $(q^1,\cdots, q^n, p_1, \cdots,p_n)$ associated to
a coordinate system $(x^1,\cdots, x^n)$ on the base $N$. The right hand side expression does not
depend on the choice of this latter coordinate system of the base $N$. In our discussion of the aforementioned bar notations
henceforth, this $\theta$ and $\vec E$ will be the touchstone which will be unchanged.

The Liouville one-form $\theta$ and the Euler vector field $\vec E$ satisfy
\be\label{eq:equation1}
\vec E \intprod \omega_0 = \theta, \quad \omega_0 = -d\theta,
\ee
or equivalently
\be\label{eq:equation2}
(- \vec E) \intprod \omega_0 = \overline\theta, \quad \omega_0 = d \overline\theta
\ee
where
$$
\overline\theta:=-\theta.
$$

Now we observe that $SQ$ is invariant under the fiber scaling and
so the Euler vector field $\vec E$ is tangent to $SQ$. In the following discussion, we 
always take as the base symplectic form on $SQ$ the pull-back 
$$
i^*\omega_0 = - d(i^*\theta).
$$
We often simply write $i^*\omega_0$ and $i^*\theta$ as $\omega_0$ and $\theta$ as long as
there is no danger of confusion. Then the associated Liouville vector field is $Z = \vec E$. The dual 
$\overline{SQ}$ then is given by
$$
(SQ, - \omega_0) = (SQ, -d\overline\theta)
$$
and its associated Liouville vector field is  $-\vec E$. Now we are ready to give
the definition of $\overline Q \star Q'$.

\begin{defn}[$\overline Q_0 \star Q_1$] Consider the product symplectic manifold
$$
(SQ_0 \times SQ_1, -\omega_0^0 \oplus \omega_0^1)
$$
and consider its Liouville vector field $E_{\bar{0}1}: =(-\vec E_0) \oplus E_1$.  We consider its flow 
on $SQ_0\times SQ_1$ given by
$$
(t, \beta_1, \beta_2)  \mapsto (e^{-t} \beta_1,e^t \beta_2) \quad \beta_i \in SQ_i, \, t \in \R.
$$
Then we define
$$
\overline Q_0 \star Q_1: = SQ_0 \times SQ_1/\sim \, = \P_+(\overline{SQ}_0 \times SQ_1)
$$
equipped with the induced contact structure thereon. 
\end{defn}

\subsection{The metric coorientations of $\P_+(SQ)$ and of $\overline Q_0 \star Q_1$}

Recall that the set of Riemannian metrics on a given compact or tame manifold $N$ is convex 
and so contractible. For a given metric $g$ on $N$, we
consider their geodesic flow of $TN$ which is generated by
the associated geodesic spray. We denote their pushforward
by the diffeomorphisms  $\widetilde g: TN \to T^*N$ induced by the map
$X \mapsto g(\cdot, X)$ . 

We now specialize the discussion in the above paragraph
 to contact manifolds $N = Q$ and consider the subbundle $\xi \to TQ$ over $Q$.
When we are given a metric $g$ on $Q$, we have the orthogonal decomposition
$$
TQ = \xi \oplus_g \xi^{\perp_g}
$$
By dualizing it, we obtain the splitting
$$
T^*Q = (\xi^{\perp_g})^{\perp} \oplus_g \xi^{\perp}.
$$
We note that $SQ = \xi^\perp \setminus \{0\}$ by definition of $SQ \subset T^*Q \setminus \{0\}$

Let $H_g: T^*Q \to \R$ be the  fiberwise quadratic Hamiltonian  on $T^*Q$ defined by
$$
H_g(\alpha) = \frac12 g^\sharp(\alpha,\alpha).
$$
Its associated Hamiltonian flow on $T^*Q$ preserves the level set $U_g^*Q = H_g^{-1}(2)$, and defines
the Reeb vector field of the canonically induced contact form 
on $U_g^*Q$ given by
$$
\alpha_g: = - \iota_g ^*\theta
$$
where $ \iota_g:U_g^*Q \to T^*Q$ is the inclusion map of the hypersurface $U_g^*Q =H_g^{-1}(1)$.
We put
\be\label{eq:SgQ}
S_g Q: = U^*_g Q \cap SQ = (H_g|_{S_gQ})^{-1}(1)
\ee
Then the aforementioned Hamiltonian flow is also tangent to $S_gQ$ which is also 
the Hamiltonian vector field of the restriction $H_g|_{S_gQ}$.  

We have a natural map 
$$
(S_{g_0}Q_0 \cap \ker \theta_0)\times (S_{g_1}Q_1 \cap \ker \theta_1) 
\to Q_0 \star Q_1 = SQ_0 \times SQ_1 /\sim.
$$
Furthermore we have
$$
T_{\beta_0}(SQ_0) \oplus T_{\beta_1}(SQ_1) = 
(\ker \theta_{g_0}) \oplus (\ker \theta_{g_1}) \oplus \R \{R_{g_0}, R_{g_1}\}.
$$
This implies
\beastar
T_x(Q_1 \star Q_2) & = & T_{\beta_0}(SQ_0) \oplus T_{\beta_1}(SQ_1)/\sim \\ 
& \cong & \left[ (\ker \theta_{g_0}) \oplus (\ker \theta_{g_1})\right]
 \oplus \R \{ -a R_{g_0} + b R_{g_1}\} \oplus \R \{ a R_{g_0} + b R_{g_1}\}
\eeastar
with $a +b = 0$. 
We compute 
$$
\theta(R_g(\beta)) = \alpha_g(R_g(\beta)) =  -1.  
$$
On the other hand,
\beastar
\theta(X_H(\beta) )& = & (E(\beta) \intprod -d\theta) (X_H(\beta)) \\
& = & \omega_0(E(\beta), X_H(\beta))
= - dH_g (E(\beta)) = -\CL_{E}(H_g) (\beta) \\
& = & -2 H_g(\beta)= -2 \neq 0.
\eeastar
This shows $R_g = \frac12 X_{H_g}$. Using the canonical isomorphism $U_g^*Q \cap SQ \cong \P_+(SQ)$,
this contact form $\alpha_g$ naturally induces a coorientation on $\widetilde Q= \P_+(SQ)$ which we 
call the \emph{metric coorientation} of $\widetilde Q$.

Applying the above discussion to $Q = Q_0 \star Q_1$,
we are now ready to systematically study the coorientation changes on the $\star$-product
and explain the precise meaning of the notation such as
$$
\overline Q_0 \star Q_1.
$$
We would like to alert readers that \emph{there is no natural meaning of the $\overline Q$ alone
when $Q$ is non-coorientable!}

\begin{defn}[Canonical coorientation of $\overline Q_0 \star Q_1$] 
Let $g_0, \, g_1$ be any Riemannian metrics of $Q_0, \, Q_1$ respectively.
The map $\widetilde g_i: TQ_i \to T^*Q_i$ restricts to $UQ_i \to U^*Q_i$ 
where $UQ_i$ and $U^*Q_i$ are the unit tangent and cotangent bundle respectively.
Denote by $R_g$ and $R_{g'}$ the Reeb vector fields on $U_{g_0}^*Q_0$ and $U_{g_1}^*Q_1$.
Then we have
$$
T_x(Q_1 \star Q_2) = \Xi \oplus \R\{-R_{g_0} + R_{g_1}\}
$$
where we put
$$
\Xi: =  \left[ (\ker \theta_{g_0}) \oplus (\ker \theta_{g_1})\right]
 \oplus \R \{R_{g_0} + R_{g_1}\} 
$$
is the model for the $\xi_1 \star \xi_2$ represented by a subbundle of $TS_1 \times TS_2$.
In particular,  the vector field $- R_{g_0} + R_{g_1}$ descends to 
$Q_1 \star Q_2$ as a nowhere vanishing transverse vector field of $\xi_1 \star \xi_2 \subset T(Q_1\star Q_2)$.

We call the corresponding coorientation the \emph{canonical (metric) coorientation} 
of $\overline Q_0 \star Q_1$.
\end{defn}

Using the convexity of $\Riem (Q_i)$,  it follows that the aforementioned metric 
coorientation does not depend on the choice of metrics.

\section{Definition of the composition of Legendrian correspondences}
\label{sec:composiiton}

It turns out that to be able to define the composition of Legendrian correspondence as 
well as to make our exposition more functorial, irrespective of cases of  coorientable or not, 
it is important to take a Legendrian submanifold in the product
with the first factor with its coorientation changed. 

Recall $\widetilde Q: = \P_+(SQ)$ is coorientable and in particular so is $Q_1 \star Q_2 =\P_+(SQ_1 \times SQ_2)$:
The canonical Euler vector field on $SQ_1 \times SQ_2$ is given by
$$
\vec E_1 \oplus \vec E_2 =: \vec E_{12}
$$
and $\P_+(SQ_1 \times SQ_2) = SQ_1 \times SQ_2/\sim$ where $\sim$ is the quotient by the flow of $\vec E_{12}$.
\begin{defn} [{$\overline Q_a \star Q_b$}]  Denote by $\vec E_{\bar a b}$ the vector field
\be\label{eq:Ebarab}
\vec E_{\bar a b}: = - \vec E_a \oplus E_b.
\ee
We define
\be\label{eq:barQ1starQ2}
\overline Q_a \star Q_b: = SQ_a \times SQ_b/\sim, \quad \Xi_{\bar a b} := \ker[ -\theta_a \oplus \theta_b]
\ee
where $\sim$ is the quotient by the flow of $\vec E_{\bar a b}$. When $Q_1$ is coorientable, we also denote 
the associated contact distribution by $\overline \xi_1 \star \xi_2$.
\end{defn}
For notational consistency, we will also denote by $\overline \xi_a \star \xi_b = \Xi_{\bar a b}$
even when $(Q_a,\xi_a)$ is not coorientable in general, although the notation $\overline \xi_a$ does not make sense as it is
in that case.
Recall that $\vec E_{\bar a b}$ is the Liouville vector field of 
$$
(SQ_a \times SQ_b, -\theta_a \oplus \theta_b)
$$
of the Liouville one-forms $\theta_\bullet$ on $T^*Q_\bullet$.
We will systematically use this notation from now on.

After this convention clearly established, let
$$
R_{12} \subset \overline Q_1 \star Q_2, \quad  R_{23} \subset  \overline Q_2 \star Q_3
$$
be two Legendrian correspondences. We consider the subset
$$
\overline Q_1 \star (\bar \Delta_{Q_2}^\star) \star Q_3 \subset 
\overline Q_1 \star (Q_2 \star \overline Q_2) \star Q_3.
$$
\begin{defn}[Contact diagonal] For given contact manifold $(Q,\xi)$, we define the
\emph{contact diagonal} 
$$
\Delta_Q^\star  \subset Q \star Q
$$
to be the subset 
\be \label{eq:DeltaQstar}
\{[(\beta,\beta) \in SQ \times SQ \mid \beta \neq 0 \}/\sim \quad \subset \overline Q \star Q
\ee
where $\sim$ is induced by the flow of $\vec E_{\bar 1 2}$. We define
$$
\overline \Delta^\star_{Q_2} \subset Q \star \overline Q
$$
the quotient given by the flow of $\vec E_{1 \bar 2}$.
\end{defn}

Recall the notation
$$
\Delta_2^S: = \Delta_{SQ_2} =  \{(\beta,\beta) \in SQ \times SQ \mid \beta \neq 0 \}
$$
from \eqref{eq:DeltaQ2S}.
We define the map
\be\label{eq:pi1Delta23}
\pi_{1,\Delta_2,3}: SQ_1 \times \Delta_2^S \times SQ_3 
 \to \overline Q_1 \star (\bar \Delta_{Q_2}^\star) \star Q_3 \subset \overline Q_1 \star (Q_2 \star
\overline Q_2) \star Q_3.
\ee
Then we have
$$
R_{12} \star_{Q_2} R_{23}: = \pi_{1,\Delta_2,3}(\pi_{12}^{-1}(R_{12}) \times \pi_{23}^{-1}(R_{23})) \subset 
\overline Q_1 \star (Q_2 \star \overline Q_2) \star Q_3
$$
when $R_{12} \subset \overline Q_1 \star Q_2$ and $R_{23} \subset \overline Q_2 \star Q_3$.

The next lemma is worthwhile to separately state to define the \emph{strongly composability of Legendrian correspondences} later. See Definition \ref{defn:composition}. 
\begin{lem} The projection 
$$
p_{\bar 13}^\star: \overline Q_1 \star \bar \Delta_{Q_2}^\star \star Q_3 \to \overline Q_1 \star Q_3
$$
is well-defined.
\end{lem}
\begin{proof}
    It naturally descends from the Cartesian projection
    $$
    SQ_1\times \D^S_2\times SQ_3\to SQ_1\times SQ_3
    $$
    since Liouville vector fields of domain and codomains are given by $(-\vec E_1,\vec E_2,- \vec E_2, \vec E_3)$ 
    and $(-\vec E_1, \vec E_3)$ respectively. Therefore the projection is equivariant
    with respect to the Liouville action and hence follows the lemma.
    
\end{proof}

The following notation is useful for further discussion on the composition operation.
\emph{We would like to emphasize that such a notation is not possible for a general pair
$$
(x_{12}, x_{23}) \in (Q_1\star Q_2) \times (Q_2\star Q_3),
$$
let alone for those coming from $(Q_1 \star Q_2) \times (Q_2\star Q_3)$.}

\begin{notation}[$x_{12,23}$] Let $x_{12} \in R_{12}$ and $x_{23} \in R_{23}$.
Suppose $\pi_2^\star(x_{12}) = \pi_2^\star(x_{23)}$. Then 
we define
$$
x_{12,23} : = \pi^\star[ \beta_1,(\beta_2,\beta_2),\beta_3]
$$
for the lifts $x_{12} = [\beta_1,\beta_2]$ and $x_{23} = [\beta_2,\beta_3]$.
\end{notation}
\begin{lem} The definition of $x_{12,23}$ is well-defined.
\end{lem}
\begin{proof} We need to show
$$
\pi^\star[ \beta_1,(\beta_2,\beta_2),\beta_3] = \pi^\star[ \beta_1',(\beta_2',\beta_2'),\beta_3']
$$
whenever $[\beta_1',\beta_2'] = [\beta_1,\beta_2]$ and $[\beta_2',\beta_3'] = [\beta_2,\beta_3]$.
This is straightforward to check and so its proof is omitted.
\end{proof}
We will also denote by the tangent vector
$$
\zeta_{12} \star \zeta_{23} \in T_{x_{12,23}}((\overline Q_1\star Q_2)\star(\overline Q_2 \star Q_3))
$$
the linearized version of the map $(x_12, x_{23}) \mapsto x_12 \star x_{23}$, which is
associated to the tangent pair 
$$
(\zeta_{12},\zeta_{23}) \in T_{x_{12}}(\overline Q_1 \star Q_2) \times T_{x_{23}}(\overline Q_2 \star Q_3).
$$

\begin{defn} We call a pair $R_{12},\, R_{23}$ \emph{composable} if 
the map $\pi_{1,\Delta_2,3}$ in \eqref{eq:pi1Delta23} is transverse to $R_{12}\star R_{23}$  in 
$(\overline Q_1 \star Q_2) \star (\overline Q_2 \star Q_3)$.
\end{defn}

The following is easy to check whose proof we will omit because more nontrivial cases will be
examined shortly.
\begin{lem}
Composability is a generic phenomenon for the pairs $(R_{12}, R_{23})$.
\end{lem}
\begin{prop} 
The subset
$$
R_{12} \star_{Q_2} R_{23} \subset  \overline Q_1 \star \bar \Delta_{Q_2}^\star \star Q_3 
$$
is an embedded submanifold for composable pair ($R_{12},R_{23}$), which is the image of
the smooth map
$$
\pi_{1,\Delta_2,3}\Big|_{\pi_{1,\Delta_2,3}^{-1}(R_{12} \star R_{23})}
 \to \overline Q_1 \star \bar \Delta_{Q_2}^\star \star Q_3.
$$
\end{prop}
\begin{proof} By definition, we have
$$
R_{12} \star_{Q_2} R_{23} = \{(\beta_{12},\beta_{23}) \in \pi^{-1}_{12}(R_{12}) \times \pi^{-1}_{23}(R_{23}) \mid \pi_2(\beta_{12}) = \pi_2(\beta_{13})\}/
\sim
$$
where $\sim$ is the quotient induced by the flow of $\vec E_{\bar 12\bar 23}$.
Clearly 
$$
\{(\beta_{12},\beta_{23}) \in \pi^{-1}_{12}(R_{12}) \times \pi^{-1}_{23}(R_{23}) \mid \pi_2(\beta_{12}) = \pi_2(\beta_{23})\}
$$
is an embedded submanifold of $\pi^{-1}_{12}(R_{12}) \times \pi^{-1}_{23}(R_{23})$. The latter is 
in turn an embedded submanifold of 
$S(\overline Q_1 \star Q_2) \times S(\overline Q_2 \star Q_3)$. Therefore it is enough to prove
that the projection map
$$
R_{12} \star_{Q_2} R_{23} \to S(\overline Q_1 \star Q_2) \times S(\overline Q_2 \star Q_3)/\sim
$$
is one-one.
Suppose that $[\beta_{12}', \beta_{23}'] = [\beta_{12},\beta_{23}]$ in  
$$
S(\overline Q_1 \star Q_2) \times S(\overline Q_2 \star Q_3)/\sim \, = \overline Q_1 \star \bar \Delta_{Q_2}^\star \star Q_3,
$$
i.e., 
\be\label{eq:1234'=1234}
(\beta_{12}', \beta_{23}')  =  \phi_{\vec E_{\bar 1 2\bar 23}}^t\left(\beta_{12}, \beta_{23}\right) 
=  \left (\phi_{\vec E_{\bar 1 2}}^t(\beta_{12}),\phi_{ \vec E_{\bar 2 3}}^t( \beta_{23})\right) 
\ee
for some constant $t$ by the definition of the equivalence relation $\sim$. Therefore 
$$
\beta_{12}' = \phi_{\vec E_{\bar 1 2}}^t( \beta_{12}),\, \quad  \beta_{23}' = \phi_{\vec E_{\bar 2 3}}^t( \beta_{23}).
$$
On the other hand, we have
$$
\beta_{12}, \, \beta_{12}' \in \pi^{-1}_{12}(R_{12}) \subset SQ_1 \times SQ_2,
$$
$$
\beta_{23}, \, \beta_{23}' \in \pi^{-1}_{23}(R_{23}) \subset SQ_2 \times SQ_3
$$
and $\pi_2(\beta_{12}) = \pi_2(\beta_{23})$, $\pi_2(\beta_{12}') = \pi_2(\beta_{23}')$. Recalling 
the definition
$$
\overline Q_1 \star Q_2 =  SQ_1 \times SQ_2/\sim,
$$
we obtain
$$
 \beta'_{12} = (e^{-t} \beta_1, e^t\beta_2).
$$
if we put
$$
\beta_{12} = (\beta_1, \beta_2).
$$
Since $R_{12}$ is a Legendrian submanifold of the one-form $[\theta_1,\theta_2]$, 
we have 
$$
[(\theta_1, \theta_2)]|_{R_{12}} = 0 \Longleftrightarrow (\theta_1 - \theta_2)|_{TSR_{12}} = 0.
$$
Similarly, we have
$$
 \beta'_{23} = (e^{-t} \beta_2, e^t\beta_3)
 $$
 and  $(\theta_2 - \theta_3)|_{TSR_{23}} = 0$
 if we put $\beta_{23} = (\beta_2, \beta_3)$.

Therefore we have derived
$$
(\beta_{12}',\beta_{23}') = ((e^{-t}\beta_1,e^t\beta_2),  (e^{-t} \beta_2,e^t\beta_3))
= \phi_{-\vec E_{\bar 1 2} \oplus \vec E_{\bar 2 3}}^t(\beta_{12}, \beta_{23}).
$$
This proves 
$$
[(\beta_{12}', \beta_{23}')] = [(\beta_{12},\beta_{23})]
$$
in $(\overline Q_1 \star Q_2) \star_{Q_2} (\overline Q_2 \star Q_3) = Q_1 \star \bar \Delta_{Q_2}^\star \star Q_3$
which finishes the proof of the proposition.
\end{proof}

The following is the fundamental proposition towards the construction of composition of Legendrian
correspondences.

\begin{prop}\label{prop:composable-pair} The map
$$
p_{\overline13}^\star|_{R_{12}\star_{Q_2} R_{23}}:R_{12}\star_{Q_2} R_{23}\to\overline{Q}_1 \star Q_3
$$
is a Legendrian immersion if the pair $(R_{12}, R_{23})$ is composable.
\end{prop}
\begin{proof}
The proof is similar to that of Lemma \ref{lem:graph}.

\emph{By the transversality}
$$
\pi_{1,\Delta_2,3} \pitchfork R_{12}\star  R_{23} 
$$
in  $(\overline  Q_1 \star Q_2) \star (\overline Q_2 \star Q_3)$, we
 can represent a tangent vector 
 $$
 \zeta_{12,23} \in T_{x_{12,23}}(R_{12}\star_{Q_2} R_{23})
 $$
by the form
$$
\zeta_{12,23} = \zeta_{12} \star \zeta_{23}
$$
as mentioned above, where $\zeta_{12} \star \zeta_{23}$ 
with  $\zeta_{12} \in T R_{12}$ and $\zeta_{23} \in T R_{23}$ 
is represented by the germ of curves at $x_{12,23}$
with $x_{12} = [(\beta_1,\beta_2)]$ and $x_{23} = [(\beta_2,\beta_3)]$ such that
$$
\pi_{12}(\zeta_{12}) = [(\zeta_1,\zeta_2)],  \quad \pi_{23}(\zeta_{23}) = [(\zeta_2',\zeta_3')]; \quad
\beta_2 = \beta_2'.
$$
(See Notation \ref{nota:x=}.):
Take any pair of lifts 
$$
\widetilde \zeta_{12} \in T(SQ_1 \times SQ_2), \quad
  \widetilde \zeta_{23} \in T(SQ_2 \times SQ_3)
$$
respectively thereof. We then take the germs of curves
\beastar
\gamma_{12} & = & (\gamma_1, \gamma_2): (-\epsilon,\epsilon) \to SQ_1 \times SQ_2,\\
\delta_{23} & = & (\delta_1, \delta_2) :  (-\epsilon,\epsilon) \to SQ_2 \times SQ_3
\eeastar 
representing $\widetilde \zeta_{12}, \,
\widetilde \zeta_{23}$ respectively, so that they satisfy
\be\label{eq:gamma2delta2(0)}
(\dot\gamma_2(0), \dot\delta_2(0))= (v,v).
\ee
We denote by $\gamma_{\bar 12}$ and $\delta_{\bar 23}$
the projections of $\gamma_{12}$ and $\delta_{23}$ to 
$\overline Q_1\star Q_2$ and $\overline Q_2 \star Q_3$, and then by
$\widetilde \gamma_{\bar 12}$ and $\widetilde \delta_{\bar 23}$ their liftings to
$S(\overline Q_1\star Q_2)$ and $S(\overline Q_2 \star Q_3)$, respectively.

We now evaluate the product Liouville one-form
$$
(\theta_{\bar 12},\theta_{\bar 23})\left(\dot {\widetilde \gamma}_{12}(0), \dot{\widetilde \delta}_{23}(0)\right)
$$
where $\theta_{\bar 12}$ and $\theta_{\bar 23}$ are the Liouville one-forms of
$S(\overline Q_1\star Q_2)$ and $S(\overline Q_2 \star Q_3)$ respectively.  By the representation
of $\gamma_{12}= (\gamma_1, \gamma_2)$ and that of $\delta_{23}= (\delta_2,\delta_3)$ with
$$
(\gamma_1(t), \gamma_2(t)) \in SQ_1 \times SQ_2,\quad(\delta_2(t), \delta_3(t)) \in SQ_2 \times SQ_3,
$$
that satisfy \eqref{eq:gamma2delta2(0)}, we have
\be\label{eq:dotgamma2delta2}
\dot \gamma_2(0) = v,  \quad \dot \delta_2(0) = v.
\ee
Then we evaluate
\beastar
\theta_{13}(\dot\gamma_1(0),\dot\delta_3(0))
& = & -\theta_1(\dot\gamma_1(0))+\theta_3(\dot\delta_3(0))\\
& = & \big(-\theta_1(\dot\gamma_1(0))+\theta_2(\dot\gamma_2(0))\big)+
\big(-\theta_2(\dot\delta_2(0))+\theta_3(\dot\delta_3(0))\big)\\
& = & \theta_{12}(\widetilde\zeta_{12})+\theta_{23}(\widetilde\zeta_{23})
= 0+0 = 0
\eeastar
where the vanishing of each summand in the penultimate sum follows since 
we have
$$
\dot \gamma_{12}(0) \in \overline \xi_1 \star \xi_2 =\ker\theta_{12}, \quad  \dot \delta_{23}(0) 
\in \xi_2 \star \xi_3=\ker\theta_{23}.
$$
This in particular proves
$$
T (R_{12} \circ R_{23}) \subset \xi_{13}
$$
which shows that the $\pi_{1,\Delta_2,3}$ is isotropic.

Finally we show the immersion property. We first note that
the transversality hypothesis implies
\bea\label{eq:transversality}
&{} & T(SQ_1 \times SQ_2 \times SQ_2 \times SQ_3) \nonumber\\
& = & d\pi_{\bar 12, \bar 2 3}^{-1}((R_{12}  \star R_{23})  +
T(SQ_1 \times \Delta_{Q_2} \times SQ_3).
\eea
It also implies that 
$$
R_{12,23}: = \pi_{1,\Delta_2,3}^{-1}(R_{12}  \star_{Q_2} R_{23}),
$$
it is a smooth manifold and the map
$$
\widetilde \Pi_{13}: =  \pi_{13}^\star \circ \pi_{1,\Delta_2,3}\Big|_{R_{12,23}}
$$
is smooth. Then it follows from
$$
\widetilde\Pi_{13}=p_{\overline13}^\star|_{R_{12}\star_{Q_2} R_{23}}\circ\pi_{\overline12\overline23}|_{R_{12,23}}
$$
that
$p_{\overline13}|_{R_{12}\star_{Q_2} R_{23}}$ is an immersion if and only if
\be\label{eq:kerneltildePi13}
\ker \widetilde \Pi_{13} = \span\{\vec E_{\bar 12 \bar 23}\}.
\ee
Therefore it is enough to show \eqref{eq:kerneltildePi13}. 

\begin{lem} The identity  \eqref{eq:kerneltildePi13} holds.
\end{lem}
\begin{proof}
 By definition, we also have
\bea\label{eq:TR1223}
TR_{12,23} & = & d\pi_{\bar 1, \Delta_2,3}^{-1}(R_{12} \star R_{23}) 
\bigcap   T(SQ_1 \times \Delta_{Q_2} \times SQ_3) \nonumber \\
& = & (V_{12} \times V_{23}) \bigcap 
 T(SQ_1 \times \Delta_{Q_2} \times SQ_3)
\eea
where we put
$$
V_{12} = \pi_{12}^{-1}(R_{12}), \quad V_{23} =  \pi_{23}^{-1}(R_{23}).
$$
(We mentioned that $V_{12}$ (resp. $V_{23}$ is the tangent space of the 
cylindrical Lagrangian submanifold $L_{12}$ (resp. $L_{23}$) associated to $R_{12}$ (resp. $R_{23}$) on $SQ_1\times SQ_2$ 
(resp. on $SQ_2 \times SQ_3$)
with respect to the canonical symplectic forms $d\theta_{\bar 1 2}$ (resp. with respect to $d\theta_{\bar 2 3}$)).
We have
$$
\dot \gamma_2(0) = v,  \quad \dot \delta_2(0) = v
$$
from  \eqref{eq:dotgamma2delta2}. 
Therefore we have
\beastar
\zeta_{12,34} & = &  \left(-\dot \gamma_1(0),\dot \gamma_2(0), -\dot \delta_2(0),\dot \delta_3(0)\right) \\
& = &\left (-\dot \gamma_1(0), v, v,\dot \delta_3(0)\right)\\
& \in & \span \{\vec E_1(\beta_1)\} \oplus T\bar \Delta_{Q_2}^S \oplus \span\{ \vec E_3(\beta_3)\}.
\eeastar
Furthermore we also have
$$
\vec E_{\bar 1 2,\bar 2 3} = -\vec E_1(\beta_1) \oplus \vec E_2(\beta_2) \oplus (-\vec E_2(\beta_2)) \oplus \vec E_3(\beta_3).
$$
Now suppose 
$$
d \widetilde \Pi_{13}(\zeta_{12,23}) = 0, \quad \zeta_{12,23} \in TR_{12,23}.
$$
  On the other hand, we can explicitly write
\beastar
d \widetilde \Pi_{13}(\zeta_{12,23}) & = & 
d \widetilde \Pi_{13} \left (-\dot \gamma_1(0), v,v,\dot \delta_3(0)\right) \\
& = & \left (-\dot \gamma_1(0),\dot \delta_3(0)\right) \equiv 0 \mod \span\{\vec E_{\bar 13}\}.
\eeastar
Combining the two, we obtain
$$
 \left (-\dot \gamma_1(0),\dot \delta_3(0)\right) = c'  ( -\vec E_1(\beta_1), \vec E_3(\beta_3)) 
 = c' \vec E_{\bar 13}(x_{12}\star x_{23}).
 $$
Substituting this back into $\zeta_{12,34}$ above, we have derived
 $$
\zeta_{12,34} =  (- c'\vec E_1(\beta_1), v ,v, c' \vec E_3(\beta_3)).
$$
 It remains to show $c = c'$. We rewrite
 \beastar
 \zeta_{12,34} 
 & =  &  c'(-\vec E_1(\beta_1), -\vec E_2(\beta_2) ,- \vec E_2(\beta_2)) ,  \vec E_3(\beta_3)) \\
 &{}&  +  (0, v+c'\vec E_2(\beta_2) ,v+c' \vec E_2(\beta_2), 0)\\
 & \equiv &  \left(0, v+c'\vec E_2(\beta_2) ,v+c' \vec E_2(\beta_2),0\right)\mod \span\{\vec E_{\bar 12, \bar 23}\}.
 \eeastar
Therefore  from the fact that $R_{12}$ (and $R_{23}$) is Legendrian (and so maximally isotropic
with respect to $d[\vec \theta_{\bar 12,\bar 23}]|_{\xi_{\bar 12,\bar 23}}$, we derive 
\beastar
&{}& \left(0, v+c'\vec E_2(\beta_2) ,v+c' \vec E_2(\beta_2),0\right)  \\
& \in & (d\pi_{\bar 12, \bar 2 3}^{-1}((R_{12}  \star_{Q_2} R_{23}))^{^{d\theta_{\bar 12, \bar 23}}} \bigcap 
\left(T(SQ_1 \times \Delta_{Q_2} \times SQ_3) \right)^{d\theta_{\bar 12, \bar 23}}
\eeastar
where the first inclusion follows from the Lagrangian property of $\pi_{\bar 12, \bar 2 3}^{-1}((R_{12}  \star_{Q_2} R_{23}))$ and the second follows from the Lagrangian property of $\Delta_{Q_2}$
with respect to $d\theta_{2\bar 2} = d\theta_2 \oplus (-d\theta_2)$. Obviously we have
\beastar
&{}&  (d\pi_{\bar 12, \bar 2 3}^{-1}((R_{12}  \star_{Q_2} R_{23}))^{^{d\theta_{\bar 12, \bar 23}}} \bigcap 
\left(T(SQ_1 \times \Delta_{Q_2} \times SQ_3) \right)^{d\theta_{\bar 12, \bar 23}}\\
& = &  \left(d\pi_{\bar 12, \bar 2 3}^{-1}((R_{12}  \star_{Q_2} R_{23})  +
T(SQ_1 \times \Delta_{Q_2} \times SQ_3) \right)^{d\theta_{\bar 12, \bar 23}}.
\eeastar
However the last set is equal to $\{0\}$ by the transversality hypothesis \eqref{eq:transversality}
$$
\pi_{1,\Delta_2,3} \pitchfork R_{12}\star  R_{23} 
$$
in  $(\overline  Q_1 \star Q_2) \star (\overline Q_2 \star Q_3)$, which gives rise to
$$
 \left(0, v+c'\vec E_2(\beta_2) ,v+c' \vec E_2(\beta_2),0\right)=0
$$
and
$$
v=-c'\vec E_2(\beta_2)
$$
which proves $\zeta_{12,34} = c'\vec E_{\bar 12, \bar 23}$.
This finishes the proof of the lemma.
\end{proof}

Finally, the proof of Proposition \ref{prop:composable-pair} is completed.
\end{proof}
We denote $p_{\overline13}^\star|_{R_{12}\star_{Q_2} R_{23}}$ by $p_{\overline13}^\star$ with a slight abuse of a notation.

Combining the two propositions, we have proved
\begin{cordefn}[Composition of Legendrian correspondences]\label{defn:composition}
The map 
$$
p_{\overline13}^\star : R_{12}\star_{Q_2} R_{23} \to \overline Q_1 \star Q_3
$$
is a Legendrian immersion for any composable pair.
We denote the image thereof  by 
$$
R_{12}\circ R_{23} : =  p_{\overline13}^\star(R_{12}\star_{Q_2} R_{23})
$$
and call it \emph{the composition
of $R_{12}$ and $R_{23}$}. If this map is embedded in addition, we call the
pair \emph{strongly composable}.
\end{cordefn}

\section{Generic composition of Legendrian correspondences}

Now we reach an important result of the present paper, which makes a stark difference from
the calculus of Lagrangian correspondence in symplectic geometry.
The following theorem then shows that the strong composability is also a generic phenomenon.
The entire section of this  will be occupied by the proof of this theorem.

\begin{thm}\label{thm:strongly-composable}
 For any composable pair $(R_{12}, R_{23})$, there exists a pair of $C^\infty$-small Legendrian isotopies
of each factor, say, $(R_{12}',R_{23}')$ which become strongly composable.
\end{thm}

It is well-known that any embedded Legendrian isotopy can be realized by an ambient contact Hamiltonian isotopy.
Therefore we can achieve the genericity of embeddedness of Legendrian embeddings by investigating 
perturbation of contact vector fields. (See \cite[Section 3.6]{oh:book1} e.g., for a detailed proof of the
symplectic counterpart  for the exact Lagrangian submanifolds which can be easily modified for
the current Legendrian case.)

For this purpose, we consider the contact projection 
    $$
    p_{\bar 13}^\star: (\overline Q_1\star Q_2)\star_{Q_2} (\overline Q_2\star Q_3)\to \overline Q_1\star Q_3
    $$
and the commutative diagram
\be
\xymatrix{
   SQ_1\times (SQ_2 \times_{Q_2} SQ_2) \times SQ_3 \ar[d] \ar[r] 
   & SQ_1\times SQ_3 \ar[d]^{\pi_{\bar13}}\\
        (\overline Q_1\star Q_2)\star_{Q_2} (\overline Q_2 \star Q_3) \ar[r]_>>>>>>>{p_{\bar13}^\star}              & \overline Q_1\star Q_3
}
\ee
With this being mentioned, we will indeed prove the following stronger version of 
Theorem \ref{thm:strongly-composable}.

\begin{thm}\label{thm:psi-perturbation}
Let  $(R_{12}, R_{23})$ be any composable pair. 
Then there exists a pair of sufficiently $C^\infty$ small contactomorphism
$$
(\varphi,\psi)\in \Cont (\overline Q_1\star Q_2)\times \Cont (\overline Q_2\star Q_3)
$$ 
such that $(\varphi(R_{12}),\psi(R_{23}))$ becomes strongly composable. 
\end{thm}

\subsection{Parametric description of the self-intersection set of $\varphi(R_{12}) \circ \psi(R_{23})$}

By our definition the Legendrian correspondence that requires it to be embedded, 
we want the composition $\varphi(R_{12}) \circ \psi(R_{23})$ to be embedded for 
$(\varphi,\psi)$ that is contained in a $C^\infty$ neighborhood 
$\CU$ of the identity of $\Cont(\overline Q_1 \star Q_2)\times \Cont(\overline Q_2 \star Q_3)$ 
which is sufficiently small. We will scrutinize 
the statement
$$
\text{\rm ``$\varphi(R_{12}) \circ \psi(R_{23})$ to be embedded.''}
$$
by unravelling the definition of the composition appearing in this statement. 

Recall the definition of contact diagonal 
$\Delta_{Q_2}^\star \subset \overline Q_2 \star Q_2$ and the contact product
$$
\overline Q_1 \star \bar \Delta_{Q_2}^\star \star Q_3 \subset Q_1 \star (\overline Q_2 \star Q_2)\star Q_3.
$$
Then we have the composition 
\be\label{eq:R12oR23}
R_{12} \circ R_{23} = p_{\bar13}^\star (R_{12} \star_{Q_2} R_{23})
\ee
 
$$
p_{\bar13}^\star : \overline Q_1 \star \bar \Delta_{Q_2}^\star \star Q_3  \to \overline Q_1 \star Q_3.
$$
Therefore the only source of the failure of the embedding property arises from the failure of
the restriction map
$$
p_{\bar13}^\star \Big|_{R_{12} \star_{Q_2} R_{23}}
$$
to be one-to-one: We already know  that the last map is a Legendrian immersion. 
We will now show that  the map
$$
p_{\bar13}^\star\Big|_{\varphi(R_{12}) \star_{Q_2} \psi(R_{23})}
$$
becomes one-to-one, after making the perturbation $(\varphi(R_{12}), \psi(R_{23}))$ 
contactomorphisms $\varphi \in \Cont (\overline Q_1 \star Q_2)$ 
and $\psi \in \Cont (\overline Q_2 \star Q_3)$ which can be made arbitrarily 
$C^\infty$-close to the identity.

\begin{defn}[Parameterized intersection]\label{defn:para-intersection}
We denote by $\mathfrak{O}^2_{(\varphi,\psi)}(R_{12},R_{23})$ 
the parameterized intersection set 
\be\label{eq:CO2}
\mathfrak{O}^2_{(\varphi,\psi)}(R_{12},R_{23})
 : = \{(\varphi,\psi)\} \times \left((R_{12} \star R_{23})^2\setminus\Delta_{R_{12} \star R_{23}}\right)
\ee
and form the fiber bundle
\be\label{eq:D2R12R13}
\mathfrak{O}^2(R_{12},R_{23})
 :=  \bigcup_{(\varphi,\psi) \in \Cont_{12} \times \Cont_{23}} \{(\varphi,\psi) \} 
\times \mathfrak{O}^2_{(\varphi,\psi)}(R_{12},R_{23})
\ee
over 
$$
\Cont_{12} \times \Cont_{23} := \Cont (\overline Q_1 \star Q_2)\times \Cont (\overline Q_2 \star Q_3).
$$
\end{defn}
Then we  consider the fiberwise map
$$
\Upsilon:\mathfrak{O}^2(R_{12},R_{23}) \to \left((\overline Q_1 \star Q_2) \star (\overline Q_2 \star Q_3)\right)^2\times
(\overline Q_1\star Q_3)^2
$$
defined by 
\be\label{eq:Upsilon}
\Upsilon(\varphi,\psi,x_{12,23}, y_{12,23})
= \Upsilon_{(\varphi,\psi)}(x_{12,23}, y_{12,23})
\ee
where the fiber map 
$$
\Upsilon_{(\varphi,\psi)}:\mathfrak{O}^2_{(\varphi,\psi)}(R_{12},R_{23}) \to 
  \left((\overline Q_1 \star Q_2) \star (\overline Q_2 \star Q_3)\right)^2\times
(\overline Q_1\star Q_3)^2,
 $$
is given by the formula
\bea \label{eq:Upsilon-psi}
 \Upsilon_{(\varphi,\psi)}(x_{12,23}, y_{12,23}): = 
\big (\,(x_{12,23}^{\varphi,\psi}, y_{12,23}^{\varphi,\psi}), \, (\pi^\star_{\bar13}(x_{12,23}^{\varphi,\psi}),
\pi^\star_{\bar13}(y_{12,23}^{\varphi,\psi}))\,\big),
\eea
where we denote
$$
x_{12,23} = x_{12} \star y_{23}, \quad y_{12,23} = y_{12} \star y_{23}.
$$
(See \eqref{eq:xstarx'} for the definition of the right hand sides.)
Then we define
\be\label{eq:xphipsiy}
x_{12,23}^{\varphi,\psi}= (\varphi\star\psi)(x_{12,23}),\quad y_{12,23}^{\varphi,\psi}
= (\varphi \star\psi)(y_{12,23})
\ee
for $x_{12},\, y_{12} \in R_{12}, \quad x_{23},\, y_{23} \in R_{23}$.

\subsection{Preparation for the off-shell framework for the Fredholm theory}

For the Fredholm theory, we first attempt to take  the full product
\be\label{eq:Cont1223}
\Cont(Q_{\bar 12,\bar 23}): = \{(\varphi,\psi) \in \Cont(\overline Q_1 \star Q_2)\times \Cont (\overline Q_2 \star Q_3)\}
\ee
as the perturbation parameter space.

\begin{defn}[$\mathfrak{Int}^{\mathfrak P}(R_{12},R_{13})$]
Let $\mathfrak P \subset \Cont(Q_{\bar 12,\bar 23})$ be a given subset.
Consider the $\mathfrak{P}$-family of the $\star$ fiber product
$$
\varphi(R_{12}) \star_{Q_2} \psi(R_{23}), \quad (\varphi,\psi) \in \mathfrak{P}
$$
over $Q_2$ parameterized by the set $\mathfrak{P}$. We call the preimage 
\be\label{eq:self-intersection}
\mathfrak{Int}^{\mathfrak P}(R_{12},R_{23}) : = \Upsilon^{-1}\left((\overline Q_1
 \star \bar \Delta_{Q_2}^\star \star Q_3)^2\times\D_{\overline Q_1\star Q_3}\right) 
\ee
in $ \mathfrak{O}^2(R_{12},R_{23})$ \emph{the $\mathfrak P$-family intersections} of $R_{12}$ and 
$R_{23}$.
\end{defn}
We want this set to carry the structure of a (infinite dimensional) Frechet manifold.
To achieve \emph{smoothness of parameterized intersection set} \eqref{eq:self-intersection}, it 
would follow from the transversality
\be\label{eq:diagonal-transversality}
\Upsilon \pitchfork \left((\overline Q_1
 \star \bar \Delta_{Q_2}^\star \star Q_3)^2\times\D_{\overline Q_1\star Q_3}\right)
\ee
in 
$$
\left((\overline Q_1 \star Q_2) \star (\bar Q_2 \star Q_3)\right)^2 \times (\overline Q_1\star Q_3)^2.
$$
\emph{Assuming this transversality}, we keep proceeding and consider  the projection map
$$
\Pi: \Upsilon^{-1}((\overline Q_1 \star \bar \Delta_{Q_2}^\star \star Q_3)^2\times\D_{\overline Q_1\star Q_3}) \to \mathfrak{P}
$$
which is the restriction of the first projection 
$$
\pi_1:  \mathfrak{P} \times((R_{12} \star R_{23})^2\setminus\Delta_{R_{12} \star R_{23}})\to 
\mathfrak{P}
$$
to $\mathfrak{Int}^{\mathfrak P}(R_{12},R_{23})$. We put
$$
\mathfrak{Int}_{(\varphi,\psi)}^{\mathfrak P}(R_{12},R_{23}): = \mathfrak{Int}^{\mathfrak P}(R_{12},R_{23}) \cap \Pi^{-1}(\varphi,\psi)
$$
where 
$$
\mathfrak{Int}^{\mathfrak P}(R_{12},R_{23}) = \bigcup_{(\varphi,\psi) \in \mathfrak P} \{(\varphi,\psi)\} \times
\mathfrak{Int}_{(\varphi,\psi)}^{\mathfrak P}(R_{12},R_{23}).
$$
Recall that by definition, we have
$$
\varphi(R_{12}) \circ \psi(R_{23}) =\Pi^\star_{\bar13}
\left(\Upsilon_{(\varphi,\psi)}^{-1}\left((\overline Q_1 \star \bar \Delta_{Q_2}^\star \star Q_3)^2\times(\overline Q_1\star Q_3)^2\right)\right)
 \subset \overline Q_1\star Q_3
$$
and
\bea\label{eq:self}
\text{\rm Self}(\varphi(R_{12}) \circ \psi(R_{23})) 
& = & \Pi^\star_{\bar13}\left(\Upsilon_{(\varphi,\psi)}^{-1}\left((\overline Q_1 \star \bar \Delta_{Q_2}^\star 
\star Q_3)^2\times\D_{\overline Q_1\star Q_3}\right)\right) 
\nonumber\\
& = & \Pi^\star_{\bar13}\left(\mathfrak{Int}_{(\varphi,\psi)}^{\mathfrak P}(R_{12},R_{23})\right) \subset \overline Q_1\star Q_3
\eea
where 
$$
\Pi^\star_{\bar13}:= \pi_{\overline Q_1\star Q_3}\circ (p_{\bar13}^\star,p_{\bar13}^\star) 
$$
for $\pi_{\overline Q_1\star Q_3}:(\overline Q_1\star Q_3)^2\to\overline Q_1\star Q_3$ the usual projection.

Therefore we would like to show 
\be\label{eq:self-intersection=}
\mathfrak{Int}_{(\varphi,\psi)}^{\mathfrak P}(R_{12},R_{23}) = \emptyset
\ee
for each regular values $(\varphi,\psi)$ of $\Pi$. This would already have completed the proof of 
Theorem \ref{thm:psi-perturbation} if nothing had gone wrong in the above discussion.

\section{Perturbations by non-liftable contactomorphisms}
\label{sec:nonliftable}

The upshot is that the  transversality \eqref{eq:diagonal-transversality} might not hold 
over the full  group $\Cont (Q_{\bar 12,\bar 23}) $ given in \eqref{eq:Cont1223} in our first attempt,
unless we restricted the choice of the group elements to some smaller subset thereof.

We will denote the aforementioned subset by 
$$
\mathfrak{P} \subset \Cont(Q_{\bar 12,\bar 23}).
$$ 
We will shortly give the precise description of this subset that we are going to use.

\subsection{Openness of the set of non-liftable contactomorphisms}

To deal with the transversality issue of $\mathfrak{Int}_{(\varphi,\psi)}^{\mathfrak P}(R_{12},R_{23})$, 
identifying a correct perturbation space  as a Frechet manifold is the first step.

We recall here the definition of the subgroup
$$
\Cont(Q) \star \Cont(Q') \subset \Cont(Q \star Q')
$$
from Section \ref{sec:star-product}.

\begin{defn}\label{defn:liftable} A \emph{liftable contactomorphism} $\varphi$  
of $Q \star Q'$  is defined to be one of the form
$$
\varphi = \varphi_1 \star \varphi_2 : =  [(d\varphi_1^{-1})^*,(d\varphi_2^{-1})^*],
$$
for a pair $(\varphi_1, \varphi_2)$ of contactomorphisms of $Q, \, Q'$
inducing a symplectomorphism $S\varphi_i = (d\varphi_i^{-1})^*:SQ\to SQ$ and similarly for $Q'$.
We put
$$
\Cont^{\text{\rm lft}}(Q\star Q'): = \Cont(Q,\xi) \star \Cont(Q',\xi').
$$
\end{defn}

 We will see from the way our proof of generic transversality below goes 
that it is essential to involve \emph{non-liftable} contactomorphisms to achieve our purpose
of establishing the generic strong composability. 
For this purpose, we will show  the following openness.

\begin{prop}\label{prop:open} The subset
$$
\Cont^{\text{\rm nlft}}(\overline Q \star Q')\subset \Cont(\overline Q \star Q')
$$
is open. In particular, $\Cont^{\text{\rm lft}}(\overline Q \star Q')$ is a closed subgroup of
$\Cont(\overline Q \star Q')$.
\end{prop}
\begin{proof} 
We first note the non-liftability of $\varphi \in \Cont(\overline Q \star Q')$
implies that the associated $\mathbb R_+$-equivariant lift 
$\widetilde{\varphi} = S\varphi$ on the symplectization $SQ \times SQ'$ can be written as
	$$
\widetilde \varphi(\a_1,\a_2)=(\varphi_1(\a_1,\a_2),\varphi_2(\a_1,\a_2))
	$$
	and the dependence on either of $\a_i$ on each component
	 $\varphi_i$ for $i=1, \, 2$ cannot be removed, i.e.,
$\widetilde \varphi$ cannot be written into a `triangular form'. 
	This dependence condition can be encoded into the following existence condition:
The set
\be\label{eq:density}
\left\{(x,y) \in (\overline Q \star Q')^2 \, \Big|\, 
\begin{cases}
(\pi^\star_1\times\eta)(x) = (\pi^\star_1\times\eta)(y),\\
\pi^\star_2(x)\neq\pi^\star_2(y),\\
\pi_1^\star\varphi(x)\neq\pi_1^\star\varphi(y)
\end{cases}
\right\} 
\ee
should be non-empty, where we recall $\eta: \overline Q \star Q' \to \R$ is
the $\R$-coordinate function defined in Definition \ref{defn:eta}.
Then it is enough to show that
this condition is an open condition for $\varphi$ in $C^\infty$ topology of 
$\Cont(\overline Q \star Q')$.

Suppose $\varphi \in \Cont(\overline Q \star Q')^{\text{\rm nlft}}$.
Then there is a point $(x,y)$ such that
\be\label{eq:defining}
\begin{cases}
(\pi^\star_1\times\eta_1)(x)=  (\pi^\star_1\times\eta_1)(y) \quad &\text{\rm in }\, Q \times \R \\
\pi^\star_2(x)\neq\pi^\star_2(y) \quad & \text{\rm in }\, Q' \\
\pi_1^\star\varphi(x)\neq\pi_1^\star\varphi(y) \quad & \text{\rm in }\,  Q
\end{cases}
\ee
The second condition implies $p_2 \neq p_2'$ and
 $x \neq y$  in $\overline Q \star Q'$, in particular, if we put
$$
p_2: = \pi_2^\star(x), \, p_2' : = \pi_2^\star(y).
$$
The first condition is obviously equivalent to
$$
q_1 = q_1'  \quad \& \quad \eta(x) = \eta(y).
$$ 
where we put $ q_1:= \pi^\star_1(x)$ and $q_1': =   \pi^\star_1(y)$.

We put $q_2 := \pi_1^\star\varphi(x)$ and $q_2': = \pi_1^\star\varphi(y)$.
 The third condition implies $q_2 \neq q_2'$ in $Q$ as well as
$\varphi(x) \neq \varphi(y)$ in $\overline Q\star Q'$. The latter in particular implies
that at least one of $q_2$ or $q_2'$ is different from $q_1= q_1'$.

WLOG, assume $q_2 \neq q_1$. Then we can
make a $C^\infty$-small perturbation of $\varphi$  to
$$
\varphi' = (\psi \star \id_{Q'}) \circ \varphi
$$
by composing $\varphi$ by a contactomorphism of the type $\psi \star \id_{Q'}$with $\psi \in \Cont(Q)$ sufficiently $C^\infty$-close to the identity $\id_{Q}$ that is
is supported in a small neighborhood  $U_2$ of $q_2$ such that $q_1 \in Q \setminus \overline U_1$.
We put $y' : = \varphi'(y)$.
Then it follows that all the defining conditions for the pair
$(x,y') \in  (\overline Q \star Q')^2$
are still satisfied, possibly except
$$
\eta(x) = \eta(y').
$$
But this $\eta$-coordinates can be adjusted by considering the projection of 
the flow $\vec E_{\bar12} = (-E_1,E_2)$ on $SQ \times SQ'$ which does not affect
the projection $\pi_i^\star$ for both $i = 1, 2$. 
This finishes the proof of the required openness.
\end{proof}

Recalling the standard fact between the tangent space $T_eG$ of a Lie group $G$
 its associated Lie algebra $\mathfrak g$ and $T_gG \cong TL_g(T_eG)$ at $g \in G$, 
 when we identify $(\mathfrak g, [\cdot, \cdot])$ with the set of \emph{left invariant vector fields} 
 equipped with their Lie bracket (see \cite{spivakI} for example),  we immediately obtain the following 
 corollary of Proposition \ref{prop:open} by applying the facts to $G=\Cont(\overline Q \star Q')$.
 
\begin{cor}\label{cor:closedness} We have
$$
T_\id \Cont^{\text{\rm nlft}}(\overline Q \star Q') = \mathfrak X(\overline Q\star Q'),
$$
and hence
    $$
    T_\varphi \Cont^{\text{\rm nlft}}(\overline Q \star Q') = (\varphi)_*(\mathfrak X(\overline Q\star Q'))
    $$
  at all $\varphi \in \Cont^{\text{\rm nlft}}(\overline Q \star Q')$.
  \end{cor}

\subsection{Smoothness of {$\mathfrak{Int}^{\mathfrak P}(R_{12}, R_{23})$}}

For the proof of \eqref{eq:self-intersection}, we follow the standard procedure of applying
 the Sard-Smale theorem.

The first matter of business is to prove smoothness of $\mathfrak{Int}^{\mathfrak P}(R_{12}, R_{23})$
by proving the following transversality of the map 
$$
\Upsilon:\mathfrak{O}^2(R_{12},R_{23}) \to \left((\overline Q_1 \star Q_2) \star (\overline Q_2 \star Q_3)\right)^2\times
(\overline Q_1\star Q_3)^2
$$
defined in \eqref{eq:Upsilon}, where we recall from \eqref{eq:D2R12R13} that  the domain is defined to be
 the fiber bundle
$$
\mathfrak{O}^2(R_{12},R_{23})
 = \bigcup_{(\varphi,\psi) \in \Cont_{12} \times \Cont_{23}} \{(\varphi,\psi) \} 
\times \mathfrak{O}^2_{(\varphi,\psi)}(R_{12},R_{23}).
$$

\begin{prop} The map $\Upsilon$ is transverse to
 $\left(\overline Q_1 \star \bar \Delta_{Q_2}^\star \star Q_3\right)^2\times\D_{\overline Q_1\star Q_3}$.
\end{prop}
\begin{proof}
We decompose the derivative
\beastar
&{}& D\Upsilon(\varphi,\psi; x_{12,23}, y_{12,23}) \\
& = &
D_{(\varphi,\psi)}\Upsilon(\varphi,\psi; x_{12,23}, y_{12,23}) +  D' \Upsilon(\varphi,\psi; x_{12,23}, y_{12,23}).
\eeastar
We note 
$$
D' \Upsilon(\varphi,\psi; x_{12,23}, y_{12,23}) = D\Upsilon_{(\varphi,\psi)}(x_{12,23}, y_{12,23}).
$$

For the notational simplicity we denote by $\aleph = (\aleph_1,\aleph_2)$ a point given by
\be\label{eq:aleph}
\aleph : = {\Upsilon_{(\varphi,\psi)}(x_{12,23},y_{12,23})}
\ee
which is contained in the product
$$
\left(\overline Q_1 \star \bar \Delta_{Q_2}^\star \star Q_3\right)^2\times\D_{\overline Q_1\star Q_3}
\subset \left(\overline Q_1 \star \bar \Delta_{Q_2}^\star \star Q_3\right)^2\times
(\overline Q_1\star Q_2)^2.
$$
More specifically, we have
$$
\aleph
: = \big (\,(x_{12,23}^{\varphi,\psi}, y_{12,23}^{\varphi,\psi}), \, (\pi^\star_{\bar13}(x_{12,23}^{\varphi,\psi}),
\pi^\star_{\bar13}(y_{12,23}^{\varphi,\psi}))\,\big).
$$
from \eqref{eq:Upsilon-psi}. Therefore  we have
\be\label{eq:aleph12}
\aleph_1 =  (x_{12,23}^{\varphi,\psi}, y_{12,23}^{\varphi,\psi}), \quad 
\aleph_2 =  \big (\, \pi^\star_{\bar13}(x_{12,23}^{\varphi,\psi}),
\pi^\star_{\bar13}(y_{12,23}^{\varphi,\psi})\,\big).
\ee
We also write
$$
D_{(\varphi,\psi)}\Upsilon(x_{12,23}, y_{12,23})=: D_{(\varphi,\psi)}^{\bar 12,\bar 23}\Upsilon.
$$
Then   the map $D_{(\varphi,\psi)}^{\bar 12,\bar 23}\Upsilon$ is a linear map to
\beastar
T_{\aleph_1}
((\overline Q_1 \star Q_2) \star (\overline Q_2\star Q_3))^2\oplus 
T_{\aleph_2} (\overline Q_1\star Q_3)^2
\eeastar
whose domain is the vector space
$\mathfrak X (\overline Q_1 \star Q_2)\times\mathfrak X (\overline Q_2 \star Q_3)$. 
(See Lemma \ref{lem:frakX-lft}.)
For the simplicity of notation, we put
\bea\label{eq:V1234}
V_{1234}: = T_{\aleph_1} ((\overline Q_1 \star Q_2) \star (\overline Q_2\star Q_3))^2\oplus 
T_{\aleph_2} (\overline Q_1\star Q_3)^2
\eea
and write 
$$
D_{(\varphi,\psi)}^{\bar12,\bar34}\Upsilon:\mathfrak X (\overline Q_1 \star Q_2)\times\mathfrak X (\overline Q_2 \star Q_3)\to V_{1234}.
$$

We now compute this linearization map along the direction of $\varphi$, i.e. in the direction of
$$
 T_\varphi \Cont^{\text{\rm nlft}}(\overline Q \star Q') = (\varphi)_*(\mathfrak X(\overline Q\star Q'))
 \cong  \mathfrak X(\overline Q\star Q')
$$
(see Corollary \ref{cor:closedness}) at a point  
$$
(\varphi, \psi,x_{12,23}, y_{12,23}) \in \mathfrak{Int}^{\mathfrak P}(R_{12}, R_{23}).
$$
We represent the various relevant points 
$$
x_{12} = [\alpha_1,\alpha_2], \quad  x_{23} = [\alpha'_2,\alpha'_3]\quad\text{and}\quad
x_{12,23}=[\a_1,\a_2,\a_2',\a_3']
$$
$$
y_{12} = [\b_1,\b_2], \quad  y_{23} = [\b'_2,\b'_3]\quad\text{and}\quad
y_{12,23}=[\b_1,\b_2,\b_2',\b_3']
$$
in terms of $\a_i,\a'_i,\b_i,\b'_i\in SQ_i$ for $i=1,2,\, 3$. The following lemma is easy to check whose
proof is omitted.

\begin{lem} For any contactomorphism $\varphi:  \overline Q_1 \star Q_2 \to  \overline Q_1 \star Q_2$,
we can express the image $\varphi(x_{12}) \in \overline Q_1 \star Q_2$
as
$$
\varphi(x_{12}) = [\varphi_1(\a_1,\a_2), \varphi_2(\a_1,\a_2)]
$$
for some $\R_+$-equivariant maps 
\be\label{eq:varphii}
\varphi_i: SQ_1 \times SQ_2 \to SQ_i, \quad i=1, \, 2.
\ee
\end{lem}
Using the lemma, we can rewrite the value $\Upsilon(\varphi,\psi,x_{12,23}, y_{12,23})$ as
\bea\label{eq:Upsilon-lifting}
&{}& \Upsilon(\varphi,\psi,x_{12,23}, y_{12,23})\nonumber\\
& = &\left([\varphi_1(\alpha_1,\a_2), \varphi_2(\a_1,\alpha_2)] \star [\psi_2(\a'_2, \a'_3),\psi_3(\a'_2, \a'_3)], 
\right.\nonumber  \\
& {} &
[\varphi_1(\b_1,\b_2), \varphi_2(\b_1,\b_2)]\star [\psi_2(\b'_2, \b'_3),\psi_3(\b'_2, \b'_3)],\nonumber\\
& {} &\left. ( [\varphi_1(\a_1,\a_2),\psi_3(\a_2',\a_3')],[\varphi_1(\b_1,\b_2),\psi_3(\b_2',\b_3')] \right)
\nonumber \\
& = &\left ([\varphi_1(\a_1,\a_2), \varphi_2(\a_1,\a_2), \psi_2(\a'_2, \a'_3),\psi_3(\a'_2, \a'_3)], \right.
\nonumber\\
& {} &
[\varphi_1(\b_1,\b_2), \varphi_2(\b_1,\b_2) ,\psi_2(\b'_2, \b'_3),\psi_3(\b'_2, \b'_3)],\nonumber\\
& {} & \left. [\varphi_1(\a_1,\a_2),\psi_3(\a_2',\a_3')],[\varphi_1(\b_1,\b_2),\psi_3(\b_2',\b_3')]\right).
\nonumber  \\
&{}&
\eea
Since the value of $ \Upsilon$ lies in
$$
(\overline Q_1 \star \bar  \Delta_{Q_2}^\star \star Q_3)^2\times\D_{\overline Q_1\star Q_3} \subset 
(\overline Q_1 \star \bar  \Delta_{Q_2}^\star \star Q_3)^2\times (\overline Q_1\star Q_3)^2,
$$
we have
\bea
& {} &\varphi_2(\a_1,\a_2)=\psi_2(\a'_2,\a'_3),\, \varphi_2(\b_1,\b_2)=\psi_2(\b'_2,\b'_3),\label{eq:ImageInDQ2}\\
& {} & \big(\varphi_1(\a_1,\a_2),\psi_3(\a_2',\a'_3)\big)
\sim \big(\varphi_1(\b_1,\b_2),\psi_3(\b_2',\b'_3)\big).
\label{eq:ImageInQ1Q3}
\eea
By definition, the equivalence relation
\eqref{eq:ImageInQ1Q3} means
\be\label{eq:constraint}
\varphi_1(e^t\a_1,e^t\a_2)=\varphi_1(\b_1,\b_2),\,\psi_3(e^t\a'_2,e^t\a'_3)=\psi_3(\b'_2,\b'_3)
\ee
for some $t$. Therefore we have
\bea\label{eq:varphix_12x_23}
& {} &
\varphi(x_{12})\star \psi(x_{23})\nonumber\\
& = &
[\varphi_1(\a_1,\a_2), \varphi_2(\a_1,\a_2), \psi_2(\a'_2, \a'_3),\psi_3(\a'_2, \a'_3)]\nonumber\\
& = &
[e^t\varphi_1(\a_1,\a_2), e^t\varphi_2(\a_1,\a_2), e^t\psi_2(\a'_2, \a'_3),e^t\psi_3(\a'_2, \a'_3)]\nonumber\\
& = &
[\varphi_1(\b_1,\b_2), e^t\varphi_2(\a_1,\a_2), e^t\psi_2(\a'_2, \a'_3),\psi_3(\b'_2, \b'_3)]
\eea
We note here the following obvious fact.
\begin{lem} For any pair $(\varphi,\psi)$, we have
    $$
x_{12,34}^{\varphi,\psi} \neq y_{12,34}^{\varphi,\psi}
$$
provided $(x_{12,23},y_{12,23}) \in \text{\rm Domain} \Upsilon_{(\varphi,\psi)}$.
\end{lem}
\begin{proof} We first recall the domain of the map $\Upsilon_{(\varphi,\psi)}$ is
$$
\left(R_{12}\star R_{23}\right)^2\setminus\Delta_{R_{12}\star R_{23}},
$$
and so 
\bea\label{eq:diffptsxy}
x_{12,34}\neq y_{12,34}\in (\overline Q_1\star Q_2)\star(\overline Q_2\star Q_3).
\eea
Moreover, by the definition \eqref{eq:xphipsiy}, $x_{12,34}^{\varphi,\psi}$ (respectively, $y_{12,34}^{\varphi,\psi}$) 
is nothing but the image of $x_{12,34}$ (respectively, of $y_{12,34}$) under the contactomorphism 
$$
\varphi\star\psi:(\overline Q_1\star Q_2)\star (\overline Q_2\star Q_3) \to (\overline Q_1\star Q_2)\star (\overline Q_2\star Q_3)
$$
by definition, which is, In particular, a one-to-one map. Combining the above, we derive
$$
x_{12,34}^{\varphi,\psi} = (\varphi\star \psi)(x_{12,34}) \neq (\varphi\star \psi)(_{12,34}) 
= y_{12,34}^{\varphi,\psi}.
$$
This finishes the proof.
\end{proof}

Then, combining \eqref{eq:varphix_12x_23} and this lemma,  we achieve the
inequlaity
$$
e^t\varphi_2(\a_1,\a_2)\neq \varphi_2(\b_1,\b_2).
$$
Then by further adjusting $\a_i$ and $\a'_i$ by multiplying suitable scaling constant $e^t$ thereto,
we can transform the  conditions  \eqref{eq:ImageInDQ2} and \eqref{eq:constraint} into the following
set of conditions
\be\label{eq:diffptinphi}
\begin{cases}
	\varphi_1(\a_1,\a_2)=\varphi_1(\b_1,\b_2),\,\psi_3(\a_2',\a_3')=\psi_3(\b_2',\b_3')\\
	\varphi_2(\a_1,\a_2)\neq \varphi_2(\b_1,\b_2),\,\psi_2(\a'_2,\a'_3)\neq \psi_2(\b'_2,\b'_3).
\end{cases}
\ee
In particular this implies
\be\label{eq:diffptindiagonal}
(\a_1,\a_2)\neq(\b_1,\b_2),\,(\a'_2,\a'_3)\neq(\b'_2,\b'_3).
\ee

To wrap up the proof of the proposition,  we need to show the following.

\begin{lem} Let $\aleph$ be as in \eqref{eq:aleph}. Then we have
\bea\label{eq:transversalitytoDiagonal} 
V_{1234}=\Image D_{(\varphi,\psi)}^{\bar 12, \bar 23}\Upsilon
+ T_{\aleph_1} (\overline Q_1\star \bar \Delta_{Q_2}^\star \star Q_3)^2
 +  T_{\aleph_2} \Delta_{\overline Q_1\star Q_3}.
\eea
\end{lem}
\begin{proof} We first show
\be\label{eq:neq}
x_{12}\neq y_{12}, \quad x_{23}\neq y_{23}.
\ee	
Suppose to the contrary that  $x_{12}\neq y_{12}$. Then we have
	$$
	e^\tau(\a_1,\a_2)=(\b_1,\b_2)
	$$
	with $\tau\neq0$ by \eqref{eq:diffptindiagonal}, and
		$$
	\varphi_1(\b_1,\b_2)=\varphi_1(e^\tau\a_1,e^\tau\a_2)= e^\tau\varphi_1(\a_1,\a_2)
	$$
where the multiplication by $e^\tau \in \R$ is the \emph{fiberwise action} induced by the Euler vector field.
On the other hand, the constraint $\varphi_1(\a_1,\a_2)=\varphi_1(\b_1,\b_2)$ in \eqref{eq:diffptinphi} implies
	$$
	e^\tau\varphi_1(\a_1,\a_2)=\varphi_1(\a_1,\a_2).
	$$
This implies $e^\tau=1$ which contradicts to $\tau\neq 0$, showing the claim. The same argument also shows $x_{23}\neq y_{23}$.
	
	We also recall the equalities
	$$
	T_\varphi\Cont^\text{nlft}(\overline{Q}_1\star Q_2)=\varphi_*\left(\overline{Q}_1\star Q_2)\right)=\mathfrak X(\overline{Q}_1\star Q_2)
	$$
from Corollary \ref{cor:closedness}. 
	
With these preparation made, we proceed with the proof of the proposition.
	Let $v_1,v_2\in T_*\left(\overline Q_1\star (Q_2\star\overline Q_2)\star Q_3\right)$ and $w\in T_*\left(\overline Q_1\star Q_3\right)$.
Recalling \eqref{eq:self-intersection}, \eqref{eq:V1234} and the decomposition
\bea\label{eq:decompostion}
D_{(\varphi,\psi)}^{\bar 12,\bar 23}\Upsilon= D_{\varphi}^{\bar 12, \bar 2 3}\Upsilon+ D_{\psi}^{\bar 12,\bar 23}\Upsilon,
\eea
by which
we have only to investigate the nature of $D_{\varphi}^{\bar 12,\bar 23}\Upsilon$ and
$D_{\psi}^{\bar 12,\bar 23}\Upsilon$ separately. 
	
Using Corollary \ref{cor:2point-ampleness}, 
we can construct a contact vector field $X\in\varphi_*\left(\mathfrak X(Q_1,Q_2)\right)$ with
	$$
	\begin{cases}
		d_{\varphi(x_{12})}\pi^\star_1(X)=d\pi^\star_1(w)\\
		d_{\varphi(y_{12})}\pi^\star_1(X)=0
	\end{cases}
	$$
	and
	$$
	\begin{cases}
		d_{\varphi(x_{12})}(\pi^\star_2\times\eta)(X)=d(\pi^\star_2\times\eta_2)(v_1)\\
		d_{\varphi(y_{12})}(\pi^\star_2\times\eta)(X)=d(\pi^\star_2\times\eta_2)(v_2).
	\end{cases}
	$$
	Here $\pi^\star_2$ is the first $Q_2$ component projection of $\overline Q_1\star( Q_2\star\overline Q_2)\star Q_3$ and $\eta_2$ is the $\R$-coordinate function of $Q_2\star\overline Q_2$ defined in Definition \ref{defn:eta}.
	
	Similarly, we construct a contact vector field $Y\in\psi_*\mathfrak X(Q_2,Q_3)$ satisfying
	$$
	\begin{cases}
		d_{\psi(x_{23})}(\pi^\star_3\times\eta)(Y)=d(\pi^\star_3\times\eta)(w)\\
		d_{\psi(y_{23})}(\pi^\star_3\times\eta)(Y)=0.
	\end{cases}
	$$
	and
	$$
	\begin{cases}
		d_{\psi(x_{23})}\pi^\star_2(Y)=0\\
		d_{\psi(y_{23})}\pi^\star_2(Y)=0.
	\end{cases}
	$$
	In other words, we have
	$$
	\begin{cases}
		X(\varphi(x_{12}))=(d\pi_1^\star(w),d\pi_2^\star(v_1),d\eta_2(v_1))\\
		X(\varphi(y_{12}))=(0,d\pi_2^\star(v_2),d\eta_2(v_2))
	\end{cases}
	$$
in the $\widetilde{Q}_1\times\widetilde{Q}_2\times\R$ coordinates,	and
	$$
	\begin{cases}
		Y(\psi(x_{23}))=(0,d\pi_3^\star(w),d\eta(w))\\
		Y(\psi(y_{23}))=0
	\end{cases}
	$$
in the $\widetilde{Q}_2\times\widetilde{Q}_3 \times\R$ coordinates.

The above holds for all $v_1,v_2$ and $w$ thus we have shown both
	\be\label{eq:1st-inclusion}
	\Image D_{(\varphi,\psi)}^{\bar 12,\bar 23}\Upsilon
	 \supseteq  \Big(\vec 0\star T_*\left(\widetilde{Q}_2\times\{*\}\times\R\right)\star\vec 0\Big)^2
	\ee
	and
	\be\label{eq:2nd-inclusion}
	\Image D_{(\varphi,\psi)}^{\bar 12,\bar 23}\Upsilon
	 \supseteq 
	T_*\left(\overline Q_1\star Q_3\right)\times\vec 0:
	\ee
Here
	\be\label{eq:3rd-inclusion}
	\Big(\vec 0\star T_*\left(\widetilde{Q}_2\times\{*\}\times\R\right)\star\vec 0\Big)^2
	\subset T_{\aleph_1}\left(\overline{Q}_1\star(Q_2\star\overline{Q}_2)\star Q_3\right)^2
	\ee
	is the subset consisting of tangent vectors such that only $Q_2\star\overline {Q}_2$ component exists, and
	$$
	T_*\left(\overline Q_1\star Q_3\right)\times\vec0\subset T_{\aleph_2}\left(\overline Q_1\star Q_3\right)^2.
	$$
We recall \eqref{eq:aleph12} for the specific expressions of $(\aleph_1,\aleph_2) = \aleph$.	
	Since we already know
	$$
	\bar\D^\star_{Q_2}\pitchfork ( \widetilde{Q}_2\times\{*\}\times \R)
	$$
	in $Q_2\star \overline{Q}_2\cong\widetilde{Q}_2\times\widetilde{Q}_2\times\R$, 
	the above three inclusions \eqref{eq:1st-inclusion}-\eqref{eq:3rd-inclusion} finish the proof of the lemma.
	\end{proof}
This finishes the proof of the proposition.
\end{proof}

\subsection{Application of the Sard-Smale theorem}

Now we go back to the universal moduli space and its projection
$$
\Pi: \mathfrak{Int}^{\mathfrak P}(R_{12},R_{23}) \to \mathfrak{P}.
$$
 Now we are ready to apply the Sard-Smale theorem.
Since the fiber of $\Pi$ is finite dimensional, the following Fredholm property
is automatic.

Here we denote
	$$
	\Delta:=(\overline Q_1\star\overline\Delta^\star_{Q_2}\star Q_3)^2\times\Delta_{\overline Q_1\star Q_3}
	$$
	for the notational simplicity.

\begin{prop} Take any Banach completion of $\mathfrak{P}$
such as in $C^k$-topology $k \geq 1$. Then
the map $\Pi: \Upsilon^{-1}(\D) \to \mathfrak{P}$
is a nonlinear Fredholm map of index $-1$.
\end{prop}
\begin{proof}
		It follows from \eqref{eq:Upsilon-psi} that 
$$
D\Upsilon_{(\varphi,\psi)}(\varphi,\psi,x_{12,23}, y_{12,23})
$$
is a linearized projection, so it is bounded and so Fredholm.

We next do a dimension counting. The following identity 
$$
\Index \Pi(\varphi,\psi,x_{12,23}, y_{12,23}) =\Index D \Upsilon_{(\varphi,\psi)}(\varphi,\psi,x_{12,23}, y_{12,23})
$$
is standard.
(See \cite[Appendix 3.6]{mcduff-salamon:quantum} or \cite[Proposition 10.4.7]{oh:book1} for its proof, for example.)

Note that the linear map  $D \Upsilon_{(\varphi,\psi)}$ is one between two finite dimensional vector spaces.
Therefore its index is given by
$$
\dim \, \text{\rm domain} D \Upsilon_{(\varphi,\psi)} - \codim\,\Delta.
$$
Namely,
\beastar
&{}& \Index D \Upsilon_{(\varphi,\psi)}(\varphi,\psi,x_{12,23}, y_{12,23}) \\
    &=&\dim(R_{12}\star R_{23})^2-\codim ((\overline Q_1\star \bar \Delta_{Q_2}^\star \star \overline Q_3)^2\times\D_{\overline Q_1\star Q_3})\\
    &=&2(n_1+2n_2+n_3+3)-\left(2(2n_2+2)+\big(2(n_1+n_3+1)+1\big)\right) = -1.
\eeastar
This finishes the proof.
\end{proof}

\begin{proof}[Wrap-up of the proof of Theorem 7.3]
By the Sard-Smale Theorem, there is a residual set 
$$
\mathfrak{P}_{\text{\rm reg}}\subset \Cont (\overline Q_1\star Q_2)\times \Cont (\overline Q_2\star Q_3)
$$
consisting of regular values of $\Pi$
(with respect to $C^\infty$ topology), such that
$$
 \mathfrak{Int}^{\mathfrak P} _{(\varphi,\psi)}(R_{12},R_{13}) = \Upsilon_{(\varphi,\psi)}^{-1}(\Delta)
 $$
 which
is a smooth manifold of dimension $-1$.
Therefore  $ \mathfrak{Int}^{\mathfrak P} (R_{12},R_{13}) $ must be empty 
for any regular value $(\varphi,\psi)$. Recalling from its definition
 that the set $\mathfrak{Int}^{\mathfrak P} (R_{12},R_{13}) $
 is nothing but the set of pairs 
 $$
 (p,q)\in (R_{12}\star R_{23})^2\setminus\D_{R_{12}\star R_{23}}
 $$
  such that $(\varphi(p),\psi(q))$ is a self-intersection point of the map
  $$
  p_{\bar13}^\star:\varphi(R_{12})\star_{Q_2}\psi(R_{23}) \to \overline Q_1\star Q_3.
  $$
  This implies that the perturbed pair 
  $$
  (R_{12}', R_{23}'): = (\varphi(R_{12}), \psi(R_{23}))
  $$
is indeed  strongly composable.
\end{proof}

\section{Contact product as a bifunctor}
\label{sec:bifunctor}

 In the contact topology, one may regard the Legendrian Fukaya-type category 
 constructed by the first-named author in \cite{oh:LCI-category} as the contact counterpart of 
 the Fukaya category in symplectic topology.
 Like the Lagrangian correspondence in symplectic topology, one may regard Legendrian correspondence as 2-morphisms. We anticipate that the 
Legendrian correspondence and composition of functors play an important role
in the functorial study of contact topology.
 
 More specifically, $R_{12}$ is a functor from $Fuk_{LCI}(Q_1)$ to $Fuk_{LCI}(Q_2)$ 
which means $R_1\circ R_{12}:=R_{12}(R_1)\subset Q_2$ is a Legendrian submanifold.
We will now explain that the product $\star$ is functorial in the sense that 
$\star$ is a bifunctor on the category where morphisms $R_{01}:Q_0\to Q_1$ are 
Legendrian correspondences, i.e., Legendrian submanifolds of $\overline Q_0\star Q_1.$ 
We briefly recall the definition of the notion of bifunctor in the Appendix.

\begin{rem}\label{rem:swap}
    Then $R_{01}\star R_{23}$ is supposed to be the morphism from $\overline Q_0\star Q_2$ to $\overline Q_1\star Q_3$, 
    meaning that a Legendrian in 
    $$
    \overline{(\overline Q_0\star Q_2)}\star(\overline Q_1\star Q_3).
    $$
Note that, according to the definition thereof, we a priori have 
$$
R_{01}\star R_{23} \subset \overline{(\overline Q_0\star Q_1)}\star (\overline{Q_2}\star Q_3).
$$
However we compute
\beastar
\overline{(\overline Q_0\star Q_1)}\star (\overline{Q_2}\star Q_3) & = &
(Q_0\star \overline Q_1) \star (\overline Q_2\star Q_3) \\
& = & (Q_0\star \overline Q_2) \star (\overline Q_1\star Q_3) \\
& \cong & (Q_0\star \overline Q_2) \star (\overline Q_1\star Q_3)\\
& = & \overline{(\overline Q_0\star Q_2)}\star (\overline Q_1\star Q_3)
\eeastar
 where we can swap their domains by the map $\s$ of Proposition \ref{prop:sigma}.
In this regard, one can treat $R_{01}\star R_{23}$ as a Legendrian correspondence between
 $\overline Q_0\star Q_2$ and $\overline Q_1\star Q_3)$, instead of that between
  $(\overline Q_0\star Q_1)$ and $ (\overline Q_2\star Q_3).$
\end{rem}
Recall the category $\mathfrak{Cont}$ defined in Definition \ref{defn:contactcat}.

\begin{prop}\label{prop:bifunctor} The $\star$-product defines a bifunctor
denoted by
 $$
\star:  \mathfrak{Cont} \times  \mathfrak{Cont}  \to \mathfrak{Cont}
$$
that is given by the map  $(Q_1,Q_2)\mapsto Q_1\star Q_2$  in the object level.   
\end{prop}
\begin{proof} Recall that
morphisms of $\mathfrak{Cont}$ from $Q$ to $Q'$ consist of Legendrian submanifolds of $\overline Q\star Q'$. 
Consider Legendrians(morphisms) $R_{ab}\subset \overline Q_a\star Q_b$, $R_{bc}\subset \overline Q_b\star Q_c$, $R_{\alpha\beta}\subset \overline Q_\alpha\star Q_\beta$ and $R_{\beta\gamma}\subset \overline Q_\beta\star Q_\gamma$ such that 
$$
(R_{ab},R_{bc}),\,(R_{\alpha\beta},R_{\beta\gamma})
$$
are composable pairs. We denote by $\bar\star$ the object level map of the $\star$ functor.

Then we have
\bea\label{eq:functorfirst}
\bar\star(R_{bc},R_{\beta\gamma})\circ\bar\star(R_{ab},R_{\alpha\beta})
& = & \s(R_{bc}\star R_{\beta\gamma})\circ\s(R_{ab}\star R_{\alpha\beta})\\
& {} & \subset (Q_a\star \overline Q_\alpha)\star(\overline Q_c\star Q_\gamma)\nonumber
\eea
and
\bea\label{eq:compositionfirst}
\bar\star(R_{bc}\circ R_{ab},R_{\beta\gamma}\circ R_{\alpha\beta})
& = & \s((R_{bc}\circ R_{ab})\star(R_{\beta\gamma}\circ R_{\alpha\beta}))\\
& {} & \subset ( Q_a\star \overline Q_\alpha)\star(\overline Q_c\star  Q_\gamma).\nonumber
\eea
We now show that \eqref{eq:functorfirst} and \eqref{eq:compositionfirst} are \emph{set-theoretically the same}:
The following sequence of equalities holds if everything is composable:
\beastar
\eqref{eq:functorfirst}
& = &  p\left(\s(R_{bc}\star R_{\beta\gamma})_{(\pi_{Q_b}(R_{ab}),\pi_{Q_\beta}(R_{\alpha\beta}))}\star_{(\pi_{Q_b}(R_{bc}),\pi_{Q_\beta}(R_{\beta\gamma}))}\s(R_{ab}\star R_{\alpha\beta})\right)\\
& = & p\left(\s((R_{ab}\star_{Q_b}R_{bc})\star (R_{\alpha\beta}\star_{Q_\beta}R_{\beta\gamma}))\right)\\
& = & \s\left(p((R_{ab}\star_{Q_b}R_{bc})\star (R_{\alpha\beta}\star_{Q_\beta}R_{\beta\gamma}))\right)\\
& = & \s((R_{bc}\circ R_{ab})\star(R_{\beta\gamma}\circ R_{\alpha\beta}))
= \eqref{eq:compositionfirst}.
\eeastar
Here $p=p_{\overline Q_a\star Q_\alpha,\overline Q_c\star Q_\gamma}.$ The third equality is set-theoretic. 
This exactly means that $\bar\star$ respects compositions, implying that it is indeed a bifunctor.
\end{proof}

We proved in Section \ref{section:proofofpentagon} an associativity constraint of $\a_{Q_a,Q_\alpha,Q_1}$ called the \emph{pentagon axiom}. 
Now we  prove another kind of associativity constraint that is essential for the \emph{monoidal structure on $\mathfrak{Cont}$}.

\begin{prop}
    The following general diagram is commutative.
     \begin{center}
\begin{equation}
     \begin{tikzcd}
    (Q_a\star Q_\alpha)\star Q_1\arrow[d,"(R_{ab}\star R_{\alpha\beta})\star R_{12}"] \arrow[r,"\a_{Q_a,Q_\alpha,Q_1}"] & Q_a\star(Q_\alpha\star Q_1)\arrow[d, "R_{ab}\star (R_{\alpha\beta}\star R_{12})"]\\
        (Q_\alpha\star Q_\beta)\star Q_2 \arrow[r,"\a_{Q_\alpha,Q_\beta,Q_2}"]                                                  &   Q_\alpha\star (Q_\beta\star Q_2)
\end{tikzcd}
\end{equation}
   \end{center}
\end{prop}

\begin{proof}
    We will substitute the notion of $\s$ even if it is involved in our formula, to make equations look simpler.\\
    Recall from Lemma \ref{lem:preimage}, one can express the contact product of Legendrian submanifolds as
    $$
    R_i\star R_j=(\pi^\star_i,\pi^\star_j)^{-1}(R_i,R_j)
    $$
    where the map
    $$
    \left(\pi^\star_i,\pi^\star_j\right):Q_i\star Q_j\to Q_i\times Q_j
    $$
    is defined by using contact projections $\pi^\star_i$ and $\pi^\star_j$ of \ref{prop:pi}. For the reader's convenience, we also denote contact projection $\pi^\star$ by $\pi^\star_Q$, with its \emph{codomain} $Q$.\\
    Then we have
    \bea\label{eq:triplepreimage1}
&{}&     (R_{ab}\star R_{\alpha\beta})\star R_{12} \nonumber\\
    & = & \big(\pi^\star_{(Q_a\star Q_b)\star(Q_\alpha\star Q_\beta)},\pi^\star_{Q_1\star Q_2}\big)^{-1}\big((\pi^\star_{Q_a\star Q_b},\pi^\star_{Q_\alpha\star  Q_\beta})^{-1}(R_{ab},R_{\alpha\beta}),R_{12}\big)\nonumber\\
    & = &
    \left((\pi^\star_{Q_a\star Q_b},\pi^\star_{Q_\alpha \star Q_\beta}),\pi^\star_{Q_1\star Q_2}\right)^{-1}(R_{ab},R_{\alpha\beta},R_{12})
    \eea
    and
    \bea\label{eq:triplepreimage2}
    &{}&     R_{ab}\star (R_{\alpha\beta}\star R_{12} )\nonumber\\
    & = & \big(\pi^\star_{Q_a\star Q_b},\pi^\star_{(Q_\a\star Q_\b)\star(Q_1\star Q_2)}\big)^{-1}\big(R_{ab},(\pi^\star_{Q_\a\star Q_\b},\pi^\star_{Q_1\star  Q_2})^{-1}(R_{\a\b},R_{12})\big)\nonumber\\
    & = &
    \left(\pi^\star_{Q_a\star Q_b},(\pi^\star_{Q_\alpha \star Q_\beta},\pi^\star_{Q_1\star Q_2})\right)^{-1}(R_{ab},R_{\alpha\beta},R_{12}).
    \eea
    
    Consider a \emph{triple contact product} and its contact projections
    $$
    \pi^\star_i:P_1\star P_2\star P_3\to P_i.
    $$
    $P_1\star P_2\star P_3$ can be defined by iterated contact products $P_{(12)3}:=(P_1\star P_2)\star P_3$ or $P_{1(23)}:=P_1\star(P_2\star P_3)$ of Section \ref{sect:iterated}.\\
    In Section \ref{sect:iterated} we regard $P_{(12)3}$ and $P_{1(23)}$ as hypersurfaces of $SP_1\times SP_2\times SP_3$. Lemma \ref{lem:samesubset} and Proposition \ref{lem:SSt} tell us that $P_{(12)3}$ and $P_{1(23)}$ represent the same hypersurface \emph{set-theoretically}, and the associator
    $$
    \a_{P_1,P_2,P_3}:P_{(12)3}\to P_{1(23)}
    $$  
    is just a coordinate-change map on the $\R$-components(more precisely, it is a linear map from $\R^2$ to $\R^2$) associated with coordinate functions $\eta$ in Definition \ref{defn:eta}.\\
    Then it is immediate from this description that $P_i$-components, not $\R$-components, are invariant under the associator $\a_{P_1,P_2,P_3}$, namely
    $$
    \a_{P_1,P_2,P_3}\circ \pi^\star_{P_i}
    = \pi^\star_{P_i}\circ\a_{P_1,P_2,P_3}.
    $$
    This enables us to write triple-contact projections $((\pi^\star_{P_1},\pi^\star_{P_2}),\pi^\star_{P_3})$ and $(\pi^\star_{P_1},(\pi^\star_{P_2},\pi^\star_{P_3}))$
    \emph{independently} on iterated contact products:
    $$
    (\pi^\star_{P_1},\pi^\star_{P_2},\pi^\star_{P_3}):P_1\star P_2\star P_3\to P_1\times P_2\times P_3.
    $$
    
    Return to the previous case and apply above observations with
    $$
    P_1=Q_a\star Q_b,\,P_2=Q_\a\star Q_\b\hspace{0.2cm}\text{and}\hspace{0.2cm}P_3=Q_1\star Q_2,
    $$
    then the maps
    \beastar
    \left((\pi^\star_{Q_a\star Q_b},\pi^\star_{Q_\alpha \star Q_\beta}),\pi^\star_{Q_1\star Q_2}\right)
    &:&\left((Q_a\star Q_b)\star (Q_\a\star Q_\b)\right)\star(Q_1\star Q_2)\\
    & {} &\to\left((Q_a\star Q_b)\times (Q_\a\star Q_\b)\right)\times(Q_1\star Q_2)
    \eeastar
    and
    \beastar
    \left(\pi^\star_{Q_a\star Q_b},(\pi^\star_{Q_\alpha \star Q_\beta},\pi^\star_{Q_1\star Q_2})\right)
    &:&(Q_a\star Q_b)\star \left((Q_\a\star Q_\b)\star(Q_1\star Q_2)\right)\\
    & {} &\to(Q_a\star Q_b)\times\left( (Q_\a\star Q_\b)\times(Q_1\star Q_2)\right)
    \eeastar
    in \eqref{eq:triplepreimage1} and \eqref{eq:triplepreimage2} coincides:
    \beastar
    \left(\pi^\star_{Q_a\star Q_b},\pi^\star_{Q_\alpha \star Q_\beta},\pi^\star_{Q_1\star Q_2}\right)
    &:&(Q_a\star Q_b)\star (Q_\a\star Q_\b)\star(Q_1\star Q_2)\\
    & {} &\to(Q_a\star Q_b)\times (Q_\a\star Q_\b)\times(Q_1\star Q_2).
    \eeastar
    Therefore again \eqref{eq:triplepreimage1} and \eqref{eq:triplepreimage2} show $(R_{ab}\star R_{\alpha\beta})\star R_{12}$ and $R_{ab}\star (R_{\alpha\beta}\star R_{12})$ represent the same subset in $P_1\star P_2\star P_2$, and equivariant under $\a_{P_1,P_2,P_3}$ if we illustrate for them to lie in iterated contact products. This exactly means the commutativity of diagram.
\end{proof}

\begin{thm}[Monoidal category $\mathfrak{Cont}$]
    $\left(\mathfrak{Cont},\star,\a_{(\cdot,\cdot,\cdot)}\right)$ is a (non-unital) monoidal category.
\end{thm}    

\appendix

    \section{Bifunctor}  
        In this appendix, we recall basic definitions of bifunctors from the category theory.
        First we recall the definition of category and functor.
        \begin{defn}[Category]
            A category $\CC$ consists of two sets(or more generally classes) $\emph{Obj}$(called objects) and $\emph{Mor}$(called morphisms) with the structure map(called composition)
            $$
            \circ:\text{\emph{Mor}}(X_0,X_1)\times\text{\emph{Mor}}(X_1,X_2)\to \text{\emph{Mor}}(X_0,X_2)
            $$
            for any $X_i\in$\emph{Obj}, satisfying a commutativity criterion.
        \end{defn}
\begin{defn}[Functor]
    Let $\CC$ and $\CD$ be categories. A map
    $$\CF:\CC\to \CD$$
    sending objects to objects and morphisms to morphisms is called a \emph{functor}, if it respects structure maps(compositions) of its domain and codomain. More precisely,
    $$\CF(f\circ g)=\CF(f)\circ \CF(g)$$
    for any $f,g\in\text{\emph{Mor}}_{\CC}.$
\end{defn}
Now we have the following simple way of defining the bifunctor using so called 'product category'.
\begin{defn}[Product Category]
    Let $\CC$ be a category. The \emph{Product Category} $\CC\times\CC$ has its objects set
    $$\text{\emph{Obj}}_{\CC\times\CC}:=\text{\emph{Obj}}_{\CC}\times\text{\emph{Obj}}_{\CC}$$
    and morphism set
    $$\text{\emph{Mor}}_{\CC\times\CC}:=\text{\emph{Mor}}_{\CC}\times \text{\emph{Mor}}_{\CC}.$$
    The composition is given by the obvious relation:
    $$(f,g)\circ_{\CC\times\CC}(h,i)=(f\circ_{\CC} h,g\circ_{\CC} i)$$
    for $f,g,h,i\in\text{\emph{Mor}}_{\CC}.$
\end{defn}
\begin{defn}[Bifunctor]
   Let $\CC,\CD$ and $\CE$ be categories. A functor $\CF$
$$\CF:\CC\times \CD\to \CE$$
from the product category $\CC\times \CD$ to $\CE$ is called a bifunctor.
\end{defn}


\def\cprime{$'$}
\providecommand{\bysame}{\leavevmode\hbox to3em{\hrulefill}\thinspace}
\providecommand{\MR}{\relax\ifhmode\unskip\space\fi MR }
\providecommand{\MRhref}[2]{%
  \href{http://www.ams.org/mathscinet-getitem?mr=#1}{#2}
}
\providecommand{\href}[2]{#2}

\end{document}